\documentclass[10pt]{amsart}
\usepackage[margin=1in,letterpaper,portrait]{geometry}
\usepackage{amssymb,amsmath,xypic,epsfig,verbatim,shuffle} 
\usepackage{tikz}
\usetikzlibrary{shapes.geometric, intersections, calc}
\usepackage{xspace}
\newcommand{\excise}[1]{}
\usepackage{array}

\newtheorem{Theorem}{Theorem}[section]

\newtheorem{Proposition}[Theorem]{Proposition}
\newtheorem{Lemma}[Theorem]{Lemma}

\newtheorem{Corollary}[Theorem]{Corollary}

\newtheorem{Conjecture}[Theorem]{Conjecture}

\theoremstyle{definition}
\newtheorem{Definition}[Theorem]{Definition}
\newtheorem{Example}[Theorem]{Example}
\newtheorem*{Definition-nonum}{Definition}
\newtheorem*{Example-nonum}{Example}
\theoremstyle{remark}
\newtheorem{Remark}[Theorem]{Remark}

\newcommand\ZZ{{\mathbb{Z}}}
\newcommand\NN{{\mathbb{N}}}
\newcommand\QQ{{\mathbb{Q}}}
\newcommand\RR{{\mathbb{R}}}

\newcommand\EE{{\mathbb{E}}}

\newcommand\Omegaunif{\Omega^{\operatorname{unif}}}
\newcommand\Omegachain{\Omega^{\operatorname{chain}}}
\newcommand\Xdd{X}
\newcommand\Ydd{Y}
\newcommand\Prob{{\mathbf{Prob}}}
\newcommand\CDE{{CDE}\xspace}
\newcommand\mCDE{{mCDE}\xspace}
\newcommand\tCDE{{tCDE}\xspace}

\newcommand{\shifted}{{\operatorname{shifted}}}
\newcommand{\Des}{\operatorname{Des}}

\newcommand\noninv{{\operatorname{ninv}}}
\newcommand\rank{{\operatorname{rank}}}
\newcommand{\Red}{\operatorname{Red}}
\newcommand{\MaxChains}{{\mathcal{M}}}
\newcommand\opp{{*}}

\newcommand\sss{{\mathbf{s}}}
\newcommand\ttt{{\mathbf{t}}}
\newcommand\xxx{{\mathbf{x}}}

\newcommand\symm{{\mathfrak{S}}}

\newcommand\HHH{{\mathcal{H}}}
\newcommand\LLL{{\mathcal{L}}}

\newcommand{\lambdanab}[3]{ \delta_{#1}({#3}^{#2}) }
\newcommand{\CST}[2]{{\tt CST}_{\leq{#2}}({#1})}

\newcommand{\bm}[1]{\mbox{\boldmath $#1$}}

\newcommand{\cellsize}{13}
\newlength{\cellsz} 
\newsavebox{\cell}
\sbox{\cell}{\begin{picture}(\cellsize,\cellsize)
\put(0,0){\line(1,0){\cellsize}}
\put(0,0){\line(0,1){\cellsize}}
\put(\cellsize,0){\line(0,1){\cellsize}}
\put(0,\cellsize){\line(1,0){\cellsize}}
\end{picture}}
\newcommand\cellify[1]{\def\thearg{#1}\def\nothing{}%
\ifx\thearg\nothing
\vrule width0pt height\cellsz depth0pt\else
\hbox to 0pt{\usebox{\cell} \hss}\fi%
\vbox to \cellsz{
\vss
\hbox to \cellsz{\hss$#1$\hss}
\vss}}
\newcommand\tableau[1]{\vtop{\let\\\cr
\baselineskip -16000pt \lineskiplimit 16000pt \lineskip 0pt
\ialign{&\cellify{##}\cr#1\crcr}}}
\newcommand{\kellsize}{20}
\newlength{\kellsz} 
\newsavebox{\kell}
\sbox{\kell}{\begin{picture}(\kellsize,\kellsize)
\put(0,0){\line(1,0){\kellsize}}
\put(0,0){\line(0,1){\kellsize}}
\put(\kellsize,0){\line(0,1){\kellsize}}
\put(0,\kellsize){\line(1,0){\kellsize}}
\end{picture}}
\newcommand\kellify[1]{\def\thearg{#1}\def\nothing{}%
\ifx\thearg\nothing
\vrule width0pt height\kellsz depth0pt\else
\hbox to 0pt{\usebox{\kell} \hss}\fi%
\vbox to \kellsz{
\vss
\hbox to \kellsz{\hss$#1$\hss}
\vss}}
\newcommand\ktableau[1]{\vtop{\let\\\cr
\baselineskip -16000pt \lineskiplimit 16000pt \lineskip 0pt
\ialign{&\kellify{##}\cr#1\crcr}}}

\newcommand{\tellsize}{6}
\newlength{\tellsz} 
\newsavebox{\tell}
\sbox{\tell}{\begin{picture}(\tellsize,\tellsize)
\put(0,0){\line(1,0){\tellsize}}
\put(0,0){\line(0,1){\tellsize}}
\put(\tellsize,0){\line(0,1){\tellsize}}
\put(0,\tellsize){\line(1,0){\tellsize}}
\end{picture}}
\newcommand\tellify[1]{\def\thearg{#1}\def\nothing{}%
\ifx\thearg\nothing
\vrule width0pt height\tellsz depth0pt\else
\hbox to 0pt{\usebox{\tell} \hss}\fi%
\vbox to \tellsz{
\vss
\hbox to \tellsz{\hss$#1$\hss}
\vss}}
\newcommand\ttableau[1]{\vtop{\let\\\cr
\baselineskip -16000pt \lineskiplimit 16000pt \lineskip 0pt
\ialign{&\tellify{##}\cr#1\crcr}}}

\begin{document} 
\title[Poset edge densities, reduced words, tableaux]
{Poset edge densities,
nearly reduced words, 
and barely set-valued tableaux}

\author{Victor Reiner}
\address{School of Mathematics, University of Minnesota, Minneapolis, MN 55455, USA}
\author[Bridget Eileen Tenner]{Bridget Eileen Tenner}
\address{Department of Mathematical Sciences, DePaul University, Chicago, IL 60614, USA}
\author{Alexander Yong}
\address{Department of Mathematics, University Illinois at Urbana-Champaign, Urbana, IL 61801, USA}

\begin{abstract}
In certain finite posets, the expected down-degree of
their elements is the same whether computed with respect to either
the uniform distribution or the
distribution weighting an element by the number of maximal chains 
passing through it.
We show that this coincidence of expectations holds for Cartesian products of chains, connected
minuscule posets, weak Bruhat orders on finite Coxeter groups, 
certain lower intervals in Young's lattice, and certain lower intervals
in the weak Bruhat order below dominant permutations.
Our tools involve formulas for counting nearly reduced factorizations in $0$-Hecke algebras; that is, factorizations that are one letter longer than the Coxeter group length.
\end{abstract}

\maketitle

\section{Introduction}
\label{section:intro}

The \emph{edge density} of a finite poset $P$ 
is the ratio of the number of its covering relations $q \lessdot p$ 
to its cardinality $\#P$. One can also interpret this ratio as the {expectation} $\EE(\Xdd)$ of a 
random variable $\Xdd$ on $P$, counting the elements covered by $p \in P$.
That is, the random variable $\Xdd$ computes the down-degree of $p$ in the Hasse diagram of $P$,
with respect to the uniform distribution.

If, instead, one changes this distribution by assigning to each $p \in P$
a probability proportional to the number of {maximal chains through $p$} in $P$, then one can define a random variable $\Ydd$ on $P$ whose value is again the down-degree of $p$ in the Hasse diagram, but now weighted by that probability. 

Given the different distributions in play, one would generally 
not expect the expectations for $\Xdd$ and $\Ydd$ to be equal. However, we prove that, in a variety of interesting settings, one does indeed find equality, and we conjecture that equality holds in several additional settings, as well.

\begin{Definition-nonum}
A finite poset $P$ has \emph{coincidental down-degree expectations (\CDE)} if 
$
\EE(\Xdd) =
\EE(\Ydd)$. 
\end{Definition-nonum}

We may also refer to $P$ as \emph{being} \CDE.
This terminology will be made more precise in Definition~\ref{defn:cde}. To motivate our study, consider the following examples of \CDE posets.

\begin{itemize}
\item
\emph{Disjoint unions of chains} are \CDE because
the two probability distributions are the same in this setting.

\item
\emph{Cartesian products of chains} are \CDE because Proposition~\ref{products-preserve-EEE}
will show that \CDE is preserved under Cartesian products of
graded posets.  

\item
\emph{Weak Bruhat order} on a finite Coxeter group is \CDE. In fact, any weak order on the chambers of a (central, essential) \emph{simplicial hyperplane arrangement} in $\mathbb{R}^r$ (or, more generally, the topes of an \emph{oriented matroid} of rank $r$) is CDE, as will be shown in Corollary~\ref{simplicial-arrangement-corollary}.

\item
\emph{Tamari lattices} on polygon triangulations are \CDE,
as will be shown in Corollary~\ref{Tamari-corollary}.

\item
\emph{Connected minuscule posets} are \CDE, as will be shown in
Theorem~\ref{minuscule-theorem}.  Also the distributive lattices
$J(P)$ associated to arbitrary
minuscule posets $P$ are \CDE, as will be shown in 
Theorem~\ref{Gaussian-poset-CDE-theorem}. 

\item 
Our main result, Theorem~\ref{main-theorem}, 
exhibits a rich class of lower intervals in \emph{Young's lattice}
and in weak Bruhat orders on permutations, all of which are \CDE. (In fact, this paper 
grew from an attempt to understand 
Corollary~\ref{motivating-corollary} of Theorem~\ref{main-theorem}
in two different ways.)
\end{itemize}

Before stating our main result, we recall a few definitions.  
\emph{Young's lattice} is the partial order on 
integer partitions $\lambda$ according to containment of
their \emph{Ferrers diagrams} $\mu \subset \lambda$.
The \emph{(right) weak Bruhat order} on permutations 
in the symmetric group $\symm_n$ is the transitive
closure of the relation $u \lessdot w$ if $w=us$
for some adjacent transposition $s = \sigma_i=(i,i+1)$ with $u(i) < u(i+1)$.
A permutation $w=w(1)\cdots w(n) \in \symm_n$ is 
\emph{vexillary} if it is $2143$-avoiding; that is, if
there are no quadruples $i_1 < i_2 < i_3 < i_4$
with $w(i_2) < w(i_1) < w(i_4)< w(i_3)$.  Such a vexillary permutation has \emph{shape $\lambda$} if
$\lambda$ is the weakly decreasing rearrangement of
its \emph{Lehmer code} $c(w):=(c_1(w),c_2(w),\ldots)$, where
$c_i(w):=\#\{j > i: w(i) \geq w(j)\}$.
Within the class of vexillary permutations we will consider three subclasses.
\begin{itemize}
\item
A permutation $w$ is \emph{dominant} if it is $132$-avoiding; that is, if there are no triples $i_1 < i_2 < i_3$
with $w(i_1) < w(i_3) < w(i_2)$.
If we regard the symmetric group $\symm_n$ as a subset of $\symm_{n+1}$
via the embedding $w \mapsto w'$, where $w'(i) = w(i)$ for $1 \le i \le n$ and $w(n+1) = n+1$,
then there is a unique dominant 
permutation in $\bigcup_{n \geq 0} \symm_n$ of shape $\lambda$, characterized 
by $c(w)=\lambda$ (without rearrangement).
\item 
A permutation $w$ is \emph{Grassmannian} if it has at most
one \emph{descent}; that is, if $w(i) > w(i+1)$ for at most one value of $i$.
\item
A permutation $w$ is \emph{inverse Grassmannian} if
$w^{-1}$ is Grassmannian; that is, if $w^{-1}(i) > w^{-1}(i+1)$
for at most one value of $i$.
\end{itemize} 

We will also want to consider a family of partitions generalizing both
\begin{itemize}
\item $a \times b$ \emph{rectangles} $b^a=(b,b,\ldots,b)$, and
\item \emph{staircases} $\delta_d=(d-1,d-2,\ldots,2,1)$. 
\end{itemize}
As such, we study the \emph{rectangular staircase partitions}
$\lambdanab{d}{a}{b}$, 
whose Ferrers diagrams are staircases $\delta_d$ 
in which each square is replaced by an $a \times b$ block. One such partition appears in Figure~\ref{figure:rectangular staircase example}.

\begin{figure}[htbp]
$$
\ttableau{
{}&{}&{}&{}&{}&{}&{}&{}&{}&{}&{}&{}&\\
{}&{}&{}&{}&{}&{}&{}&{}&{}&{}&{}&{}\\
{}&{}&{}&{}&{}&{}&{}&{}&\\
{}&{}&{}&{}&{}&{}&{}&{}\\
{}&{}&{}&{}&\\
{}&{}&{}&{}}
$$
\caption{The rectangular staircase partition $\lambdanab{4}{2}{4}=(12,12,8,8,4,4)$. The dominant permutation with Lehmer code $c(w) = \lambdanab{4}{2}{4}$ is
$w = 13 \ 14 \ 9 \ 10 \ 5 \ 6 \ 1 \ 2 \ 3 \ 4 \ 7 \ 8 \ 11 \ 12$ in  $\symm_{14}$.}\label{figure:rectangular staircase example}
\end{figure}

\begin{Theorem}
\label{main-theorem}
Let $\lambda$ be a partition, and $w$ a vexillary permutation 
of shape $\lambda$.

\begin{enumerate}\renewcommand{\labelenumi}{(\alph{enumi})}
\item
The lower intervals $[\varnothing,\lambda]$
in Young's lattice and $[e,w]$ in weak Bruhat order
have the same $\Ydd$ expectations; that is,
$$
\EE(\Ydd_{[\varnothing,\lambda]}) = \EE(\Ydd_{[e,w]}).
$$
\item
If $w$ is either Grassmannian or inverse Grassmannian, then these intervals
also have the same $\Xdd$ expectations; that is, 
$$
\EE(\Xdd_{[\varnothing,\lambda]}) = \EE(\Xdd_{[e,w]}).
$$
\item
If $\lambda=\lambdanab{d}{a}{b}$ 
is a rectangular 
staircase and $w$ is either 
\begin{itemize}
\item dominant, 
\item Grassmannian, or 
\item inverse Grassmannian,
\end{itemize}
then the posets $[\varnothing, \lambda]$ and $[e,w]$, and 
their duals $[\varnothing, \lambda]^*$ and $[e,w]^*$
are all \CDE. Moreover,
$$
\EE(\Xdd) 
=\EE(\Ydd)
= \frac{(d-1)ab}{a+b} 
$$
in each case.
\end{enumerate}
\end{Theorem}

Note that one always has 
$\EE(\Xdd_{P})=\EE(\Xdd_{P^*})$ because
the Hasse diagrams of $P$ and $P^*$ have the same edge densities,
but $\EE(\Ydd_{P})$ and $\EE(\Ydd_{P^*})$ need not be equal. Indeed, as will be described in Example~\ref{duality-does-not-preserve-CDE}, there are \CDE posets $P$ whose
dual posets $P^\opp$ are not \CDE.

\begin{Conjecture}
When $\lambda=\lambdanab{d}{a}{b}$, 
the conclusion in Theorem~\ref{main-theorem}(c) holds
for all vexillary $w$ of shape $\lambda$.
\end{Conjecture}

There is a close connection relating the
      instance of CDE given by Theorem~\ref{main-theorem}(c) to recent work of 
Chan, Mart\'{i}n, Pflueger, and Teixidor i Bigas \cite{Chan1} and of
      Chan,
      Haddadan, Hopkins, and Moci \cite{Chan2}. The result \cite[Corollary~2.15]{Chan1} (recapitulated as \cite[Theorem 1.1]{Chan2}) calculates
      the expected ``jaggedness'' of a lattice path in an $a\times b$
      grid under a certain probability distribution on paths. This is
      the central combinatorial fact used in \cite{Chan1} to 
reprove a formula of
      Eisenbud-Harris and of Pirola for the genera of Brill-Noether
      curves. Theorem~1.2 of \cite{Chan2} is a
      generalization of this jaggedness theorem to
      lattice paths in a general connected skew shape with respect to
      any ``toggle symmetric'' distribution. As is detailed further in
      Remark~\ref{remark:Chan}, Theorem~\ref{main-theorem}(c) 
provides a different proof of
      \cite[Corollary 2.15]{Chan1}, whereas \cite[Theorem 1.2]{Chan2} may
      be used to give a different proof of the assertion on 
      $[\varnothing, \lambda]$ and $[\varnothing,\lambda]^*$ 
      for $\lambda=\lambdanab{d}{a}{b}$
      in Theorem~\ref{main-theorem}(c).

After covering the groundwork for \CDE posets in Section~\ref{section:EDDE}, most of the paper is aimed toward proving
the assertions of Theorem~\ref{main-theorem}.  
We build up general techniques to compute 
$\EE(\Xdd)$ and $\EE(\Ydd)$ for 
$[\varnothing,\lambda]$ in Young's lattice 
using Young tableaux and set-valued tableaux (Section~\ref{tableaux-section}),
and for $[e,w]$ in Coxeter groups 
involving reduced words and $0$-Hecke words (Section~\ref{Coxeter-section}).
Tableaux reenter the discussion when we specialize to the symmetric
group in Section~\ref{type-A-section}, for reasons that we highlight now.

For a permutation $w$, maximal chains in the 
lower interval $[e,w]$ of weak Bruhat
order correspond to 
\emph{reduced words} for $w$. In particular, they describe factorizations
$w=\sigma_{i_1} \sigma_{i_2} \cdots \sigma_{i_\ell}$ into adjacent transpositions $\sigma_i$ 
having the minimum possible length $\ell$, called $\ell(w)$.
Stanley \cite{Stanley84} proved that the number of reduced words
for any vexillary permutation $w$ of shape $\lambda$
is $f^\lambda$, the number of \emph{standard Young tableaux} of shape $\lambda$,
which has a simple product expression known as the
Frame-Robinson-Thrall \emph{hook-length formula} 
\cite[Corollary 7.21.6]{StanleyEC2}.
More generally, one can consider factorizations 
$T_w=T_{i_1} T_{i_2} \cdots T_{i_L}$ in
the \emph{$0$-Hecke monoid} for permutations, with generators
$T_1,\ldots,T_{n-1}$ satisfying the usual braid relations 
together with the quadratic relation $T_i^2=T_i$.
The $0$-Hecke factorizations for $w$ having the minimum
length $L=\ell(w)$ correspond
to reduced words as before. 
 Using results from the
theory of \emph{Schubert} and \emph{Grothendieck polynomials} (see Section~\ref{vexillary-section}), one can
generalize Stanley's result to assert that the number of $0$-Hecke words 
having length $L$ for a vexillary permutation $w$ of shape $\lambda$
is the number of \emph{standard set-valued tableaux}
of shape $\lambda$ having entries $1,2,\ldots,L$, each appearing
exactly once.  Here a set-valued tableau has a subset of entries
filling each square, but entries still increase from left-to-right in a row, and from top-to-bottom in a column.
When $L=\ell(w)+1$, we call the corresponding
$0$-Hecke words \emph{nearly reduced} and the set-valued tableaux
\emph{barely set-valued}.  This terminology will be made precise in Definition~\ref{defn:set-valued tableaux}.

One no longer has a hook-length-style product formula 
for counting set-valued tableaux of any shape $\lambda$.
However, we derive a general recurrence for counting these objects
(Corollary~\ref{barely-set-valued-tableau-recurrence}),
and use this to show that for dominant $w$ whose
shape is a rectangular staircase $\lambdanab{d}{a}{b}$, 
one has this rephrasing of 
$\EE(\Ydd_{[e,w]})=\frac{(d-1)ab}{a+b}$.

\begin{Corollary}
\label{motivating-corollary}
Let $w$ be a dominant permutation of rectangular staircase shape
$\lambda=\lambdanab{d}{a}{b}$.  
Then the number of barely set-valued tableaux of shape $\lambda$ (equivalently, the number of nearly reduced words for $w$) is 
$$
\left(|\lambda|+1\right) \frac{(d-1)ab}{a+b} f^{\lambda}.
$$
\end{Corollary}

\begin{Example}\label{ex:long element in s3}
Taking $d=3$ 
and $a=b=1$, one has $\lambda=\lambdanab{3}{1}{1}=(2,1)$ and
$w=321$, with two reduced words: $(\sigma_1,\sigma_2,\sigma_1)$ and 
$(\sigma_2,\sigma_1,\sigma_2)$. 
Correspondingly, there are $f^{\lambda}=2$ standard Young tableaux of
shape $\lambda$:
$$
\tableau{1&2\\3}
\hspace{.5in} \text{and} \hspace{.5in}
\tableau{1&3\\2} \ .
$$
Meanwhile, there are eight $0$-Hecke words of length $4$ for $w$:
$$
\begin{aligned}
&\left\{ (T_1,T_1,T_2,T_1), (T_1,T_2,T_1,T_1), (T_1,T_2,T_1,T_2), (T_1,T_2,T_2,T_1),\right.\\ 
&\quad 
\left. (T_2,T_1,T_1,T_2), (T_2,T_1,T_2,T_1),(T_2,T_1,T_2,T_2),(T_2,T_2,T_1,T_2) \right\}.
\end{aligned}
$$
These correspond to the eight barely set-valued tableaux of
shape $\lambda=(2,1)$:
$$
\tableau{12&3\\4} \ , \ \
\tableau{12&4\\3} \ , \ \
\tableau{1&23\\4} \ , \ \
\tableau{1&4\\23} \ , \ \
\tableau{1&24\\3} \ , \ \
\tableau{1&3\\24} \ , \ \
\tableau{1&2\\34} \ , \text{ and }
\tableau{1&34\\2} \ .
$$
This agrees with Corollary~\ref{motivating-corollary}, which would have
predicted this number to be
$$
\left(|\lambda|+1\right) \frac{(d-1)ab}{a+b} f^{\lambda}
=(3+1)\frac{(3-1)\cdot1\cdot1}{1+1} \cdot 2
=8.
$$
\end{Example}

A remarkable formula of Macdonald \cite[page 91]{Macdonald91} states that 
\[\sum_{\sigma_{i_1},\sigma_{i_2},\ldots, \sigma_{i_N}} i_1 i_2\cdots i_{N}=N!,\]
where the sum is over reduced words of $w_0:=n(n-1) \cdots 21$, the unique longest length permutation of ${\mathfrak S}_n$ (of length $N={n\choose 2}$).
A generalization, due to Fomin and Kirillov \cite{Fomin.Kirillov97} (recapitulated here as
Theorem~\ref{Fomin-Kirillov-theorem}) connects this formula to enumerations of certain 
\emph{plane partitions}. Section~\ref{FK-conjecture-section}
contains a conjecture (Conjecture~\ref{conj:mainone})
which is inspired both by Corollary~\ref{motivating-corollary} 
and by the Macdonald and Fomin-Kirillov formulas. This conjecture concerns nearly reduced words of dominant permutations of rectangular staircase shape $\delta_d(b^a)$.  We show that this conjecture is equivalent to another conjecture, phrased in terms of barely set-valued tableaux. 
Thereby, we verify these conjectures for the case $d=2$ (i.e., rectangles). The proof the equivalence is based on another extension of the 
Fomin-Kirillov formula (Theorem~\ref{thm:anFKformula}), valid for any vexillary permutation. This latter extension is proved in Section~\ref{FKformula-proof-section}, by combining algebraic ideas of Fomin-Stanley
\cite{Fomin.Stanley} and Fomin-Kirillov \cite{Fomin.Kirillov:YB, Fomin.Kirillov, Fomin.Kirillov97}.


We wish to highlight here
one byproduct of our analysis.
The calculation of $\EE(\Xdd_{[\varnothing,\lambda]})$
for $\lambda$ a rectangular staircase 
(Proposition~\ref{X-expectation-for-Young})
uses the $q=1$ specialization of an easy
recurrence for the \emph{rank-generating function}
$R(\lambda,q):=\sum_{\mu \subset \lambda} q^{|\mu|}$
of the interval $[\varnothing,\lambda]$ in Young's lattice.
This recurrence encompasses 
\begin{itemize}
\item the \emph{$q$-Pascal recurrence} for \emph{$q$-binomials}
\cite[Equation~(1.67)]{StanleyEC1}
when $\lambda$ is a rectangle, and
\item the recurrence for the 
\emph{Carlitz-Riordan $q$-Catalan polynomial} which counts all
Dyck paths by their enclosed {area} \cite[Proposition 1.6.1]{Haglund08} when $\lambda$ is a staircase,
\end{itemize}
but we were unable to find it in the literature.

\begin{Proposition}
\label{q-Pascal-Catalan-generalization}
For any partition $\lambda$,
$$
R(\lambda,q)
 =\sum_{x=(i,j)} q^{i(j-1)} 
     \cdot R(\lambda_{(x)},q) 
     \cdot R(\lambda^{(x)},q),
$$
where $x$ runs over all outside corner cells of $\lambda$, lying in
row $i$ and column $j$, and where
$$
\lambda_{(x)}=(\lambda_{i+1},\lambda_{i+2},\ldots) \text{ \ \ and \ \ } \lambda^{(x)}=(\lambda_1 - j,\lambda_2 - j,\ldots,\lambda_{i-1} - j)
$$
are the subshapes of $\lambda$ in the rows strictly below
$x$ and the columns strictly to the right of $x$, respectively.
\end{Proposition}
\begin{proof}
For $j > 1$, each outside corner cell $x=(i,j)$ has a cell $y=(i,j-1)$ 
inside $\lambda$ and directly to its left.  For example, those neighboring cells are labeled 
$\{y_1,y_2,y_3,y_4\}$ in the shape below.
$$
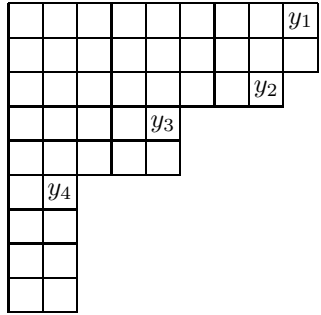
\tableau{
{}&{}&{}&{}&{}&{}&{}&{}&{y_1}\\
{}&{}&{}&{}&{}&{}&{}&{}&{ }\\
{}&{}&{}&{}&{}&{}&{}&{y_2}\\
{}&{}&{}&{}&{y_3}\\
{}&{}&{}&{}&{}\\
{}&{y_4}\\
{}&{}\\
{}&{}\\
{}&{}}
$$
The recurrence in the theorem comes from classifying a shape $\mu \subset \lambda$ according
to which, if any, is the northeasternmost cell $y_i$ contained in $\mu$.
For example, with $\lambda$ as above, consider the cell $y_3 =(i,j-1)=(4,5)$, immediately to the
left of the outside corner $x_3=(i,j)=(4,6)$.  A partition 
$\mu\subset \lambda$ for which $y_3  \in \mu$ but $y_1,y_2 \not\in\mu$
must contain 
\begin{itemize}
\item {all} of the $i(j-1)=4\cdot 5=20$
cells weakly northwest of $y_3$, labeled $\bullet$ in the figure below, and
\item {none} of the cells at the east ends of rows $1,2,\ldots,i-1 = 4$,
labeled $\times$ in the figure below.
\end{itemize}
$$\tableau{
{\bullet}&{\bullet}&{\bullet}&{\bullet}&{\bullet}&{}&{}&{}&{\times}\\
{\bullet}&{\bullet}&{\bullet}&{\bullet}&{\bullet}&{}&{}&{}&{\times}\\
{\bullet}&{\bullet}&{\bullet}&{\bullet}&{\bullet}&{}&{}&{\times}\\
{\bullet}&{\bullet}&{\bullet}&{\bullet}&{y_3}\\
{}&{}&{}&{}&{}\\
{}&{}\\
{}&{}\\
{}&{}\\
{}&{}}$$
Thus, this $\mu$ is determined by its restrictions to the shapes in the unmarked cells of the figure above. The southwesternmost of these constitute exactly $\lambda_{(x_3)}$, and the northeasternmost form a copy of $\lambda^{(x_3)}$.
\end{proof}

\section{An overview of the \CDE property}
\label{section:EDDE}

We now make precise the central theme of this paper, broached previously in Section~\ref{section:intro}.

\begin{Definition}\label{defn:cde}
Given a finite poset $(P,\leq)$, define two different probability spaces on the underlying set $P$.
\begin{itemize}
\item Let $\Omegaunif_P$ be the uniform distribution, assigning
$\Prob(p)=1/\#P$  for each $p \in P$.
\item Let $\Omegachain_P$ assign $\Prob(p)$ to be proportional to the number of maximal chains $c$ in $P$ containing $p$.  That is, if $\MaxChains(P)$ is the set of maximal chains in $P$, then
$$\Prob(p) = \frac{\#\{c \in \MaxChains(P): p \in c \}}
                        {\#\{(c,q) \in \MaxChains(P) \times P: q \in c \}}
                        = \frac{\#\{c \in \MaxChains(P): p \in c \}}{\sum\limits_{c \in \MaxChains(P)} \#c}.
$$
\end{itemize}
Define random variables $\Xdd:=\Xdd_P$ on $\Omegaunif_P$ and
 $\Ydd:=\Ydd_P$ on $\Omegachain_P$ via the same formula:  
$$
\Xdd(p) = \Ydd(p)=\#\{ q \in P: q \lessdot p \}.
$$
A poset $P$ has 
\emph{coincidental down-degree expectations} (equivalently, $P$ \emph{is} \CDE) 
if $\EE(\Xdd) = \EE(\Ydd)$.
\end{Definition}

Whenever the poset $P$ is graded, there is a natural way to
interpolate between $X$ and $Y$, pointed out to the authors by
S.~Hopkins, and suggested by the work in \cite{Chan2}. We describe this interpolation using the language of multichains.

\begin{Definition}
Given a finite poset $P$ and a positive integer $m$,
define a probability space $\Omega^{{\scriptscriptstyle (}m{\scriptscriptstyle)}}_P$ on the
underlying set $P$, with $\Prob(p)$ proportional to the
number of $m$-element \emph{multichains} 
$p_1 \leq p_2 \leq \cdots \leq p_m$ in $P$ that pass through $p$.
On this probability space $\Omega^{{\scriptscriptstyle (}m{\scriptscriptstyle)}}_P$, define the
random variable $X^{{\scriptscriptstyle (}m{\scriptscriptstyle)}} := X^{{\scriptscriptstyle (}m{\scriptscriptstyle)}}_P$ as before, where $X^{{\scriptscriptstyle (}m{\scriptscriptstyle)}}(p)=\#\{q \in P: q \lessdot p\}$ records
the down-degree of the element $p$.  
\end{Definition}

Two extreme cases are of particular interest.
When $m=1$, we have 
$(\Omega^{{\scriptscriptstyle (}1{\scriptscriptstyle)}}_P,X^{{\scriptscriptstyle (}1{\scriptscriptstyle)}})=(\Omegaunif_P,\Xdd)$.  On the other hand,
if $P$ is {graded of rank $r$}, that is, if 
all of its (inclusion-)maximal chains have exactly $r+1$ elements,
then it is not hard to see (cf. \cite[Proposition 2.9]{Chan2})
that the pair $(\Omega^{{\scriptscriptstyle (}m{\scriptscriptstyle)}}_P,X^{{\scriptscriptstyle (}m{\scriptscriptstyle)}})$ approaches
$(\Omegachain_P,\Ydd)$ in the limit as $m \rightarrow \infty$. Indeed,
one can easily check (cf.~\cite[\S 3.12]{StanleyEC1}) 
that in this graded setting, 
the number of $m$-element multichains passing through $p$ is a polynomial
in $m$ of degree $r$, and that the leading coefficient of this polynomial is $(1/r!) \cdot \#\{c \in \MaxChains(P): p \in c\}$.

\begin{Definition}
\label{mCDE-definition}
A finite poset $P$ is \emph{multichain-\CDE} (written {\mCDE}) 
if $\EE(X^{{\scriptscriptstyle (}m{\scriptscriptstyle)}})$ is constant and independent of $m$, for $m \geq 1$.
\end{Definition}

In particular, observe that if $P$ is both graded and \mCDE, then $P$ is also \CDE; indeed, in that case we would have
$\EE (\Ydd) =\lim_{m \rightarrow \infty} \EE (X^{{\scriptscriptstyle (}m{\scriptscriptstyle)}}) =\EE (X^{{\scriptscriptstyle (}1{\scriptscriptstyle)}}) = \EE (\Xdd)$.

It will be helpful to know that for the distributive
lattice $J(P)$ of order ideals $I$ in
$P$, the probability distribution $\Omega^{{\scriptscriptstyle (}m{\scriptscriptstyle)}}_{J(P)}$ is \emph{toggle-symmetric}, a concept defined in \cite{Chan2}
and which we explain now.  

\begin{Definition}{\rm (\cite[Definition 2.2]{Chan2})}
\label{toggle-symmetry-definition}
Let $P$ be a finite poset, and let $I$ denote an order ideal in $P$. 
For a subset $A \subseteq P$, let $\max(A)$ (respectively, $\min(A)$) 
denote the subset of $P$-maximal (respectively, $P$-minimal elements) in $A$. 
A probability distribution 
on the finite distributive lattice $J(P)$
is {\it toggle-symmetric} if, for every $p \in P$,
the distribution assigns the same probability 
to the event that $\max(I)$ contains $p$ as it assigns 
to the event that $\min(P \setminus I)$ contains $p$.
\end{Definition}

\begin{Proposition}
\label{multichain-distribution-is-toggle-symmetric}
For finite posets $P$, the distribution 
$\Omega^{{\scriptscriptstyle (}m{\scriptscriptstyle)}}_{J(P)}$ on $J(P)$ is toggle-symmetric.
\end{Proposition}
\begin{proof}
This is equivalent to showing that, for every $p \in P$,
the following two sets have the same cardinality:
\begin{itemize}
\item
the set of all 
pairs $(I,c)$ in which $I$ is an order ideal of $P$ with $p \in \max(I)$,
and $c=(I_1 \subseteq \cdots \subseteq I_m)$ is 
an $m$-element multichain in $J(P)$ that passes through $I$, and
\item
the set of all
pairs $(I',c')$ in which $I'$ is an order ideal of $P$ with 
$p \in \min(P \setminus I')$,
and $c=(I'_1 \subseteq \cdots \subseteq I'_m)$ is 
an $m$-element multichain in $J(P)$ that passes through $I'$. 
\end{itemize}
We provide a bijection between these two sets.  Given $(I,c)$,
define two consecutive intervals of indices 
$$
\begin{aligned}[]
[i_0+1,i_0+a]&:=\{ i: p \in \min(P \setminus I_i) \}, \text{ and}\\
[i_0+a+1,i_0+a+b]&:=\{ i: p \in \max(I_i) \},\\
\end{aligned}
$$
so that $a,b \geq 0$, and 
one must have $I=I_{i_0+a+b_0}$ for some $b_0$ in the range 
$1 \leq b_0 \leq b$.
To define the desired bijection, map $(I,c) \mapsto (I',c')$, where 
$I'_j \setminus \{p\} = I_j \setminus \{p\}$ for all $j$, and 
$I'_j=I_j$ if $j \notin [i_0+1,i_0+a+b]$, 
but
$$
\begin{aligned}[]
[i_0+1,i_0+b]&:=\{ i: p \in \min(P \setminus I'_i) \}, \text{ and}\\
[i_0+1+b,i_0+a+b]&:=\{ i: p \in \max(I'_i) \},\\
\end{aligned}
$$
and $I':=I'_{i_0+b_0}$.  It is not hard to see that this map is a bijection. Its inverse is similarly defined.
\end{proof}

\begin{Remark}
\label{remark:Chan}
We can now further explain the connection between Theorem~\ref{main-theorem}(c), 
and the two results \cite[Corollary 2.15]{Chan1} 
and \cite[Theorem~1.2]{Chan2} that were alluded to in 
Section~\ref{section:intro}. 
The result \cite[Theorem~1.2]{Chan2} gives a vast generalization of \cite[Corollary 2.15]{Chan1}, which applies not only to Young's lattice intervals $[\varnothing,\lambda]$ with $\lambda = b^a$ a rectangle, but also applies to arbitrary intervals $P = [\mu,\lambda]$. 
Such intervals are always finite distributive lattices,
and the authors consider several families of toggle-symmetric probability distributions
on them, including
\begin{itemize}
\item 
the uniform distribution $\Omegaunif$ used in defining $\Xdd$, and 
\item the distribution $\Omegachain$ used in defining $\Ydd$.
\end{itemize}
They show that for any toggle-symmetric distribution on 
the Young's lattice interval $[\mu,\lambda]$, if one defines
$A$ and $B$ to be the number of nonempty rows and columns occupied by
the {skew shape} $\lambda/\mu$, then 
the down-degree random variable $d: P \rightarrow \NN$
has expectation
\begin{itemize}
\item 
given by a formula \cite[Theorem~1.2]{Chan2} showing it to be
{approximately} equal to $\frac{AB}{A+B}$, and
\item
{exactly} equal to $\frac{AB}{A+B}$ when $\lambda/\mu$ satisfies a condition
that they call \emph{balanced} \cite[Corollary~3.8]{Chan2}.  
\end{itemize}
We are lying slightly here, as 
the authors of \cite{Chan2} work not with down-degree, but with what they call \emph{jaggedness}, which is down-degree plus up-degree.  For toggle-symmetric probability distributions, this is equivalent to computing the expectation of down-degree: Definition~\ref{toggle-symmetry-definition}
immediately implies that a toggle-symmetric probability distribution 
assigns down-degree and up-degree the same expectation, 
which must therefore be half the expectation that it assigns to
the jaggedness statistic.

It is not hard to see that when $\mu=\varnothing$
and $\lambda=\delta_d \circ b^a$ is a rectangular staircase, 
then $\lambda/\mu=\lambda$ is balanced.
Since, in this case, $A=(d-1)a$ and $B=(d-1)b$, 
their result not only predicts our formula
from Theorem~\ref{main-theorem}(c), but also shows that
$[\varnothing, \lambda]$ is \mCDE, with
$$
\EE (X) = \EE (X^{{\scriptscriptstyle (}m{\scriptscriptstyle)}}) = \EE (Y) =\frac{AB}{A+B} = (d-1)\frac{ab}{a+b}.
$$
\end{Remark}

\begin{Remark}
After seeing this, one might wonder
whether some of the weak order intervals that our Theorem~\ref{main-theorem}
asserts are \CDE have the stronger \mCDE property.  
However, this can fail even for the intervals $[e,w]$ where $w$ is dominant
of rectangular staircase shape $\lambda=\lambdanab{d}{a}{b}$ when $d \geq 3$.  
For example, if $d=3$, $a=1$, and $b=2$, so that $\lambda=\lambdanab{3}{1}{2} = (4,2)$, 
and $w=53124 \in \symm_5$ is dominant of shape $\lambda$,
then the weak order interval $[e,w]$ has 
$$
\EE (X^{{\scriptscriptstyle (}m{\scriptscriptstyle)}})=
\frac{2(14m^3 + 111m^2 + 199m + 76)}{21m^3 + 168m^2 + 299m + 112}
$$ 
according to computations in SAGE.\footnote{{\tt SAGE} code for calculating $\EE (X^{{\scriptscriptstyle (}m{\scriptscriptstyle)}})$ as a rational function in $m$ is available from the first author.}   As predicted by Theorem~\ref{main-theorem}, this rational function has the correct value $(d-1)ab/(a+b)=4/3$ at $m=1$, and also in the limit as 
$m \rightarrow \infty$, but is not $4/3$ for integers $m\geq 2$.
\end{Remark}

\subsection{Examples of \CDE posets}

We begin with some simple instances of \CDE and \mCDE posets.

\begin{Example}
Finite disjoint unions of {chains} (that is, totally ordered sets) are \CDE because each of their elements lie on exactly one maximal chain, and thus $(\Omegaunif_P, \Xdd) = (\Omegachain_P, \Ydd)$.  If all of the chains have the same size, so that the poset is graded, then their union is also \mCDE. This is because, similarly,
$(\Omegaunif_P, \Xdd) = (\Omega^{{\scriptscriptstyle (}m{\scriptscriptstyle)}}_P,X^{{\scriptscriptstyle (}m{\scriptscriptstyle)}})=(\Omegachain_P, \Ydd)$.
On the other hand, one can check that when the chains have different sizes, the poset is \CDE but might not be \mCDE.
\end{Example}

The following poset family is similarly straightforward,
and will be used in the proof of Theorem~\ref{minuscule-theorem}.

\begin{Example}
\label{four-parameter-family}
Consider the following poset $P_{a,b,c,d}$, parametrized by four positive integers $a,b,c,d$.
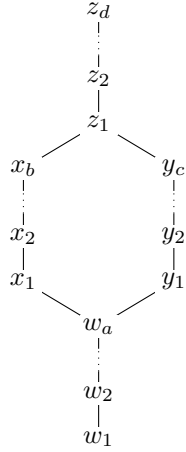
$$\begin{tikzpicture}[xscale=1,yscale=.6]
\draw (0,-.5) -- (0,1);
\draw[dotted] (0,1) -- (0,1.5);
\draw (0,1.5) -- (0,2) -- (1,3) -- (1,4.5);
\draw (0,2) -- (-1,3) -- (-1,4.5);
\draw[dotted] (1,4.5) -- (1,5);
\draw[dotted] (-1,4.5) -- (-1,5);
\draw (1,5) -- (1,5.5) -- (0,6.5) -- (0,8);
\draw (-1,5) -- (-1,5.5) -- (0,6.5);
\draw[dotted] (0,8) -- (0,8.5);
\draw (0,8.5) -- (0,9);
\foreach \x in {(0,-.5), (0,.5), (0,2), (-1,3), (-1,4), (-1,5.5), (1,3), (1,4), (1,5.5), (0,6.5), (0,7.5), (0,9)} {\fill[white] \x++(-.25,-.25) rectangle ++(.5,.5);}
\draw (0,-.5) node {$w_1$};
\draw (0,.5) node {$w_2$};
\draw (0,2) node {$w_a$};
\draw (-1,3) node {$x_1$};
\draw (-1,4) node {$x_2$};
\draw (-1,5.5) node {$x_b$};
\draw (1,3) node {$y_1$};
\draw (1,4) node {$y_2$};
\draw (1,5.5) node {$y_c$};
\draw (0,6.5) node {$z_1$};
\draw (0,7.5) node {$z_2$};
\draw (0,9) node {$z_d$};
\end{tikzpicture}$$
Fix $m \geq 1$, and denote by $f(p)$
the number of $m$-element multichains through $p$.
One then computes
$$
\begin{aligned}
f(w_i)&=f(z_j) \text{ is constant for all }i=1,2,\ldots,a,\text{ and }j=1,2,\ldots,d,\\
f(x_i)&\text{ is constant for }i=1,2,\ldots,b, \text{ and }\\
f(y_i)&\text{ is constant for }i=1,2,\ldots,c.
\end{aligned}
$$
Thus
\begin{align*}
\EE(X^{{\scriptscriptstyle (}m{\scriptscriptstyle)}}) &=\frac{(a-1) f(w_i) 
          + b f(x_i)
          +c f(y_i)
          +(d-1)f(z_i)+2 f(z_i)}
           {a f(w_i) 
          + b f(x_i)
          + c f(y_i)
          + d f(z_i)}\\
&= \frac{(a+d) f(w_i) 
          + b f(x_i)
          +c f(y_i)}
           {(a+d) f(w_i) 
          + b f(x_i)
          + c f(y_i)}\\
&= 1,
\end{align*}
and
$$\EE(\Ydd)
 =\displaystyle\frac{2 \cdot 0 + 2 (a-1)
	  + 1 \cdot b
          + 1 \cdot c
                     + 2 \cdot 2 + 2(d-1) 
          }{2(a + d) + b + c} 
          = 1,$$
so every poset $P_{a,b,c,d}$ is both \mCDE and \CDE, 
whether it is graded (that is, whether $b=c$) or not.
\end{Example}

The list of \CDE posets in Section~\ref{section:intro} mentioned another important family:
the \emph{minuscule posets}, which arise in the representation
theory of Lie algebras, and have many amazing enumerative properties (see, for example, \cite[Chapter 11]{Green} and \cite{Proctor}).
Up to poset isomorphism, the connected minuscule posets 
can be classified into three infinite families and two exceptional cases:
\begin{enumerate}\renewcommand{\labelenumi}{(\alph{enumi})}
\item the Cartesian product of two chains,
\item the interval $[\varnothing, b^2]$ in Young's lattice,
\item the special case $P_{a,1,1,a}$ of the posets $P_{a,b,c,d}$
from Example~\ref{four-parameter-family}, and
\item the posets $P(E_6)$ and $P(E_7)$ shown in Figure~\ref{fig:minuscule posets},
with each element $p$ labeled by $\#\{c \in \MaxChains(P): p \in c\}$.
\end{enumerate}

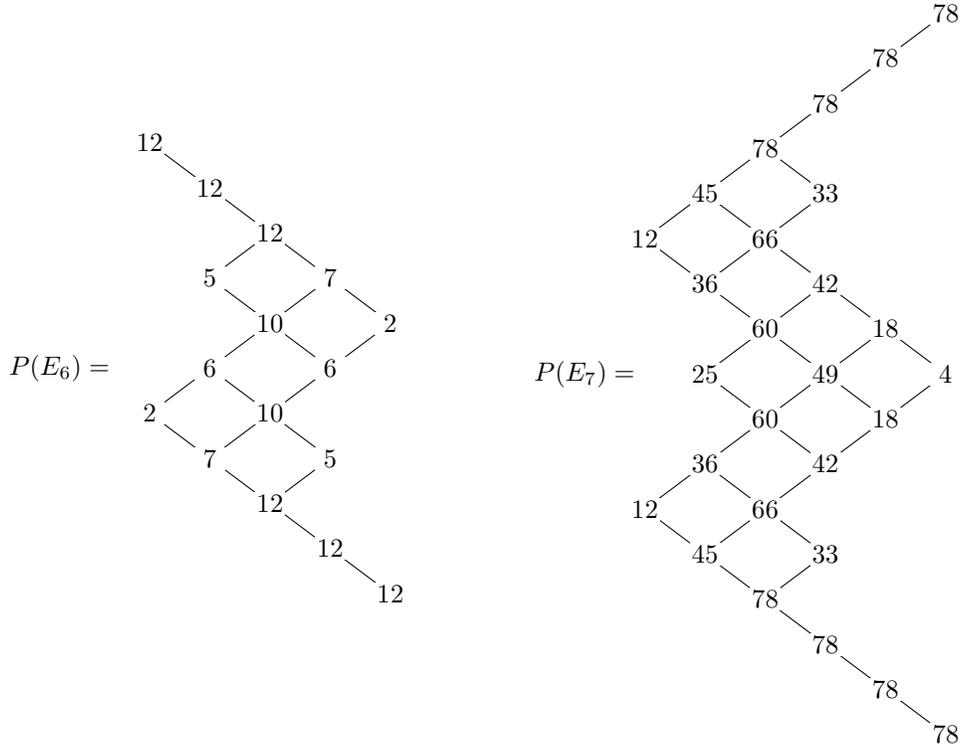
\begin{figure}[htbp] 
\begin{tikzpicture}[xscale=.8,yscale=.6]
\draw[white] (0,-7.5) -- (0,8.5);
\draw (0,0) -- (4,-4);
\draw (1,1) -- (3,-1);
\draw (1,3) -- (3,1);
\draw (0,6) -- (4,2);
\draw (2,-2) -- (3,-1);
\draw (1,-1) -- (4,2);
\draw (0,0) -- (3,3);
\draw (1,3) -- (2,4);
\foreach \x in {(0,0),(4,2)} {\fill[white] \x++(-.25,-.25) rectangle ++(.5,.5); \draw \x node {$2$};}
\foreach \x in {(3,-1),(1,3)} {\fill[white] \x++(-.25,-.25) rectangle ++(.5,.5); \draw \x node {$5$};}
\foreach \x in {(1,-1), (3,3)} {\fill[white] \x++(-.25,-.25) rectangle ++(.5,.5); \draw \x node {$7$};}
\foreach \x in {(1,1), (3,1)} {\fill[white] \x++(-.25,-.25) rectangle ++(.5,.5); \draw \x node {$6$};}
\foreach \x in {(2,0), (2,2)} {\fill[white] \x++(-.25,-.25) rectangle ++(.5,.5); \draw \x node {$10$};}
\foreach \x in {(0,6),(1,5),(2,-2),(2,4),(3,-3),(4,-4)} {\fill[white] \x++(-.25,-.25) rectangle ++(.5,.5); \draw \x node {$12$};}
\draw (-1.5,1) node {$P(E_6) = $};
\end{tikzpicture}
\hspace{.5in}
\begin{tikzpicture}[xscale=.8,yscale=.6]
\draw (0,0) -- (5,-5);
\draw (1,1) -- (3,-1);
\draw (1,3) -- (3,1);
\draw (0,6) -- (4,2);
\draw (1,7) -- (5,3);
\draw (2,8) -- (3,7);
\draw (2,-2) -- (3,-1);
\draw (1,-1) -- (5,3);
\draw (0,0) -- (4,4);
\draw (1,3) -- (3,5);
\draw (1,5) -- (3,7);
\draw (0,6) -- (5,11);
\foreach \x in {(0,0), (0,6)} {\fill[white] \x++(-.25,-.25) rectangle ++(.5,.5); \draw \x node {$12$};}
\foreach \x in {(1,-1), (1,7)} {\fill[white] \x++(-.25,-.25) rectangle ++(.5,.5); \draw \x node {$45$};}
\foreach \x in {(1,1), (1,5)} {\fill[white] \x++(-.25,-.25) rectangle ++(.5,.5); \draw \x node {$36$};}
\foreach \x in {(1,3)} {\fill[white] \x++(-.25,-.25) rectangle ++(.5,.5); \draw \x node {$25$};}
\foreach \x in {(2,-2), (3,-3), (4,-4), (5,-5), (2,8), (3,9), (4,10), (5,11)} {\fill[white] \x++(-.25,-.25) rectangle ++(.5,.5); \draw \x node {$78$};}
\foreach \x in {(2,0), (2,6)} {\fill[white] \x++(-.25,-.25) rectangle ++(.5,.5); \draw \x node {$66$};}
\foreach \x in {(2,2), (2,4)} {\fill[white] \x++(-.25,-.25) rectangle ++(.5,.5); \draw \x node {$60$};}
\foreach \x in {(3,-1), (3,7)} {\fill[white] \x++(-.25,-.25) rectangle ++(.5,.5); \draw \x node {$33$};}
\foreach \x in {(3,1), (3,5)} {\fill[white] \x++(-.25,-.25) rectangle ++(.5,.5); \draw \x node {$42$};}
\foreach \x in {(3,3)} {\fill[white] \x++(-.25,-.25) rectangle ++(.5,.5); \draw \x node {$49$};}
\foreach \x in {(4,2), (4,4)} {\fill[white] \x++(-.25,-.25) rectangle ++(.5,.5); \draw \x node {$18$};}
\foreach \x in {(5,3)} {\fill[white] \x++(-.25,-.25) rectangle ++(.5,.5); \draw \x node {$4$};}
\draw (-1,3) node {$P(E_7) = $};
\end{tikzpicture}
\caption{The minuscule posets $P(E_6)$ and $P(E_7)$, with elements labeled by 
the number of maximal chains passing through them.}\label{fig:minuscule posets}
\end{figure}

\begin{Theorem}
\label{minuscule-theorem} 
Connected minuscule posets are \mCDE, and, because they are graded, also \CDE.
\end{Theorem}
\begin{proof}
The above classification lets one verify this case-by-case.
\begin{enumerate}\renewcommand{\labelenumi}{(\alph{enumi})}
\item Products of two chains will be shown to be 
\mCDE in Corollary~\ref{chain-product-CDE-corollary}.
\item Intervals $[\varnothing,b^2]$ in Young's lattice are
\mCDE by Proposition~\ref{multichain-distribution-is-toggle-symmetric}
and \cite[Corollary~3.8]{Chan2}.
\item The posets $P_{a,1,1,a}$ of the family $P_{a,b,c,d}$
are \mCDE by Example~\ref{four-parameter-family}.
\item For $P(E_6)$ and $P(E_7)$, 
calculations
in {\tt SAGE} showed that 
$$
\begin{aligned}
&\EE(\Xdd_{P(E_6)}) 
=\EE(X^{{\scriptscriptstyle (}m{\scriptscriptstyle)}}_{P(E_6)})
 = \frac{5}{4}
  =\EE(\Ydd_{P(E_6)}), \text{ and} \\
\ \hspace{1.8in}  \raisebox{0in}[.3in][0in]{}
&\EE(\Xdd_{P(E_7)})
=\EE(X^{{\scriptscriptstyle (}m{\scriptscriptstyle)}}_{P(E_7)}) 
 = \frac{4}{3}
  =\EE(\Ydd_{P(E_7)}). \hspace{1.8in} \qedhere
\end{aligned}
$$
\end{enumerate}
\end{proof}

Accompanying Theorem~\ref{minuscule-theorem} is Theorem~\ref{Gaussian-poset-CDE-theorem}, concerning the distributive
lattice of order ideals $J(P)$ when $P$ is a minuscule poset.  We can characterize this lattice $J(P)$ 
in terms of the root system $\Phi$, the Weyl group $W$, 
and the minuscule dominant weight
$\omega$ or simple root $\alpha$ corresponding to $P$ (see \cite{Proctor}). In particular, one way to specify $P$ is to pick a minuscule simple root $\alpha$ and take the restriction of the poset of positive roots $\Phi^+$ to the positive roots lying weakly above $\alpha$. Then $J(P)$ has the following two reinterpretations.
\begin{itemize}
\item  $J(P)$ is the restriction of the 
(strong) Bruhat order to the set of minimum
length coset representatives for $W/W_\omega$, where $W_\omega$ is
the maximal parabolic subgroup fixing $\omega$.
\item $J(P)$ is the weight poset on the $W$-orbit of $\omega$,
which indexes the weight spaces (all having multiplicity one)
in the associated minuscule representation
of the Lie algebra.
\end{itemize}

\begin{Theorem}
\label{Gaussian-poset-CDE-theorem}
For (not necessarily connected) minuscule posets $P$, 
the distributive lattice $J(P)$ is \CDE.
\end{Theorem}

\begin{proof}
Because disjoint unions affect the lattice of order ideals in a convenient way, namely
$$J(P_1 \sqcup \cdots \sqcup P_k) \cong J(P_1) \times \cdots \times J(P_k),$$
and finite distributive lattices $J(P)$ are always graded,
we can apply Proposition~\ref{products-preserve-EEE} (below) to reduce to the case where $P$ is connected. Now we again rely upon the classification of connected minuscule posets $P$ preceding Theorem~\ref{minuscule-theorem}. 

For the family (a), where $P=\bm{a} \times \bm{b}$ is the Cartesian product of two chains $\bm{a}$ and $\bm{b}$
having $a$ and $b$ elements, respectively, we have that $J(P) \cong [\varnothing,b^a]$ is \CDE by Theorem~\ref{main-theorem}
(and, in fact, \mCDE by 
Proposition~\ref{multichain-distribution-is-toggle-symmetric}
and \cite[Corollary~3.8]{Chan2}).

For the family (b), where $P=[\varnothing, b^2]$ in Young's lattice, 
we have $J(P) \cong [\varnothing,\delta_b]_{\shifted}$ for the strict partition 
$\delta_{b+2}=(b+1,b,\ldots,3,2,1)$. S.~Hopkins \cite[Thm.~4.2]{Hopkins}
has shown using the methods of \cite{Chan2} that $J(P)$ is not only
\CDE, but actually \mCDE.

For the family (c) (that is, the special case $P=P_{a,1,1,a}$ among 
the posets $P_{a,b,c,d}$ from Example~\ref{four-parameter-family}), 
one finds that $J(P)  \cong P_{a+1,1,1,a+1}$, and hence it is also \CDE
(and, in fact, \mCDE).

For the family (d) of Figure~\ref{fig:minuscule posets},
one finds that $J(P(E_6)) \cong P(E_7)$, 
which we checked is \CDE (and, in fact, \mCDE) as part of
Theorem~\ref{minuscule-theorem}. We have checked 
separately (both by hand and by computer) that $J(P(E_7))$ is \CDE.
\end{proof}

In work that appeared since this article first circulated as a preprint, D.~Rush has presented a uniform proof of Theorem~\ref{Gaussian-poset-CDE-theorem} \cite[Theorem 1.5]{rush}, and even the following strengthening that we had conjectured.

\begin{Theorem}{\rm (\cite[Theorem~1.5]{rush})}
\label{Gaussian-poset-mCDE-conjecture}
For any minuscule poset $P$, 
the distributive lattice $J(P)$ is \mCDE, and hence \CDE.
\end{Theorem}

\noindent
In fact, Rush's result is stated slightly differently, in two ways:
\begin{itemize}
\item Strictly speaking, his result shows that $J(P)$ is \mCDE when $P$ is {\it connected}
and minuscule.  However, Proposition~\ref{mCDE-products-questions} below, due to S. Hopkins, then implies that $J(P)$ is 
\mCDE also for disconnected minuscule posets $P$.
\item Rush's result actually shows that $J(P)$ for a connected minuscule poset $P$ has
the stronger property of being $\tCDE$, as introduced by Hopkins \cite[Defn.~2.5]{Hopkins}.
The fact that $\tCDE$ implies $\mCDE$ is essentially our Lemma~\ref{multichain-distribution-is-toggle-symmetric}; see \cite[Lem.~2.3]{Hopkins}.
\end{itemize}

\subsection{\CDE and poset operations}

Most poset operations do not consistently respect \CDE. For example, disjoint union does not preserve the \CDE property (Example~\ref{ex:disjoint union doesn't preserve cde}), nor does ordinal sum (Example~\ref{ex:ordinal sum doesn't preserve cde}). The Cartesian product of graded posets, however, is an exception. 

\begin{Proposition}
\label{products-preserve-EEE}
If two graded posets $P$ and $Q$ are \CDE, then their Cartesian
product $P \times Q$ is also \CDE.
\end{Proposition}

As Example~\ref{graded-hypothesis-necessary-example} will demonstrate, the ``graded'' assumption in Proposition~\ref{products-preserve-EEE} is essential, and thus the collection of all \CDE posets is not closed under Cartesian product.

Before embarking on the proof of Proposition~\ref{products-preserve-EEE},
we make an observation about graded posets.
Recall that a finite poset $P$ is graded with $\rank(P)=r$ 
if all maximal chains $c \in \MaxChains(P)$ contain
$r+1$ elements; that is, each $c$ has the form $\{p_0 \lessdot p_1 \lessdot \cdots \lessdot p_{r-1} \lessdot p_r\}$.
Here are some straightforward reformulations of 
$\EE(\Xdd)$ and $\EE(\Ydd)$, the
first of which was mentioned in Section~\ref{section:intro}:
\begin{align}
\label{EX-is-edge-density}
\EE(\Xdd) 
  &=\frac{ \#\{(q,p) \in P \times P: q \lessdot p\}}{\#P},\\
\nonumber\EE(\Ydd)
  &=\frac{ \#\{(c,q,p) \in \MaxChains(P) \times P \times P: 
                p \in c \text{ and }q \lessdot p\}}
         { \#\{(c,p) \in \MaxChains(P) \times P: p \in c \}},
\end{align}
and, in the case that $P$ is graded, $\EE(\Ydd)$ can be rephrased as
\begin{equation}\label{EY-graded-reformulation}
\frac{ \#\{(c,q,p) \in \MaxChains(P) \times P \times P: 
                p \in c \text{ and }q \lessdot p\}}
         { (\rank(P)+1) \cdot \#\MaxChains(P) } .
\end{equation}

\begin{proof}[Proof of Proposition~\ref{products-preserve-EEE}]
The down-degree function, $d_P: P \rightarrow \NN$, satisfies 
$
d_{P\times Q}(p,q)=d_P(p) + d_Q(q).
$
Thus, 
$$\begin{aligned}
\EE \left(X_{P\times Q} \right)
 &= \frac{1}{\#P \cdot \#Q} \sum_{(p,q) \in P \times Q} d_{P\times Q}(p,q)\\
 &= \frac{1}{\#P \cdot \#Q} \sum_{p \in P} \sum_{q \in Q} 
                        \left( d_{P}(p)+d_{Q}(q) \right) \\
 &= \frac{1}{\#P \cdot \#Q} \left( \#Q\sum_{p \in P} d_{P}(p)+
                         \#P \sum_{q \in Q} d_{Q}(q)\right)\\
 &= \frac{1}{\#P} \sum_{p \in P} d_{P}(p) 
         +\frac{1}{\#Q}\sum_{q \in Q} d_{Q}(q)\\
 &= \EE (X_P) + \EE (X_Q).
\end{aligned}
$$
It therefore only remains to show that when $P$ and $Q$ are graded, one has
$$\EE (Y_{P \times Q})=\EE (Y_P) + \EE (Y_Q).$$
If the rank of $P$ is $r_P$, then one can rephrase Expression~\eqref{EY-graded-reformulation} as 
$$
\EE (Y_P) =\frac{1}{(r_P+1)\#\MaxChains(P)} 
           \sum_{c_P \in \MaxChains(P)} \sum_{p \in c_P} d_P(p). 
$$
Thus,  regarding $d_P(p)$ as a variable, its coefficient
in $\EE (Y_P)$ (and also in $\EE (Y_P) + \EE (Y_Q)$) is 
\begin{equation}
\label{coefficient-in-one-expectation}
\frac{\#\{ c_P \in \MaxChains(P): p \in c_P\}}{(r_P+1)\#\MaxChains(P)}.
\end{equation}

We now argue that $d_P(p)$ has the same coefficient in
$\EE (Y_{P \times Q})$.  Note that maximal chains in $P \times Q$ are chains
that lie within the set $c_P \shuffle c_Q$ of all
shuffles of some pair $(c_P,c_Q)$ in $\MaxChains(P) \times \MaxChains(Q)$.
Therefore
$$
\EE \left(Y_{P \times Q} \right) 
 =\frac{1}{N} \quad \sum_{(c_P,c_Q)} \quad \sum_{(p,q)} \quad
           \sum_{\substack{c \in  c_P \shuffle c_Q:\\ (p,q) \in c}}
              \left( d_P(p) + d_Q(q) \right),
$$
where $(c_P,c_Q)$ runs over $\MaxChains(P) \times \MaxChains(Q)$ 
in the outer sum, $(p,q)$ runs over $c_P \times c_Q$ in the inner sum,
and
\begin{equation}
\label{size-of-product-omega}
N:= \#\MaxChains(P) \cdot \#\MaxChains(Q) \binom{r_P+r_Q}{r_P}(r_P+r_Q+1).
\end{equation}
The coefficient of $d_P(p)$ in $\EE (Y_{P \times Q})$ is therefore
\begin{equation}
\label{coefficient-in-product-expectation}
N^{-1} \cdot
 \#\big\{(c_P,c_Q,q,c): (c_P,c_Q) \in \MaxChains(P) \times \MaxChains(Q), 
    \text{ and }(p,q) \in c \in c_P \shuffle c_Q \big\}.
\end{equation}
A priori, because the posets are graded, the number of pairs $(q,c)$ completing a 
quadruple $(c_P,c_Q,q,c)$ as above should not depend on the 
chain $c_P$ or $c_Q$, as long as $p$ lies in $c_P$.  
Thus, one might as well replace $c_P$ and $c_Q$ by fixed chains $[0,r_P]$ and $[0,r_Q]$ 
of the appropriate ranks, and fix $i:=\rank_P(p)$ in the chain
$[0,r_P]$, while letting $q$ vary over all values $j$
in the chain $[0,r_Q]$.
Then Expression~\eqref{coefficient-in-product-expectation} may be rewritten as
\begin{equation}
\label{rewritten-product-coefficient}
N^{-1} \cdot 
 \#\big\{c_P \in \MaxChains(P):p \in c_P\big\} \cdot \#\MaxChains(Q) \cdot
 \#\big\{(j,c): (i,j) \in c \in [0,r_P] \shuffle [0,r_Q] \big\}.
\end{equation}
The cardinality of the set of pairs $(j,c)$ in this set is
$\binom{r_P + r_Q+1}{r_P+1}$, via the bijection
$$
\begin{array}{ccc}
\big\{ (j,c): (i,j) \in c \in [0,r_P] \shuffle [0,r_Q] \big\}
&\longrightarrow & [0,r_P+1] \shuffle [0,r_Q] \\
(j,c) & \longmapsto & c'
\end{array}
$$
which forms $c'$ from $(j,c)$ by adding an extra step to $c$ of the form $(i,j) \rightarrow (i+1,j)$, just after $c$ passes through $(i,j)$.  The reverse bijection ``contracts out'' of $c'$ its unique step of the form $(i,j) \rightarrow (i+1,j)$ for some $j$, producing $c$ in the pair $(j,c)$.

Plugging this and Equation~\eqref{size-of-product-omega} into Expression~\eqref{rewritten-product-coefficient}, yields
the coefficient of $d_P(p)$ in $\EE (Y_{P \times Q})$:
$$
\frac{\#\{ c_P \in \MaxChains(P): p \in C_P\} \cdot 
          \#\MaxChains(Q) \binom{r_P + r_Q+1}{r_P}}
     {\#\MaxChains(P) \cdot \#\MaxChains(Q) \binom{r_P+r_Q}{r_P}(r_P+r_Q+1)}  
=
\frac{\#\{ c_P \in \MaxChains(P): p \in C_P\}}{(r_P+1)\#\MaxChains(P)}.
$$
This is the same as its coefficient in $\EE (Y_P) + \EE (Y_Q)$, given in Expression~\eqref{coefficient-in-one-expectation}, completing the proof. 
\end{proof}

\begin{Example}
\label{graded-hypothesis-necessary-example}
Both $P$ and $Q$ must be graded in Proposition~\ref{products-preserve-EEE},
as illustrated by the following non-\CDE product of two \CDE posets.
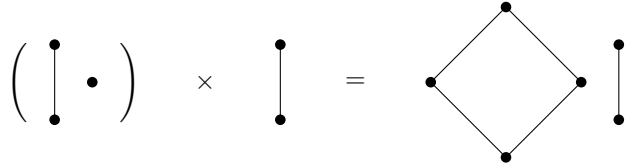
$$\begin{tikzpicture}
\draw (-.5,.5) node {$\Bigg($};
\foreach \x in {(0,0), (0,1), (.5,.5)} {\fill \x circle (2pt);}
\draw (0,0) -- (0,1);
\draw (1,.5) node {$\Bigg)$};
\draw (2,.5) node {$\times$};
\foreach \x in {(3,0), (3,1)} {\fill \x circle (2pt);}
\draw (3,0) -- (3,1);
\draw (4,.5) node {$=$};
\foreach \x in {(6,-.5), (6,1.5), (7,.5), (5,.5), (7.5,0), (7.5,1)} {\fill \x circle (2pt);}
\draw (6,-.5) -- (5,.5) -- (6,1.5) -- (7,.5) -- (6,-.5);
\draw (7.5,0) -- (7.5,1);
\end{tikzpicture}
$$
\end{Example}

Note that Proposition~\ref{products-preserve-EEE} immediately implies that
all finite products of chains are \CDE, and, in particular, that finite Boolean algebras are \CDE.
Moreover, such products enjoy the stronger \mCDE property, as discussed in
Corollary~\ref{chain-product-CDE-corollary} below.

Example~\ref{graded-hypothesis-necessary-example}
also shows that disjoint unions $P_1 \sqcup P_2$ of
\CDE posets need not be \CDE.

\begin{Example}\label{ex:disjoint union doesn't preserve cde}
The poset product in Example~\ref{graded-hypothesis-necessary-example}, is isomorphic to $P \sqcup Q$
for two Boolean algebras $P$ and $Q$. Boolean algebras are \CDE, by 
Proposition~\ref{products-preserve-EEE}, 
but the disjoint union $P \sqcup Q$ depicted above is not.  In fact, this can fail even when the two posets in the disjoint union are both graded and of the same rank. For example, both
$$
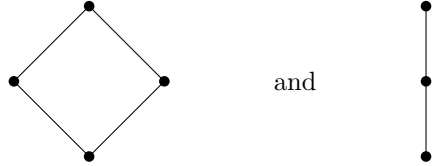
\begin{tikzpicture}
\foreach \x in {(0,0), (0,2), (1,1), (-1,1)} {\fill \x circle (2pt);}
\draw (0,0) -- (-1,1) -- (0,2) -- (1,1) -- (0,0);
\end{tikzpicture}
\hspace{.5in}
\begin{tikzpicture}
\draw (0,0) node {and};
\draw[white] (0,-1) circle (2pt);
\end{tikzpicture}
\hspace{.5in}
\begin{tikzpicture}
\foreach \x in {(0,0), (0,1), (0,2)} {\fill \x circle (2pt);}
\draw (0,0) -- (0,2);
\end{tikzpicture}
$$
are \CDE, but their disjoint union is not.
\end{Example}

The next example shows that ordinal sum, like disjoint union,
does not always preserve the \CDE property.

\begin{Example}\label{ex:ordinal sum doesn't preserve cde}
Let $P$ be a $1$-element antichain and $Q$ a $2$-element antichain. Both of these posets are \CDE because $\EE (X) = \EE (Y) = 0$ in each case. However, their ordinal sum
$$
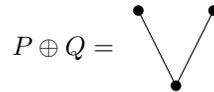
\begin{tikzpicture}
\draw (-1.5,.5) node {$P \oplus Q =$};
\draw (-.5,1) -- (0,0) -- (.5,1);
\foreach \x in {(-.5,1),(0,0),(.5,1)} {\fill \x circle (2pt);}
\end{tikzpicture}
$$
is not \CDE because 
$$\EE (X_{P \oplus Q}) = 2/3 \text{\ while $\EE (Y_{P\oplus Q}) = 1/2$.}$$
\end{Example}

We noted earlier that
$
\EE(\Xdd_P)=\EE(\Xdd_{P^*})
$
for any finite poset $P$.
However, there exist posets for which
$\EE(\Ydd_P) \neq \EE(\Ydd_{P^*})$. Moreover, poset duality does not preserve the \CDE property.

\begin{Example}
\label{duality-does-not-preserve-CDE}
Consider the following pair of dual posets $P$ and $P^\opp$, with each element
labeled by the number of maximal chains passing through it.
$$
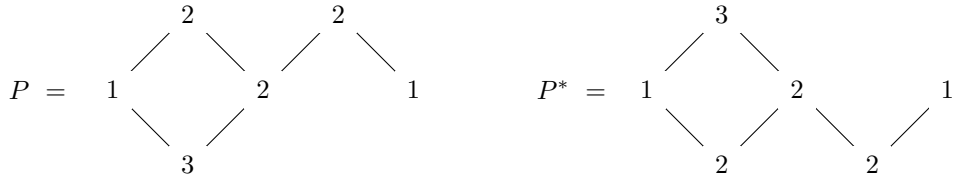
\begin{tikzpicture}
\draw (0,0) -- (1,-1) -- (3,1) -- (4,0);
\draw (0,0) -- (1,1) -- (2,0);
\foreach \x in {(0,0), (1,1), (1,-1), (2,0), (3,1), (4,0)} {\fill[white] \x++(-.25,-.25) rectangle ++(.5,.5);}
\draw (-1,0) node {$P \ =$};
\draw (0,0) node {$1$};
\draw (1,-1) node {$3$};
\draw (1,1) node {$2$};
\draw (2,0) node {$2$};
\draw (3,1) node {$2$};
\draw (4,0) node {$1$};
\end{tikzpicture}
\hspace{.5in}
\begin{tikzpicture}
\draw (0,0) -- (1,1) -- (3,-1) -- (4,0);
\draw (0,0) -- (1,-1) -- (2,0);
\foreach \x in {(0,0), (1,1), (1,-1), (2,0), (3,-1), (4,0)} {\fill[white] \x++(-.25,-.25) rectangle ++(.5,.5);}
\draw (-1,0) node {$P^* \ =$};
\draw (0,0) node {$1$};
\draw (1,1) node {$3$};
\draw (1,-1) node {$2$};
\draw (2,0) node {$2$};
\draw (3,-1) node {$2$};
\draw (4,0) node {$1$};
\end{tikzpicture}
$$
It is straightforward to calculate
$$
\begin{aligned}
\EE(\Xdd_P) = \EE(\Xdd_{P^\opp}) &=\frac{0+1+1+0+2+2}{6}=1 \text{ and}\\
\EE(\Ydd_{P}) 
 &=\frac{3 \cdot 0 
        +1 \cdot 1
        +2 \cdot 1
        +1 \cdot 0
        +2 \cdot 2
        +2 \cdot 2}{11}=1, \text{ while}\\
\EE(\Ydd_{P^\opp}) 
 &=\frac{2 \cdot 0 
        +2 \cdot 0
        +1 \cdot 1
        +2 \cdot 2
        +1 \cdot 1
        +3 \cdot 2}{11}=\frac{12}{11}.
        \end{aligned}
$$
Thus $P$ is \CDE, and $P^*$ is not\footnote{One might wonder, in light of this failure, why we have not {\it forced} 
the equivalence of the \CDE property for $P$ and $P^*$ by 
re-defining the random variables $X, Y$ to be
the {\it average of the down-degree and up-degree} of $p$ in the Hasse diagram $P$ (not just the down-degree).  We feel that this would be less interesting, by
artificially hiding the distinction between $P$ and $P^*$, 
and make some of our results, such as Theorem~\ref{main-theorem}(c), 
less striking.}.
\end{Example}

We return to the discussion after 
Example~\ref{graded-hypothesis-necessary-example}, where it was
claimed that finite products of chains are not only \CDE, but also \mCDE. 
Recall that Proposition~\ref{products-preserve-EEE} showed that the 
Cartesian product of two graded \CDE posets will also be \CDE,
but Example~\ref{graded-hypothesis-necessary-example} showed that
this can fail without the gradedness hypothesis. This raises a 
question about the \mCDE property, posed 
in an earlier version of this article, and later answered
affirmatively in the following proposition of S. Hopkins
while the article was under review.

\begin{Proposition}{\rm (\cite[Proposition~1.10]{Hopkins})}
\label{mCDE-products-questions}
The Cartesian product $P_1 \times P_2$ of two \mCDE posets $P_1$ and $P_2$, whether they are graded or not, is always \mCDE.
\end{Proposition}

\begin{Corollary}
\label{chain-product-CDE-corollary}
A product of chains $\bm{a_1} \times \cdots \times \bm{a_n}$
is \mCDE and \CDE, with 
$$\EE (\Xdd) = \EE (X^{{\scriptscriptstyle (}m{\scriptscriptstyle)}}) = \EE (\Ydd) = \sum_{k=1}^n \frac{a_k-1}{a_k}.$$ 
By setting $a_1 = \cdots = a_n = 2$, we see that Boolean algebras $\bm{2}^n=\bm{2} \times \cdots \times \bm{2}$ of
rank $n$ are  \mCDE and \CDE, with $\EE (\Xdd) = \EE (X^{{\scriptscriptstyle (}m{\scriptscriptstyle)}}) = \EE (\Ydd) = n/2$.
\end{Corollary}
\begin{proof}
The chain poset $P=\bm{a}$ having $a$ elements is $\mCDE$ and $\CDE$ because
in this case, the distribution
$\Omegachain_P$ and the distributions $\Omega_P^{{\scriptscriptstyle (}m{\scriptscriptstyle)}}$ used to defined the $\mCDE$ property all coincide with
the uniform distribution $\Omegachain_P$.
Therefore the product $\bm{a_1} \times \cdots \times \bm{a_n}$
is $\mCDE$ by Proposition~\ref{mCDE-products-questions}.  Since
the product has $a_1 a_2 \cdots a_n$ elements, and exactly 
$a_1 a_2 \cdots a_{k-1} (a_k-1) a_{k+1} \cdots a_n$ Hasse diagram edges
of the form $\{ (x_1,\ldots,x_k,\ldots,x_n), 
(x_1,\ldots,x_k+1,\ldots,x_n) \}$ for each $k$,
it has edge density
$$
\EE X \left( = \EE \Ydd = \EE (X^{{\scriptscriptstyle (}m{\scriptscriptstyle)}}) \right)
= \frac{1}{a_1 a_2 \cdots a_n}\sum_{i=1}^n a_1 a_2 \cdots a_{i-1} (a_i-1) a_{i+1} \cdots a_n
= \sum_{k=1}^n \frac{a_k-1}{a_k}. \qedhere
$$
\end{proof}


We now expand upon another topic related to poset operations, namely, duality. Despite Example~\ref{duality-does-not-preserve-CDE}, {self-duality} is relevant for the \CDE property.  The authors thank both S.~Fishel and T.~McConville for (independently) pointing out the
\CDE assertion in Proposition~\ref{Fishel-McConville-proposition} below.
Recall that a poset $P$ is \emph{self-dual} if one has a poset isomorphism $P \cong P^*$, and we will say that a poset $P$ is \emph{regular of valence $\Delta$} if every element $p$ in $P$ has the same vertex degree $\Delta$ in the Hasse diagram.

\begin{Proposition}
\label{Fishel-McConville-proposition}
A finite, self-dual poset $P$ that is regular of valence $\Delta$ is always
\mCDE and \CDE, with
$\EE (X) = \EE (Y) = \Delta/2$.
\end{Proposition}
\begin{proof}
Given a poset isomorphism $\alpha: P \rightarrow P^*$, we will
show a stronger assertion:  for any probability distribution on the underlying set $P=P^*$ that is {$\alpha$-invariant} in the sense that 
$\Prob(\alpha(p))=\Prob(p)$ for all $p \in P$, the expected value of the down-degree random variable $d: P \rightarrow \NN$ is $\Delta/2$.  To see this, one calculates the expected value of $d$ as follows:
$$
\begin{aligned}
\sum_{p \in P} \Prob(p) \cdot d(p) 
&=
\frac{1}{2} \left( \sum_{p \in P} \Prob(p) \cdot d(p) +
\sum_{p \in P} \Prob(\alpha(p))\cdot  d(\alpha(p)) \right) \\
&=
\frac{1}{2} \left( \sum_{p \in P} \Prob(p) \cdot \left( d(p) +  d(\alpha(p)) \right) \right) \\
&=
\frac{1}{2} \sum_{p \in P} \Prob(p) \cdot  \Delta\\
&=\frac{\,\Delta\,}{2}.
\end{aligned}
$$
The penultimate equality used the fact that the down-degree $d(\alpha(p))$ of $\alpha(p)$ in $P$
 is the same as the up-degree of $p$ in $P$, meaning that $d(p)+d(\alpha(p))$ is
the sum of the up- and down-degrees of $p$, which is $\Delta$.

Note that this then implies 
$\EE (X) = \EE (X^{{\scriptscriptstyle (}m{\scriptscriptstyle)}}) = \EE (Y)=\Delta/2$, 
because 
\begin{itemize}
\item the uniform distribution 
$\Omegaunif_P$ on $P=P^*$ used for $X$
is obviously $\alpha$-invariant, while 
\item the distributions $\Omega^{{\scriptscriptstyle (}m{\scriptscriptstyle)}}_P$ and $\Omegachain_P$ on $P=P^*$ 
used for $X^{{\scriptscriptstyle (}m{\scriptscriptstyle)}}$ and $Y$ are 
$\alpha$-invariant because $\alpha$
bijects the chains (respectively, $m$-element multichains) through $p$ in $P$ 
with the same chains (respectively, multichains) through $\alpha(p)$ in $P^*$. \hfill $\qedhere$
\end{itemize}
\end{proof}

Proposition~\ref{Fishel-McConville-proposition}
yields several interesting families of \mCDE and \CDE posets,
many of them non-graded, which we briefly discuss here.

\subsubsection{Simplicial arrangements and oriented matroids}

The first are the {weak orders} on the {chambers} of a {(central, essential) hyperplane arrangement} in $\RR^r$ (or, more generally, the {topes} of an {oriented matroid} of rank $r$). We will stick to the language of chambers and arrangements rather than the more general oriented matroid language here. Definitions and historical references can be found in \cite[\S4.2]{OMbook}.

All such weak orders have the same underlying graph for their Hasse diagram,
having vertices given by the chambers $C$ (the maximal cones into which the
arrangement dissects the space), and an edge $\{C, C'\}$ whenever 
two chambers $C$ and $C'$ are separated by exactly one hyperplane. 
When this graph is regular of valence $r$, 
the arrangement is called \emph{simplicial}.   
In particular, this occurs for the arrangements of reflecting hyperplanes in a finite reflection group $W$ of rank $r$, where chambers correspond to the
group elements, and the weak orders are all isomorphic 
to what is called the \emph{weak Bruhat order} on $W$. 
One defines one of the \emph{weak orders} on the set of chambers generally
by picking a base chamber $C_0$, and decreeing that $C \leq C'$ 
if every hyperplane separating $C_0$ from $C$ also separates $C_0$ from $C'$.  
The map $C \mapsto -C$ is a poset anti-automorphism, showing that all weak orders are self-dual. 
Proposition~\ref{Fishel-McConville-proposition} then immediately implies the following.

\begin{Corollary}
\label{simplicial-arrangement-corollary}
For a (central, essential) simplicial hyperplane arrangement in $\RR^r$, or 
simplicial oriented matroid of rank $r$, any of its weak orders on chambers is both \mCDE and \CDE, with $$\EE (X) = \EE (X^{{\scriptscriptstyle (}m{\scriptscriptstyle)}}) = \EE (Y) = \frac{\,r\,}{2}.$$
In particular, this is true for the weak Bruhat order 
on any finite reflection group $W$.
\end{Corollary}

\subsubsection{Tamari orders and some generalizations}
\label{Tamari-subsection}

The set of all triangulations of an $n$-sided polygon carries a well-known partial
order known as the \emph{Tamari order} \cite{TamariFestschrift}.  
The underlying graph for its Hasse diagram has an edge $\{T, T'\}$ if the triangulations $T$ and $T'$ differ only by a 
single \emph{diagonal flip}, that is, from one diagonal to the other 
inside a quadrangle triangulated by both $T$ and $T'$. After labeling 
the polygon vertices cyclically as $1,2,\ldots,n$, one 
has $T \lessdot T'$ if the diagonal flip exchanges the diagonal $\{i,k\}$  
for the diagonal $\{j,\ell\}$ within a quadrangle $ijk\ell$ that
has $(1 \leq) i < j < k <\ell (\leq n)$.
As an example, the special case of the CDE family $P_{1,1,2,1}$ 
from Example~\ref{four-parameter-family} is the \emph{Tamari lattice} \cite{TamariFestschrift}
on triangulations of a pentagon.  

The Tamari order on triangulations of an $n$-gon is regular of valence $n-3$, because each
triangulation has $n-3$ internal diagonals that one can flip;  in fact,  
it is also the $1$-skeleton of a {simple} $(n-3)$-dimensional polytope, called the \emph{associahedron} \cite[Example 9.11]{Ziegler}.
The Tamari order is self-dual, because the
map on the vertices swapping $i \leftrightarrow n+1-i$ reverses the order.

Proposition~\ref{Fishel-McConville-proposition} then immediately implies the following.

\begin{Corollary}
\label{Tamari-corollary}
Tamari order on triangulations of an $n$-gon is \mCDE and \CDE, 
with 
$$\EE (X) = \EE (X^{{\scriptscriptstyle (}m{\scriptscriptstyle)}}) = \EE (Y) = \frac{n-3}{2}.$$
\end{Corollary}

The authors thank T. McConville for also pointing out the following generalizations of
Tamari orders that are all valence-regular, and, in some cases, self-dual.  Valence-regularity
stems from the fact that, in each case, the object can be described as a partial order whose underlying
Hasse diagram is the graph of all maximal simplices in a pure $(\Delta-1)$-dimensional
simplicial complex with the {pseudomanifold} property (that is,
every $(\Delta-2)$-dimensional simplex lies in exactly two maximal simplices).

\begin{itemize}
\item
N. Reading \cite{Reading} defined a \emph{Cambrian lattice} $P$ associated
to each orientation of the Coxeter diagram of a finite Coxeter group $(W,S)$.
This $P$ is always regular of valence $|S|$.  It will be self-dual (and 
hence both \mCDE and \CDE by Proposition~\ref{Fishel-McConville-proposition})
whenever the opposite orientation corresponds to a 
diagram automorphism of $(W,S)$;  see \cite[Theorem~3.5]{Reading}.
The Tamari order is the special case when the Coxeter system $(W,S)$ is of
type $A$, and its Coxeter diagram is a path that is {equioriented} (that is, the arrows all point in the same direction along the path).
 
\item
Derksen, Weyman, and Zelevinsky \cite{DWZ} introduced the notion
of a \emph{quiver with potential} $(Q,W)$, and its associated
\emph{(complete) Jacobian algebra} $A:=\hat{J}(Q,W)$ over a field $k$.
The operation of {mutation} on $(Q,W)$ gives rise to its
{exchange graph}, which is regular of valence $|Q_0|$, the number of
nodes in the quiver $Q$.  When the $k$-algebra $A$ has {finite representation type} (that is, only finitely many indecomposable modules up to isomorphism), 
this exchange graph is finite.  Under this same 
representation-finite hypothesis,
the exchange graph also carries an 
orientation that is acyclic 
and whose transitive closure is a poset $P$
that coincides with both the poset $P$ of
{support-tilting modules} for $A$ 
and the poset of {torsion-free classes} for $A$;
see \cite[\S 2, \S 3, and Theorem~3.6]{BrustleYang}
and \cite{GarverMcConville}.
Additionally, the Hasse diagram of $P$ is equal to the
exchange graph; that is, none of the directed edges of the oriented
exchange graph are implied transitively by others.

If, furthermore, there is an algebra isomorphism $A^\opp \cong A$, then the poset
$P$ will be self-dual (and hence both \mCDE and \CDE by Proposition~\ref{Fishel-McConville-proposition}) \cite[Proposition~1.3]{IRTT}.  
This occurs, for example, whenever the potential $W=0$ and $Q$ is
a representation-finite quiver whose opposite orientation 
can be achieved by applying a graph automorphism.
The Tamari order again corresponds to the special case when the quiver 
is an {equioriented} path of type $A$.

\item
Santos, Stump, and Welker \cite{SSW} introduced the 
\emph{Grassman-Tamari orders} $GT_{k,n}$
on the set of all {maximal noncrossing} families of
$k$ element subsets of $\{1,2,\ldots,n\}$. The 
Tamari order is the special case $GT_{2,n}$.  McConville \cite{McConville} generalized this further in his \emph{grid orders} $GT(\lambda)$ 
where $\lambda$ is any {shape}, meaning any finite induced subgraph of the
$\ZZ \times \ZZ$ rectangular grid.  
When $\lambda$ is a $k \times (n-k)$ rectangle, one has $GT(\lambda)=GT_{k,n}$.

Let $\lambda^*$ be the result of rotating $\lambda$ by $180^{\circ}$, and let 
$\lambda^t$ denote the shape obtained from $\lambda$ by 
transposing rows and columns.  One can check that
that $GT(\lambda^*) \cong GT(\lambda)^\opp \cong GT(\lambda^t)$ \cite[Proposition~2.19]{SSW}.  Therefore the poset $P=GT(\lambda)$ is self-dual 
(and hence both \mCDE and \CDE by Proposition~\ref{Fishel-McConville-proposition})
whenever $\lambda$ is invariant under either $180^{\circ}$ rotation as in
the case of $GT_{k,n}$, or under transposition of rows and columns.

\item
Pilaud \cite{Pilaud} introduced the poset of \emph{$(k,n)$-twists}
on the set of all {$k$-triangulations} of a convex $(n + 2k)$-gon;
the case $k=1$ recovers the Tamari poset.
One can check that this poset is always self-dual 
(and hence both \mCDE and 
\CDE by Proposition~\ref{Fishel-McConville-proposition}) 
using its description as a quotient of the weak Bruhat 
order on $W=\symm_n$ by a congruence
that is preserved under the involutive anti-automorphism
$w \mapsto w_0 w$ \cite[Definition~26]{Pilaud}.

\end{itemize}

\subsection{Further \CDE conjectures and questions}

\subsubsection{Intervals in the shifted version of Young's lattice}

For a \emph{strict partition} $\lambda=(\lambda_1 > \lambda_2 > \cdots > \lambda_\ell)$,
the \emph{shifted Ferrers diagram} for $\lambda$ is drawn with each successive row
indented one position further than its predecessor. Some examples are shown below.
There is a shifted version of Young's lattice, which is simply its induced partial order on
the subset of all strict partitions.
In light of Theorem~\ref{main-theorem}, one might ask if there exist some
strict partitions $\lambda$ whose interval $[\varnothing, \lambda]_{\shifted}$ in
the shifted version of Young's lattice is \CDE.  The original version of our
article offered two conjectural families of such partitions, the first of which is the next theorem, proven in a more general form by S. Hopkins while this article was under review.

\begin{Theorem}(a special case of Hopkins \cite[Thm.~4.2]{Hopkins})
\label{shifted-falling-by-twos-conjecture}
For integers $\ell \geq 1$ and $0 \leq k < \ell/2$, the shifted Young's lattice
interval $[\varnothing, \lambda]_{\shifted}$ below
$\lambda=(\ell,\ell-2,\ell-4,\ldots,\ell-2k)$ is \CDE, with expectations $\EE (X) = \EE (Y) =|\lambda|/(\ell+1)$.
\end{Theorem}

Note that this result is independent of the parity of $\ell$.  For example, it 
applies to both of the shifted shapes $(8,6,4)$ and $(9,7,5,3)$ depicted here.
$$
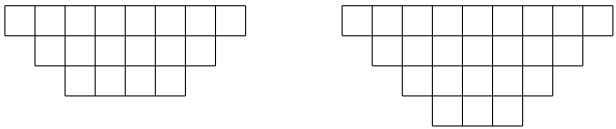
\begin{tikzpicture}[scale=.4]
\draw (0,0) -- (8,0);
\draw (0,-1) -- (8,-1);
\draw (1,-2) -- (7,-2);
\draw (2,-3) -- (6,-3);
\foreach \x in {0,8} {\draw (\x,0) -- (\x,-1);}
\foreach \x in {1,7} {\draw (\x,0) -- (\x,-2);}
\foreach \x in {2,3,4,5,6} {\draw (\x,0) -- (\x,-3);}
\draw[white] (3,-4) -- (5,-4);
\end{tikzpicture}
\hspace{.5in}
\begin{tikzpicture}[scale=.4]
\draw (0,0) -- (9,0);
\draw (0,-1) -- (9,-1);
\draw (1,-2) -- (8,-2);
\draw (2,-3) -- (7,-3);
\draw (3,-4) -- (6,-4);
\foreach \x in {0,9} {\draw (\x,0) -- (\x,-1);}
\foreach \x in {1,8} {\draw (\x,0) -- (\x,-2);}
\foreach \x in {2,7} {\draw (\x,0) -- (\x,-3);}
\foreach \x in {3,4,5,6} {\draw (\x,0) -- (\x,-4);}
\end{tikzpicture}
$$

The following conjecture, however, remains open.
\begin{Conjecture}
\label{shifted-square-staircase-conjecture}
For integers $a,d,e \ge 1$ with $d > a(e-1)+1$,  
the shifted Young's lattice interval $[\varnothing, \lambda]_{\shifted}$ below
$
\lambda=\delta_d + \delta_e \circ{a^a} 
$
is \CDE, with $\EE (X) =\EE (Y)=(d+a(e-1))/4$. 
\end{Conjecture}

We depict here the shifted shape $\delta_8 + \delta_3 \circ{2^2}$, with the cells of $\delta_3 \circ{2^2}$ shaded. 
$$
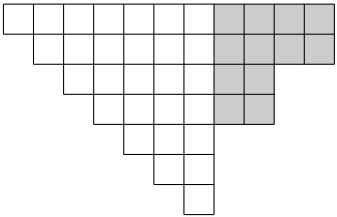
\begin{tikzpicture}[scale=.4]
\fill[black!20] (7,0) -- (11,0) -- (11,-2) -- (9,-2) -- (9,-4) -- (7,-4) -- (7,0);
\foreach \y in {0,-1} {\draw (0,\y) -- (11,\y);}
\draw (1,-2) -- (11,-2);
\draw (2,-3) -- (9,-3);
\draw (3,-4) -- (9,-4);
\draw (4,-5) -- (7,-5);
\draw (5,-6) -- (7,-6);
\draw (6,-7) -- (7,-7);
\draw (0,0) -- (0,-1);
\foreach \x in {1,10,11} {\draw (\x,0) -- (\x,-2);}
\foreach \x in {2} {\draw (\x,0) -- (\x,-3);}
\foreach \x in {3,9,8} {\draw (\x,0) -- (\x,-4);}
\foreach \x in {4} {\draw (\x,0) -- (\x,-5);}
\foreach \x in {5} {\draw (\x,0) -- (\x,-6);}
\foreach \x in {6,7} {\draw (\x,0) -- (\x,-7);}
\end{tikzpicture}
$$

\begin{Remark}
\label{positive-root-poset-remark}
The last theorem and conjecture overlap.  That is, 
Theorem~\ref{shifted-falling-by-twos-conjecture} with $(\ell,k)=(2N-1,N-1)$
and Conjecture~\ref{shifted-square-staircase-conjecture} 
with $(d,e,a)=(N+1,N,1)$ 
both assert that, for 
$\lambda=(2N-1,2N-3,\cdots,5,3,1)$, the interval
$[\varnothing,\lambda]_{\shifted}$ is \CDE with $\EE (X)=\EE (Y)=N/2$.  

Interestingly, this particular  interval $[\varnothing,\lambda]_{\shifted}$ is isomorphic to
the distributive lattice of $J(\Phi_W^+)$ of order ideals in the usual \emph{poset of 
positive roots $\Phi_W^+$} for the root systems of types $W=B_N$ or $C_N$.
It should be noted that for the root system of type $W=A_{d-1}$, 
the same lattice $J(\Phi_W^+)$ is isomorphic to the usual Young's lattice interval 
$[\varnothing, \delta_d]$, 
and hence is shown to be \CDE as part of Theorem~\ref{main-theorem}.  
Unfortunately, for the root system of type $D_4$, it 
was checked that $J(\Phi_{D_4}^+)$ is not \CDE.
\end{Remark}

\subsubsection{A few negative examples}

\begin{itemize}
\item Recall from Section~\ref{section:intro} that weak Bruhat order on a finite Coxeter group is \CDE (see Corollary~\ref{simplicial-arrangement-corollary}). One might ask whether \emph{strong Bruhat order} has the same property, 
but this fails already for the strong Bruhat order on the symmetric group $\symm_3$, shown here, because $\EE (X) = 4/3$ and $\EE (Y) = 5/4$.
$$
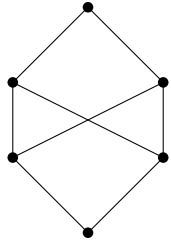
\begin{tikzpicture}
\foreach \x in {(0,0), (-1,1), (1,1), (-1,2), (1,2), (0,3)} {\fill \x circle (2pt);}
\draw (0,0) -- (-1,1) -- (-1,2) -- (0,3) -- (1,2) -- (1,1) -- (0,0);
\draw (-1,1) -- (1,2);
\draw ( 1,1) -- (-1,2);
\end{tikzpicture}
$$

\item
In light of Theorem~\ref{Gaussian-poset-CDE-theorem},
one might wonder whether to expect, more generally, that the distributive lattices
$J(P \times \bm{k})$, with $P$ minuscule, will always be \CDE.
However, this fails already for the first minuscule family,
because $J(\bm{a} \times \bm{b} \times \bm{k})$ is not \CDE
for $a=b=k=2$.

\item
In light of Remark~\ref{positive-root-poset-remark}, one might ask whether the posets $P=\Phi_W^+$
of positive roots for $W$ of types $A$ or $B/C$ might themselves be \CDE.  However, small examples
show that this is not the case.

\item
One can easily check that the \CDE property fails for the five-element modular, non-distributive lattice depicted below, which happens to be both the lattice of partitions of the set $\{1,2,3\}$ and the $n=q=2$ instance of the lattice of subspaces of $(\mathbb{F}_q)^n$.
$$
\begin{tikzpicture}
\foreach \x in {(0,0), (0,2), (1,1), (-1,1), (0,1)} {\fill \x circle (2pt);}
\draw (0,0) -- (-1,1) -- (0,2) -- (1,1) -- (0,0);
\draw (0,0) -- (0,2);
\end{tikzpicture}
$$

\item
Corollaries~\ref{Tamari-corollary} and \ref{simplicial-arrangement-corollary} 
might make one might wonder whether any of the following poset families
\begin{itemize}
\item Bergeron and Pr\'eville-Ratelle's \emph{$m$-Tamari} lattices \cite{Bergeron},
\item Kapranov and Voevodsky's \emph{higher Stasheff-Tamari} posets \cite{Kapranov},
\item Manin and Schechtman's \emph{higher Bruhat orders} \cite{Manin}, or
\item Law and Reading's lattice of \emph{diagonal rectangulations} \cite{LawReading},
\end{itemize}
all of which are related to Tamari and weak Bruhat orders, might be \CDE. However, in each case, we found small counterexamples.

\end{itemize}

\section{Young's lattice and tableaux}
\label{tableaux-section}

Computing $\EE(\Xdd)$ and $\EE(\Ydd)$
for Young's lattice intervals $[\varnothing,\lambda]$
and their duals involves various known flavors of tableaux.  
In this section, we review their definitions,
provide formulas to count them generally,
and then specialize to rectangular staircase shapes.

\begin{Definition}
A \emph{(set-valued) filling} $T$ of \emph{shape} $\lambda$ is an assignment of 
a finite subset $T(x) \subset \{1,2,\ldots\}$ to each cell $x$ in the Ferrers diagram of $\lambda$. Define the monomial
\begin{equation}\label{eqn:tableau monomial}
\xxx^T := \prod_{j \in T(y)} x_j
\end{equation}
as $y$ runs through the cells
of $\lambda$.
\end{Definition}

Several classes of fillings are of particular relevance to this work.

\begin{Definition}\label{defn:set-valued tableaux}
A \emph{column-strict set-valued tableau} $T$ of
shape $\lambda$ is a filling $T$ in which
\begin{itemize}
\item $\max \ T(x) \leq \min \ T(x')$ when 
$x$ is to the left of $x'$ in the same row of $\lambda$, and
\item $\max \ T(x) < \min \ T(x')$ when 
$x$ is above $x'$ in the same column of $\lambda$.
\end{itemize}
This $T$ is a \emph{standard set-valued tableau} if $\xxx^T = x_1x_2x_3\cdots x_N$ for some integer $N$. A column-strict set-valued tableau $T$ is a \emph{(column-strict) tableau} if $\#T(x)=1$
for every cell $x \in \lambda$, while $T$ is \emph{barely set-valued} if $\#T(x)=1$ for
all $x \in \lambda$ with the exception of a unique $x_0 \in \lambda$ for which $\#T(x_0)=2$. 
\end{Definition}

Also useful are tableaux with row-by-row bounds on their values.

\begin{Definition}\label{defn:flags}
A column-strict set-valued tableau $T$ of shape $\lambda$
is \emph{flagged} by a sequence of positive integers $\varphi = (\varphi_1,\varphi_2,\ldots)$ if every cell $x$ in row $i$ 
of $\lambda$ satisfies $\max \ T(x) \leq \varphi_i$. The
sequence $\varphi$ is called a \emph{flag}.
\end{Definition}

\begin{Example}
$$
\begin{array}{cccccc}
\ \tableau{
\scriptstyle{123} & {3} & \scriptstyle{35}\\
\scriptstyle{46} & \scriptstyle{67} \\  
{7}} 
\ & \
\tableau{
\scriptstyle{12} & {3} & \scriptstyle{56}\\
{4} & \scriptstyle{79} \\
{8}}
\ & \
\tableau{
1 & 3 & 5\\
2 & 6 \\
4}
\ & \
\tableau{
1 & 1 & 3\\
2 & 3 \\
4} 
\ & \
\tableau{
1 & \scriptstyle{13} & 3\\
3 & 4 \\ 
4}
\ & \
\tableau{
1 & \scriptstyle{12} & 2\\
2 & 3  \\
3 &
}\
\\
\text{column-strict}
&\text{standard}
&\text{standard}
&\text{column-strict}
&\text{barely}
&\text{flagged by}\\
\text{set-valued}
&\text{set-valued}
&\text{tableau}
&\text{tableau}
&\text{set-valued}
&\varphi=(2,3,4)
\end{array}
$$
\end{Example}

\begin{Proposition}
\label{tableau-bijections}
Fix $\lambda=(\lambda_1,\ldots,\lambda_\ell)$ and
set $\varphi:=(2,3,4,\ldots)$.  There are bijections between
\begin{enumerate}\renewcommand{\labelenumi}{(\alph{enumi})}
\item standard tableaux of shape $\lambda$,
and maximal chains in $[\varnothing,\lambda]$, 
or in its dual $[\varnothing,\lambda]^*$;
\item standard barely set-valued tableaux of shape $\lambda$,
and triples $(c,\mu,\nu)$, where $c$ is a maximal chain in $[\varnothing,\lambda]$, 
$\mu \in c$, and $\nu \lessdot \mu$ but $\nu$ is not necessarily in $c$; 
\item standard barely set-valued tableaux of shape $\lambda$,
and triples $(c,\mu,\nu)$, where $c$ is a maximal chain in $[\varnothing,\lambda]^*$, 
$\mu \in c$, and $\nu \lessdot \mu$ but $\nu$ is not necessarily in $c$;
\item column-strict tableaux of shape $\lambda$ 
flagged by $\varphi$, and elements of $[\varnothing,\lambda]$; and
\item barely set-valued 
column-strict tableaux of shape $\lambda$ flagged by $\varphi$, 
and covering relations $\nu \lessdot \mu$ 
in $[\varnothing,\lambda]$.
\end{enumerate}
\end{Proposition}
\begin{proof}
For (a), the bijection sends $T$ to the maximal
chain whose $i$th step is the partition $\mu^{(i)}$
occupied by the values $\{1,2,\ldots,i\}$ in $T$.
For example, 
$$
\tableau{
1 & 3 & 5\\
2 & 6 \\
4}
\quad \longmapsto 
\left(
\varnothing \ , \ 
\ttableau{{}} \ , \
\ttableau{{}\\ {}} \ , \
\ttableau{{}&{}\\ {}} \ , \
\ttableau{{}&{}\\ {}\\ {}} \ , \
\ttableau{{}&{}&{}\\ {}\\ {}} \ , \
\ttableau{{}&{}&{}\\ {}&{}\\ {}}
\right).
$$
Maximal chains in the dual $[\varnothing,\lambda]^*$
are in bijection with those in $[\varnothing,\lambda]$, by reading the chain
backwards.

For (b), assume we are given a barely set-valued
standard tableau $T$ of shape $\lambda$, with $n:=|\lambda|$,
and let $T(x_0)=\{a_0<b_0\}$ for a unique cell $x_0$.  
Then the bijection sends $T \mapsto (c,\mu,\nu)$
in which the chain $c=(\mu^{(0)},\mu^{(1)},\ldots,\mu^{(n)})$
has $\mu^{(i)}$ the shape occupied by the
$i$th smallest values among 
$\{1,2,\ldots,n+1\} \setminus \{b_0\}$,
with $\mu:=\mu^{(b_0-1)}$, and $\nu:=\mu \setminus \{x_0\}$.  
For example, $T$ shown below has 
$\{1,2,\ldots,n+1\} \setminus \{b_0\}=\{1,2,3,4,6,7\}$, and
$$
\tableau{
1 & \scriptstyle{25} & 6\\
3 & 7 \\
4}
\quad \longmapsto (c,\mu,\nu),
\text{ with } 
c=\left(
\varnothing, 
\ttableau{{}} \ , \ 
\ttableau{{}&{}} \ , \ 
\ttableau{{}&{}\\{}} \ , \ 
\ttableau{{}&{}\\{}\\{}} \ , \ 
\ttableau{{}&{}&{}\\{}\\{}} \ , \ 
\ttableau{{}&{}&{}\\{}&{}\\{}}
\right) \text{ and }
\nu=\ttableau{{}\\{}\\{}}
\lessdot
\ttableau{{}&{}\\{}\\{}}=\mu=\mu^{(5-1)}.
$$

For (c), the bijection is similar to (b),
except that now $T \mapsto (c,\mu,\nu)$
where $c=(\mu^{(0)},\mu^{(1)},\ldots,\mu^{(n)})$ 
has $\mu^{(i)}$ as the shape occupied by the
$(n-i)$th smallest values among 
$\{1,2,\ldots,n+1\} \setminus \{a_0\}$,
with $\mu:=\mu^{(n-(a_0-1))}$ and $\nu:=\mu \cup \{x_0\}$.  
For example, $T$ shown below has 
$\{1,2,\ldots,n+1\} \setminus \{a_0\}=\{1,3,4,5,6,7\}$, and
$$
\tableau{
1 & \scriptstyle{25} & 6\\
3 & 7 \\
4}
\quad \longmapsto (c,\mu,\nu),
\text{ with } 
c=\left(
\ttableau{{}&{}&{}\\{}&{}\\{}} \ , \ 
\ttableau{{}&{}&{}\\{}\\{}} \ , \ 
\ttableau{{}&{}\\{}\\{}} \ , \ 
\ttableau{{}\\{}\\{}} \ , \ 
\ttableau{{}\\{}} \ , \ 
\ttableau{{}} \ , \ 
\varnothing
\right) \text{ and }
\nu=\ttableau{{}&{}}
\gtrdot
\ttableau{{}}=\mu=\mu^{(7-(2-1))}.
$$

For (d), the bijection sends the column-strict tableau, $T$, of shape $\lambda$ and flagged by $\varphi$, to the partition
$\mu$ describing the cells $x$ in $T$ which are filled with their 
row index $i$, rather than with the flag upper bound $i+1=\varphi_i$. This is a bijection because the column-strict and flagged requirements mean that all cells in row $i$ of $T$ must be filled either with $i$ or $i+1$, and the collection of cells whose filling matches the row index must form a subshape $\mu$ within $\lambda$; that is, $\mu \in [\varnothing,\lambda]$. To get from $\mu$ back to $T$, consider $\mu$ inside of $\lambda$, label all cells of that subshape by their row indices, and label all other cells of $\lambda$ by their row indices plus one. For example, in the following map from a column-strict tableau $T$ to its corresponding subshape $\mu$, the cells of $T$ that are filled by their row indices have been written in boldface:
$$
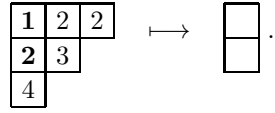
\tableau{
\mathbf{1} & 2 & 2\\
\mathbf{2} & 3 \\
4}
\quad \longmapsto \quad
\tableau{{}{}{}\\{}}\ .
$$

For (e), the bijection sends the tableau $T$ to the covering
relation $\nu \lessdot \mu$, where 
\begin{itemize}
\item $\nu$ gives the cells $x$ in $T$ filled by their row index,
\item the set difference (or \emph{skew shape}) $\lambda /\mu$
gives the cells $x$ filled by one more than their row index, and
\item the unique cell $x_0$ of $\mu/\nu$ is filled by 
$T(x_0)=\{i,i+1\}$, where $i$ is its row index.
\end{itemize}
For example, 
$$
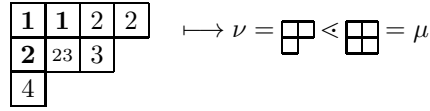
\tableau{
\mathbf{1} & \mathbf{1} & 2 & 2\\
\mathbf{2} & \scriptstyle{23}& 3 \\
4}
\quad \longmapsto 
\nu=\ttableau{{}&{}\\{}}
\lessdot
\ttableau{{}&{}\\{}&{}}=\mu
$$

\end{proof}

Recall that $f^\lambda$ is the number of standard tableau of shape $\lambda$. Let us now name the number of tableaux of each kind appearing in Proposition~\ref{tableau-bijections}.

\begin{Definition}
\label{tableaux-cardinality-definitions}
Fix a partition $\lambda=(\lambda_1,\ldots,\lambda_\ell)$.
\begin{enumerate}\renewcommand{\labelenumi}{(\alph{enumi})}
\item Let $f^\lambda(+1)$ be the number of standard barely set-valued tableaux
of shape $\lambda$.
\item Let $R(\lambda)$ be the number
of column-strict tableaux of shape 
$\lambda$ flagged by $\varphi:=(2,3,4,\ldots)$.
\item Let $R^{(+1)}(\lambda)$ be the number
of column-strict barely set-valued tableaux
of shape $\lambda$ flagged by $\varphi$.
\end{enumerate}
\end{Definition}

Combining Proposition~\ref{tableau-bijections} 
with Equations~\eqref{EX-is-edge-density} and~\eqref{EY-graded-reformulation} implies the following.
\begin{Corollary}
For any partition $\lambda$,
\begin{align}
\label{Young-X-expectation-reformulated}
\EE(\Xdd_{[\varnothing,\lambda]}) 
  &=\frac{R^{(+1)}(\lambda) }{R(\lambda)}
  =\EE(\Xdd_{[\varnothing,\lambda]^*}), \text{ and}\\
\label{Young-Y-expectation-reformulated}
\EE(\Ydd_{[\varnothing,\lambda]})
  &=\frac{f^\lambda(+1)}{(|\lambda|+1)f^\lambda}
  =\EE(\Ydd_{[\varnothing,\lambda]^*}).
\end{align}
\end{Corollary}

To count barely set-valued tableaux, our
strategy is to convert them to tableaux with extra data. The following definition is a special case of 
the map sketched to prove \cite[Theorem~6.11]{Buch02}.

\begin{Definition}
\label{uncrowding-definition}
Given a column-strict barely set-valued tableau $T$, its \emph{uncrowding} is 
the column-strict tableau $T^+$ obtained as follows:
if $T(x_0)=\{a_0<b_0\}$ for the unique cell $x_0$, then remove
$b_0$ from $x_0$ and  
use \emph{Robinson-Schensted-Knuth (RSK) row-insertion} 
(see, for example, \cite[\S 7.11]{StanleyEC2})
to bump $b_0$ into the rows of $T$ strictly below the row of $x_0$.
For example, 
$$
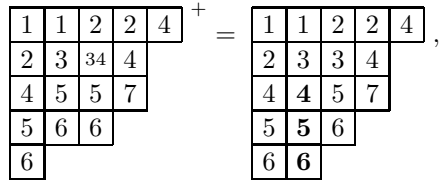
\tableau{
1&1&2&2&4\\
2&3&\scriptstyle{34}&4\\
4&5&5&7\\
5&6&6\\
6
}^{\ +} = \ 
\tableau{
1&1&2&2&4\\
2&3&3&4\\
4&\mathbf{4}&5&7\\
5&\mathbf{5}&6\\
6&\mathbf{6}
} \ ,
$$
where the bumped entries are in boldface.
\end{Definition}

\begin{Proposition}
\label{crowding-proposition}
The uncrowding operation $T \mapsto T^+$
gives a bijection between 
\begin{itemize}
\item column-strict barely set-valued
tableaux of shape $\lambda$, and
\item triples $(T^+,x,i_0)$ where 
 \begin{itemize}
  \item $T^+$ is a column-strict tableau,
  \item $x$ is one of its (inner) corner cells, and 
  \item $i_0$ is in the
range $1,2,\ldots,i-1$, where $i$ is the row-index of $x$.
 \end{itemize}
\end{itemize}
Under this bijection, the shapes $\lambda$ and $\lambda^+$ of
$T$ and $T^+$ are related by $\lambda \lessdot \lambda^+$
and $\lambda^+/\lambda=\{x\}$.  Furthermore, $i_0$ is the row-index
of the unique cell $x_0$ for which $\#T(x_0)=2$.
\end{Proposition}
\begin{proof}
Given $(T^+,x,i_0)$, the inverse bijection (``crowding'') starts
by doing \emph{reverse RSK row-insertion} out of the corner cell $x$ in
$T^+$.  However, rather than stopping when it 
reverse-bumps an entry out of row $1$, the procedure stops when an entry $b_0$ from
row $i_0+1$ is about to bump an entry $a_0$ of row $i_0$, say in
a cell $x_0$, and instead adds $b_0$ as an extra set-valued entry 
to make $T(x_0)=\{a_0,b_0\}$.
\end{proof}

\begin{Example}
If one selects the boldface corner cell $x$ in the figure below (row $5$, column $2$),
$$
T^+=\tableau{
1&1&2&2&4\\
2&3&3&4\\
4&4&5&7\\
5&5&6\\
6&\mathbf{6}
}
$$
then crowding the triples 
$(T^+,x,i_0)$ for $i_0=1,2,3,4$ yields these barely set-valued
tableaux, respectively.
$$
\tableau{
1&1&2&2&4\\
2&3&3&4\\
4&4&5&7\\
5&\scriptstyle{56}&6\\
6&
} \quad
\tableau{
1&1&2&2&4\\
2&3&3&4\\
4&\scriptstyle{45}&5&7\\
5&6&6\\
6&
} \quad
\tableau{
1&1&2&2&4\\
2&3&\scriptstyle{34}&4\\
4&5&5&7\\
5&6&6\\
6&
} \quad
\tableau{
1&1&2&\scriptstyle{23}&4\\
2&3&4&4\\
4&5&5&7\\
5&6&6\\
6&
}
$$
\end{Example}

Proposition~\ref{crowding-proposition} yields useful recurrences for counting barely set-valued tableaux.

\begin{Corollary}
\label{barely-set-valued-tableau-recurrence}
For any partition $\lambda$,
$$
f^\lambda(+1)
=\sum_{x=(i,j)} (i-1) f^{\lambda\cup\{x\}},
$$
where $x$ runs through all outside corner cells of $\lambda$.
\end{Corollary}
\begin{proof}
This is immediate from Proposition~\ref{crowding-proposition},
restricting the uncrowding bijection to standard barely set-valued tableaux $T$ 
of shape $\lambda$, so that $T^+$ in the triple 
$(T^+,x,i_0)$ is a standard tableau.
\end{proof}

\begin{Remark}
A second proof of 
Corollary~\ref{barely-set-valued-tableau-recurrence}
uses the fact that $f^\lambda$ and $f^{\lambda}(+1)$ are the
coefficients of $x_1 x_2 \cdots x_{|\lambda|}$ and 
$x_1 x_2 \cdots x_{|\lambda|+1}$ in the \emph{Schur function} $s_\lambda$
and the \emph{stable Grothendieck polynomial $G_\lambda$}; 
see Definition~\ref{Schub-Grothendieck-definitions} below.  
A formula of Lenart \cite{Lenart99} gives the expansion
\begin{equation}
\label{Lenarts-formula}
G_\lambda = 
\sum_{\mu \supset \lambda} (-1)^{|\mu/\lambda|} g_{\mu/\lambda} s_{\mu},
\end{equation}
where $g_{\mu/\lambda}$ is the number of \emph{row-strict and column-strict}
tableaux of the skew shape $\mu/\lambda$ with entries in row $i$
in the range $1,2,\ldots,i-1$.  If $\mu = \lambda \cup\{x\}$ 
with $x=(i,j)$, then $g_{\mu/\lambda}=i-1$.  
Thus, Corollary~\ref{barely-set-valued-tableau-recurrence} also follows
by extracting the coefficient of the square-free monomial 
$x_1 x_2 \cdots x_{|\lambda|+1}$ 
in Equation~\eqref{Lenarts-formula}.
\end{Remark}

\begin{Remark}
Note that $f^{\lambda^t}(+1)=f^{\lambda}(+1)$ via the conjugation involution 
on standard set-valued tableaux.  
Hence Corollary~\ref{barely-set-valued-tableau-recurrence}
implies a second identity; namely,
$$
f^{\lambda}(+1) = 
 \sum_{  x=(i,j) }     
    (j-1) f^{\lambda \cup \{x\}},
$$
where $x$ still runs through all outside corner cells of $\lambda$.
Concordance between these two identities requires the difference
between their righthand sides to equal zero; that is, we need
\begin{equation}
\label{corner-cell-content-has-mean-zero}
\sum_{ x=(i,j) }
c(x) f^{\lambda \cup \{x\}} = 0,
\end{equation}
where $c(x)=j-i$ is the \emph{content} of the cell $x=(i,j)$.
Indeed, dividing Equation~\eqref{corner-cell-content-has-mean-zero}
by $n \cdot f^\lambda$ gives a fact that was first shown by Kerov 
(\cite[Equation~(10.6)]{Kerov-FPSAC} and \cite{Kerov}):
the content $c(x)$ has mean $0$ when one considers
it as a random variable on the outside corners $x$ of
a random partition $\lambda$, grown one box at a time
using \emph{Plancherel measure}.
\end{Remark}

We can now prove part of Theorem~\ref{main-theorem}.

\begin{Proposition}
\label{Y-expectation-for-Young}
For the rectangular staircase $\lambda=\lambdanab{d}{a}{b}$, the 
interval $[\varnothing,\lambda]$ has
$
\EE(\Ydd)
=\frac{(d-1)ab}{a+b}. 
$
\end{Proposition}

\begin{proof}
Recall that Equation~\eqref{Young-Y-expectation-reformulated} says that
$
\EE(\Ydd)=\frac{ f^{\lambda}(+1) }{(|\lambda|+1) f^{\lambda} }.
$
The outside corners of $\lambda$ are
$$
x_i:= \big(1+b(d-1-i), 1+ai\big)$$ 
for $i=0,1,\ldots,d-1$, and the hook-length formula shows that for the cell $x_i$,
$$
\frac{ f^{\lambda\cup\{x_i\}} }{ (|\lambda|+1) f^{\lambda}}
= 
\frac{ \left(\frac{b}{a+b}\right)_i}{i!} \cdot
\frac{\left(\frac{a}{a+b}\right)_{d-1-i}}{(d-1-i)!}, 
$$
where $(z)_j:=z(z+1)\cdots (z+j-1)$ is the Pochhammer symbol.

Thus, Corollary~\ref{barely-set-valued-tableau-recurrence} implies
\begin{equation*}
\begin{aligned}
\EE(\Ydd)&=
\frac{ f^{\lambda}(+1) }{(|\lambda|+1) f^{\lambda} }\\
 &=\sum_{i=0}^{d-1} ai \cdot \frac{ f^{\lambda\cup\{x_i\}} }{ f^{\lambda}}\\
 &=  \sum_{i=1}^{d-1} ai \cdot 
\frac{ \left(\frac{b}{a+b}\right)_i}{i!} \cdot
\frac{\left(\frac{a}{a+b}\right)_{d-1-i}}{(d-1-i)!}\\  
 &= a \cdot 
   \frac{ \frac{b}{a+b} \left( \frac{a}{a+b} \right)_{d-2} }{(d-2)!} \cdot 
{}_2 F_{1} \left( \left. \begin{matrix} -(d-2) & \frac{b}{a+b}+1 \\ 
                 & -\left( d-3+\frac{a}{a+b} \right) \end{matrix} \right| 1 \right) \\ 
 &= a \cdot 
   \frac{ \frac{b}{a+b} \left( \frac{a}{a+b} \right)_{d-2} }{(d-2)!} \cdot 
      \frac{ \left( -(d-1) \right)_{d-2} } 
        { \left( -\left(d-3 + \frac{a}{a+b}\right) \right)_{d-2}}\\ 
&= \frac{(d-1)ab}{a+b}, 
\end{aligned}
\end{equation*}
where ${}_2 F_{1}$ is hypergeometric series notation, and the penultimate equality uses the Chu-Vandermonde summation 
$
{}_2 F_{1} \left( \left. \begin{matrix} -m & B \\
                 & C \end{matrix} \right| 1 \right)
=\frac{(C-B)_m}{(C)_m}.
$
\end{proof}

Recall from Section~\ref{section:intro} that
$
R(\lambda,q):=\sum_{\mu \subset \lambda} q^{|\mu|}
$
is the rank-generating function for $[\varnothing,\lambda]$,
and that for an outside corner cell $x=(i,j)$ of $\lambda$ in
row $i$ and column $j$, we defined two subshapes:
$$
\begin{aligned}
\lambda_{(x)}&:=(\lambda_{i+1},\lambda_{i+2},\ldots) \text{ and}\\
\lambda^{(x)}&:=(\lambda_1,\lambda_2,\ldots,\lambda_{i-1})-(j,j,\ldots,j).
\end{aligned}
$$
Recall also from Proposition~\ref{tableau-bijections} and Definition~\ref{tableaux-cardinality-definitions} that 
$
R(\lambda)=\#[\varnothing,\lambda]=\left[ R(\lambda,q) \right]_{q=1}
$
is the same as the number of column-strict tableaux of shape
$\lambda$ flagged by $\varphi=(2,3,4,\ldots)$, while $R^{(+1)}(\lambda)$ is
the number of column-strict barely set-valued tableaux of shape $\lambda$
flagged by $\varphi$.

\begin{Corollary}
\label{barely-set-valued-flagged-recurrence}
For any partition $\lambda$,
$$
R^{(+1)}(\lambda)
=\sum_{x=(i,j)} (i-1) \cdot R(\lambda_{(x)}) \cdot R(\lambda^{(x)}),
$$
where $x$ runs through all outside corner cells of $\lambda$.
\end{Corollary}
\begin{proof}
Restrict the domain of the uncrowding bijection to 
column-strict barely set-valued tableaux $T$
of shape $\lambda$ flagged by $\varphi=(2,3,4,\ldots)$.
Then, during the uncrowding of $T$,
the bumpings that occur are always to a value $i$ in row $i-1$ 
trying to bump a value $i+1$ in row $i$.  The bumping stops
only when this value $i$ comes to rest at a corner cell $x=(i,j)$ at
the end of row $i$, because row $i$ only contains the value $i$ (in particular, it has no $i+1$).
In this situation, the resulting column-strict tableau $T^+$
thus breaks into three pieces:  
\begin{itemize}
\item 
an $i \times j$ rectangle that is
weakly to the upper left of $x=(i,j)$, having every cell
filled by its row-index,
\item
a column-strict tableau $T^{(x)}$ strictly north and east of $x$,
filling $\lambda^{(x)}$, which is 
flagged by $\varphi$, and
\item
a column-strict tableau $T_{(x)}$ strictly south and west of $x$,
filling $\lambda_{(x)}$, which, after reducing all of its
entries by $i$, would be flagged by $\varphi$.
\end{itemize}
For example, here is such a $T$ and $T^+$, along with 
$T^{(x)}$ and the version of $T_{(x)}$ with entries reduced by $i=5$.
$$
T=\tableau{
1&1&1&1&1&2\\
2&2&2&2&3\\
3&3&3&\scriptstyle{34}&4\\
4&4&5&5\\
5&5\\
6&7\\
7
}
\quad \longmapsto \quad
T^+=\tableau{
\mathbf{1}&\mathbf{1}&\mathbf{1}&1&1&2\\
\mathbf{2}&\mathbf{2}&\mathbf{2}&2&3\\
\mathbf{3}&\mathbf{3}&\mathbf{3}&3&4\\
\mathbf{4}&\mathbf{4}&\mathbf{4}&5\\
\mathbf{5}&\mathbf{5}&\mathbf{5}\\
6&7\\
7
}
\quad \longmapsto \quad
\tableau{
&&&1&1&2\\
&&&2&3\\
&&&3&4\\
&&&5\\
&&\bullet\\
1&2\\
2
}
$$
The $5 \times 3$ rectangle weakly to the upper left of cell $(5,3)$ in $T^+$ has every cell filled by its row-index, in boldface.
\end{proof}

This lets us prove another part of Theorem~\ref{main-theorem}.

\begin{Proposition}
\label{X-expectation-for-Young}
For the rectangular staircase $\lambda=\lambdanab{d}{a}{b}$, the 
interval $[\varnothing,\lambda]$ has
$
\EE(\Xdd)
=\frac{(d-1)ab}{a+b}. 
$
\end{Proposition}

\begin{proof}
Recall that Equation~\eqref{Young-X-expectation-reformulated} asserts that 
$
\EE(\Xdd)=\frac{R^{(+1)}(\lambda)}{R(\lambda)}.
$
As observed earlier, $\lambda$ has outside corners 
$$
x_k:= \big(1+b(d-1-k), 1+ak\big) 
\text{ for }k=0,1,\ldots,d-1. 
$$ 
Therefore Corollary~\ref{barely-set-valued-flagged-recurrence} implies
$$
R^{(+1)}(\lambda)
=\sum_{k=0}^{d-1} ak \cdot R( \lambda_{(x_k)} ) \cdot R( \lambda^{(x_k)} ). 
$$
After dividing by $a$, this yields
\begin{equation}
\label{recurrence-divided-by-a}
\frac{ R^{(+1)}(\lambda) }{a}
=\sum_{k=0}^{d-1} k \cdot R( \lambda_{(x_k)} ) \cdot R( \lambda^{(x_k)} ). 
\end{equation}
The conjugation involution $\mu \mapsto \mu^t$
gives a bijection between the outside corners of $\lambda$ and of $\lambda^t$, 
giving rise to this counterpart for
Equation~\eqref{recurrence-divided-by-a}:
\begin{equation}
\label{recurrence-divided-by-b}
\frac{R^{(+1)}(\lambda^t)}{b}
=\sum_{k=0}^{d-1} (d-1-k) \cdot R( \lambda_{(x)} ) \cdot R( \lambda^{(x)} ). 
\end{equation}
The conjugation also gives a poset
isomorphism $[\varnothing,\lambda] \cong [\varnothing,\lambda^t]$
showing that $R(\lambda)=R(\lambda^t)$, 
and that these two posets 
have the same expectations $\EE(\Xdd)$,
meaning that $R^{(+1)}(\lambda)=R^{(+1)}(\lambda^t)$.
Therefore, adding Equations~\eqref{recurrence-divided-by-a} and~\eqref{recurrence-divided-by-b} gives 
$$
\frac{R^{(+1)}(\lambda)}{a}
+\frac{R^{(+1)}(\lambda)}{b}
=(d-1)\sum_{k=0}^{d-1} R( \lambda_{(x)} ) \cdot R( \lambda^{(x)} ) 
=(d-1)R(\lambda), 
$$
where the last equality used 
Proposition~\ref{q-Pascal-Catalan-generalization} specialized to $q=1$.
Hence one has
$$
\EE(\Xdd)=\frac{R^{(+1)}(\lambda)}{R(\lambda)}
=(d-1)\cdot \frac{1}{\frac{1}{a}+\frac{1}{b}}=\frac{(d-1)ab}{a+b}. 
$$
\end{proof}

Propositions~\ref{Y-expectation-for-Young} and \ref{X-expectation-for-Young}
imply that $[\varnothing, \lambda]$ is \CDE when $\lambda=\lambdanab{d}{a}{b}$ 
is a rectangular staircase.

\section{Coxeter groups and $0$-Hecke monoids}
\label{Coxeter-section}

Computing $\EE(\Xdd)$ and $\EE(\Ydd)$
for lower intervals in the weak Bruhat ordering
involves descent sets and factorizations
in Coxeter systems and $0$-Hecke monoids.  We begin this section by reviewing the relevant concepts from Coxeter groups,
then develop several formulas and results, and
finally do the same for $0$-Hecke monoids.

\subsection{Coxeter groups}

\begin{Definition}
A \emph{Coxeter matrix} is a finite set $S$ and a choice
of $m_{s,t}=m_{t,s}$ in $\{2,3,4,\ldots\} \cup \{\infty\}$ for $s \neq t$ in $S$.
The corresponding \emph{Coxeter system} $(W,S)$ is the
group $W$ generated by $S$, subject to relations $s^2=e$ for all
$s \in S$, and
$$
\underbrace{stst \cdots}_{m_{s,t}\text{ factors }} 
= \underbrace{tsts \cdots}_{m_{s,t}\text{ factors }} 
\text{ for }s \neq t\text{ in }S.
$$
\end{Definition}

\begin{Definition}
The \emph{length function} $\ell: W \rightarrow \NN$ for $(W,S)$ is defined by
$$
\ell(w):=
 \min\{\ell: w=s_1 s_2 \cdots s_\ell \text{ for }s_i \in S\}.
$$
A word 
$\sss:=(s_1,s_2,\ldots, s_{\ell(w)})$, for which 
$w=s_1 s_2 \cdots s_{\ell(w)}$, is a \emph{reduced word} for $w$. 
Denote by $\Red(w)$ the set of reduced words for $w$.
\end{Definition}

\begin{Definition}
The \emph{(right) descent set} of $w$ is
$$
\Des(w):=\{ s \in S: \ell(ws) < \ell(w)\}.
$$
The set of \emph{reflections} in $W$ is
$T:=\{wsw^{-1}: w \in W, s \in S\}.$ The  \emph{(left) inversion set} of $w$ in $W$ is  
$$
T_L(w):=\{t \in T: \ell(tw) < \ell(w)\},
$$
which has size $\ell(w)$, and is computable from any 
$\sss=(s_1,\ldots,s_{\ell(w)}) \in \Red(w)$ as follows \cite[Corollary~1.4.4]{BjornerBrenti05}: 
$$
T_L(w)=\left\{\begin{matrix} s_1, \\ 
        s_1s_2s_1, \\
        s_1s_2s_3s_2s_1, \\ 
         \vdots \\
        s_1 s_2 \cdots s_{\ell(w)} \cdots s_2s_1 
       \end{matrix} \right\}.
$$
\end{Definition}

\begin{Definition}\label{defn:weak bruhat order}
The \emph{(right) weak Bruhat order} $u \leq_R w$  on $W$ can be defined via 
any of the following equivalent conditions 
(see \cite[Chapter 3]{BjornerBrenti05}):
\begin{enumerate}\renewcommand{\labelenumi}{(\alph{enumi})}
\item
$u \leq_R w$ is the transitive closure of the
covering relation $u \lessdot_R w$, where $u=ws$ and $s \in \Des(w)$;
\item
there exist $u=u_0,u_1,\ldots,u_{\ell-1},u_\ell=w$ in $W$ and $s_i \in S$ such that $u_i s_i=u_{i+1}$ and $\ell(u_{i+1}) > \ell(u_i)$;
\item
there exists $(s_1,s_2,\ldots, s_{\ell(w)}) \in \Red(w)$ having a prefix
$(s_1,s_2,\ldots, s_{\ell(u)}) \in \Red(u)$; 
\item
$\ell(u) + \ell(u^{-1}w) = \ell(w)$; and
\item
$T_L(u) \subseteq T_L(w)$. 
\end{enumerate}
\end{Definition}

We note the following facts about weak Bruhat order intervals, which follow
trivially from Definition~\ref{defn:weak bruhat order}.

\begin{Proposition}
\label{trivial-weak-order-isomorphisms}
In a Coxeter system $(W,S)$,
\begin{enumerate}\renewcommand{\labelenumi}{(\alph{enumi})}
\item the weak Bruhat interval $[u,w]$ is isomorphic to the lower interval $[e,u^{-1}w]$,
via $v \mapsto u^{-1}v$, and
\item the dual poset $[e,w]^*$ to the lower interval $[e,w]$ is isomorphic to $[e,w^{-1}]$, via $u \mapsto w^{-1}u$.
\end{enumerate}
\end{Proposition}

The definition of $\leq_R$ allows us to reformulate Equations~\eqref{EX-is-edge-density} and~\eqref{EY-graded-reformulation}.

\begin{Corollary}
For any Coxeter system $(W,S)$ and $w \in W$,
\begin{align}
\label{weak-interval-X-expectation-reformulated}
\EE(\Xdd_{[e,w]}) 
  &=\frac{1}{\#[e,w]} \sum_{u \leq_R w} \#\Des(u) \text{ \ \ and} \\
\label{weak-interval-Y-expectation-reformulated}
\EE(\Ydd_{[e,w]}) 
  &=\frac{1}{ (\ell(w)+1) \cdot \#\Red(w) } 
           \sum_{\substack{\sss=(s_1,s_2,\ldots ,s_{\ell(w))}\\ \in \Red(w)}} \ \sum_{i=0}^{\ell(w)} 
               \#\Des(s_1 s_ 2 \cdots s_i).
\end{align}
\end{Corollary}

It is helpful to reformulate Equation~\eqref{weak-interval-X-expectation-reformulated}
for later use.

\begin{Proposition}
\label{weak-interval-X-expectation-complementary-formulation}
For any Coxeter system $(W,S)$ and any $w \in W$,
$$
\EE(\Xdd_{[e,w]}) 
  =\frac{1}{2}\left(\#S-
          \frac{1}{\#[e,w]} \sum_{u \leq_R w} \#\{s \in S: u \lessdot_R us \not\leq_R w\}
        \right).
$$
\end{Proposition}
\begin{proof}
Because posets $P$ and $P^*$ have the same expectation
for the variable $\Xdd$, one always has
$$
\EE(\Xdd_P)
 = \frac{ \EE(\Xdd_P)+\EE(\Xdd_{P^*}) }{2}
 =\frac{ \EE( \Xdd_P+\Xdd_{P^*} )}{2}.
$$
For each $p$, the statistic $\Xdd_P(p)$ reports the down-degree of $p$ in $P$. From the perspective of $P^*$, the statistic $\Xdd_{P^*}(p)$ computes the \emph{up-degree} of $p$ in $P$.

Specializing to $P=[e,w]$, the down-degree of an element $u \leq_R w$ is $\#\Des(u)=\#\{s \in S: us \lessdot_R u\}$ as before,
and its up-degree is $\#\{s \in S: u \lessdot us \leq_R w\}$.
Therefore the sum of an element's down- and up-degrees is
$$
\#\{ s \in S: us \lessdot_R u \text{ or }u \lessdot us \leq_R w\}
=\#S-\#\{ s\in S: u \lessdot us \not\leq_R w \}.
$$
Halving the expectation of this random variable on $[e,w]$ yields the formula in the proposition.
\end{proof}

\subsection{$0$-Hecke monoids}

We wish also to reformulate Equation~\eqref{weak-interval-Y-expectation-reformulated},
this time via words in the $0$-Hecke monoid.

\begin{Definition}
Given a Coxeter matrix $(m_{s,t})$ and Coxeter system $(W,S)$,
the associated \emph{$0$-Hecke monoid} $\HHH_W(0)$ is the monoid
generated by $\{T_s\}_{s \in S}$, subject to the quadratic relation $T_s^2=T_s$
for $s \in S$ and
$$
\underbrace{T_sT_tT_sT_t \cdots}_{m_{s,t}\text{ factors }} 
= \underbrace{T_tT_sT_tT_s \cdots}_{m_{s,t}\text{ factors }} 
\text{ for }s \neq t\text{ in }S.
$$
\end{Definition}

It turns out (see, for example, \cite{Norton79}) that
any choice of reduced word $\sss=(s_1,s_2,\ldots,s_{\ell(w)}) \in \Red(w)$ 
defines the same element 
$
T_w:=T_{s_1} \cdots T_{s_{\ell(w)}}
$ 
in the monoid $\HHH_W(0)$.  Furthermore,
$T_w=T_{w'}$ in $\HHH_W(0)$ if and only if $w=w'$ in $W$, and
thus, as a set, one has
$$
\HHH_W(0)=\{T_w: w \in W\}.
$$
This means that one can speak of a \emph{$0$-Hecke word} 
$(s_1,s_2,\ldots,s_L)$ for $w$, meaning that 
$T_w=T_{s_1} T_{s_2} \cdots T_{s_L}$ in $\HHH_W(0)$.
It also implies that one has the following relations in $\HHH_W(0)$:
\begin{equation}
\label{0-Hecke-regular-action}
T_w T_s =
\begin{cases}
T_{ws} & \text{ if } \ell(ws) > \ell(w), \text{ and}\\
T_w & \text{ if } \ell(ws) < \ell(w).
\end{cases}
\end{equation}

Factorizations $T_w=T_u T_v$ in $\HHH_W(0)$
give yet another characterization of the weak Bruhat order on $W$.

\begin{Proposition}
\label{left-factor-right-weak-corollary}
One has $u \leq_R w$ in $W$ if only if
there exists $v \in W$ with $T_u T_v = T_w$.
\end{Proposition}
\begin{proof}
If $u \leq_R w$ then $v=u^{-1}w$ will satisfy $T_u T_v=T_w$.
Conversely, if $T_u T_v= T_w$, 
then pick reduced words 
$$
\sss = (s_1,\ldots, s_{\ell(u)}) \in \Red(u) \text{ \ \ and \ \ } 
(s'_1,\ldots,s'_{\ell(v)}) \in \Red(v). 
$$
Form a subword $(s'_{i_1},\ldots, s'_{i_k})$ by omitting any $s'_i$ 
for which $\ell(us'_1 \cdots s'_{i-1}s'_i)< \ell(us'_1 \cdots s'_{i-1})$. 
Using the relations in Equation~\eqref{0-Hecke-regular-action}
and the factorization
$$
T_w=T_uT_v= T_u T_{s'_1} \cdots T_{s'_{\ell(v)}},
$$
one deduces that $T_w = T_u T_{s'_{i_1}} \cdots T_{s'_{i_k}}$. 
Moreover, we have $\ell(us'_{i_1} \cdots s'_{i_j})=\ell(u)+j$ for each $j$. 
Thus, $(s_1 ,\ldots, s_{\ell(u)},s'_{i_1}, \cdots, s'_{i_k})$ is a reduced word 
for $w$, and it contains $\sss \in \Red(u)$ as a prefix. 
Hence $u \leq_R w$.
\end{proof}

\begin{Definition}
If $T_w=T_{s_1} T_{s_2} \cdots T_{s_L}$ in $\HHH_W(0)$
with $L=\ell(w)+1$, call $(s_1,s_2,\ldots,s_L)$ 
a \emph{nearly reduced word} for $w$.
Denote by $\Red^{(+1)}(w)$ the set of
nearly reduced words for $w$.
\end{Definition}

\begin{Proposition}
For any Coxeter system $(W,S)$ and $w \in W$, one has a bijection
$$
\begin{array}{ccl}
\left\{
\begin{matrix}
((s_1,s_2, \ldots, s_{\ell(w)}),i,s) \in \Red(w) \times [0,\ell(w)] \times S:\\ 
s \in \Des(s_1 s_2 \cdots s_i) 
\end{matrix} 
\right\} 
&\longrightarrow&
\Red^{(+1)}(w) \\
 & & \\
\big((s_1,s_2,\ldots,s_{\ell(w)}),i,s\big) 
 &\longmapsto& 
  (s_1,s_2,\ldots, s_i,s,s_{i+1},\ldots, s_{\ell(w)}). 
\end{array}
$$
\end{Proposition}

\begin{proof}
The given map is well-defined because
$s \in \Des(s_1 s_2 \cdots s_i)$ implies that
$$
\begin{aligned}
T_{s_1} T_{s_2} \cdots
T_{s_i} \cdot T_s \cdot T_{s_{i+1}} \cdots T_{s_{\ell(w)}} &=T_{s_1} T_{s_2} \cdots
T_{s_i} \cdot T_{s_{i+1}} \cdots T_{s_{\ell(w)}}\\
&=T_w.
\end{aligned}
$$ 
It also surjects:  given a nearly reduced
word $\sss^+=(s_1^+,s_2^+,\ldots, s_{\ell(w)+1}^+)$ for $w$, 
if $i$ is the smallest index with 
$(s_1^+, s_2^+, \cdots, s^+_{i+1})$ not reduced, 
then $\sss:=(s^+_1, s^+_2, \cdots, s^+_i, s^+_{i+2}, s^+_{i+3}, \cdots, s_{\ell(w)+1}^+)$ 
has $(\sss,i,s^+_{i+1}) \longmapsto \sss^+$.
To show injectivity, assume that the elements $(\sss,i,s)$ and $(\ttt,j,t)$, with $i \leq j$,
have the same image under the map.  If one has strict inequality
$i < j$, then the two words must match up as follows.
$$
\begin{array}{rccccccccccl}
s_1  & s_2  & \ldots & s_i  & s       & s_{i+1} & \ldots & s_{j-1} &s_{j}&s_{j+1}    & \ldots & s_{\ell(w)} \\
\Vert & \Vert &    & \Vert & \Vert & \Vert &   & \Vert & \Vert & \Vert &  &\Vert \\
t_1 & t_2 & \ldots & t_i & t_{i+1}& t_{i+2}& \ldots & t_{j}  &t & t_{j+1} &\ldots & t_{\ell(w)}
\end{array}
$$
The word $\ttt$ is reduced,
so the same must be true of its prefix 
$(t_1 ,t_2, \cdots, t_i, t_{i+1})=(s_1, s_2,\ldots, s_i,s)$, 
which contradicts $s \in \Des(s_1 s_2 \cdots s_i)$.
Therefore $i=j$, which then implies that 
$(\sss,i,s)=(\ttt,j,t)$ .
\end{proof}

This result has two interesting corollaries, including a reformulation of Equation~\eqref{weak-interval-Y-expectation-reformulated}.

\begin{Corollary}
\label{Y-expectation-as-word-counts}
For any Coxeter system $(W,S)$ and $w \in W$,
$$
\EE(\Ydd_{[e,w]})=\frac{\# \Red^{(+1)}(w)}{(\ell(w)+1)\cdot \#\Red(w)}.
$$
\end{Corollary}

In turn, because reversing a word gives bijections
$
\Red(w) \leftrightarrow \Red(w^{-1}) 
$
and
$
\Red^{(+1)}(w) \leftrightarrow \Red^{(+1)}(w^{-1}),
$
Corollary~\ref{Y-expectation-as-word-counts} implies the following fact.

\begin{Corollary}
For any Coxeter system $(W,S)$ and $w \in W$, 
$$
\EE(\Ydd_{[e,w]})=\EE(\Ydd_{[e,w^{-1}]}) =\EE(\Ydd_{[e,w]^*}).
$$
\end{Corollary}

\section{Type $A$ and vexillary permutations}
\label{type-A-section}

There are special features to 
the weak Bruhat order intervals $[e,w]$ in the case of Coxeter systems
$(W,S)$ of type $A_{n-1}$.  
In this setting, $W$ is the \emph{symmetric group} $\symm_n$ of all permutations 
on $\{1,2,\ldots,n\}$, 
and $S=\{\sigma_1,\ldots,\sigma_{n-1}\}$ with each $\sigma_i$ being the adjacent transposition that swaps $i$ and $i+1$.  
In this section, we first explain how the theory of Schubert and
Grothendieck polynomials let one recast $\EE(\Ydd_{[e,w]})$ for vexillary permutations.
We then show how to recast  
$\EE(\Xdd_{[e,w]})$ for \emph{all} permutations. Finally, we specialize to
dominant permutations, and then specialize even further to dominant
permutations of rectangular staircase shape.

\subsection{Recasting $\EE(\Ydd_{[e,w]})$ 
for vexillary permutations}
\label{vexillary-section}

Recall that Section~\ref{section:intro} defined 
vexillary, dominant, Grassmannian, and inverse Grassmannian permutations.

\begin{Theorem}
\label{vexillary-theorem}
Fix a shape $\lambda$. For all vexillary permutation $w$ of shape $\lambda$,
$$
\EE(\Ydd_{[e,w]})
= \frac{\Red^{(+1)}(w)}{(\ell(w)+1)\Red(w)}
= \frac{f^{\lambda}(+1)}{(|\lambda|+1)f^\lambda}
= \EE(\Ydd_{[\varnothing,\lambda]}).
$$
Furthermore, for those permutations that are Grassmannian or inverse Grassmannian, 
$$
\EE(\Xdd_{[e,w]})
=\EE(\Xdd_{[\varnothing,\lambda]}).
$$
\end{Theorem}

\begin{Example}\label{ex:grassmannian example}
We illustrate Theorem~\ref{vexillary-theorem} 
for the partition $\lambda=(3,1,1)$.
\begin{enumerate}\renewcommand{\labelenumi}{(\alph{enumi})}
\item Figure~\ref{fig:grassmannian example} depicts the dual interval $[\varnothing,\lambda]^*$
and its isomorphic partner $[e,236145]$, where
$236145$ is Grassmannian with code $(1,1,3)$. 
\begin{figure}[htbp]
\begin{tikzpicture}[xscale=1.3]
\draw (0,5) -- (0,4) -- (2,2) -- (0,0) -- (-2,2) -- (0,4);
\draw (1,1) -- (-1,3);
\draw (-1,1) -- (1,3);
\foreach \x in {(-2,2),(0,2),(2,2),(-1,3),(1,3),(0,4),(0,5)}{\fill[white] \x++(-.25,-.25) rectangle ++(.5,.5);}
\foreach \x in {(0,0),(-1,1),(1,1)}{\fill[white] \x++(-.3,-.3) rectangle ++(.6,.6);}
\draw (0,5) node {$\varnothing$};
\draw (0,4) node {$\ttableau{{ }}$};
\draw (1,3) node {$\ttableau{{ }&{ }}$};
\draw (-1,3) node {$\ttableau{{ }\\ { }}$};
\draw (2,2) node {$\ttableau{{ }&{ }&{ }}$};
\draw (0,2) node {$\ttableau{{ }&{ }\\ { }}$};
\draw (-2,2) node {$\ttableau{{ }\\ { }\\ { }}$};
\draw (1,1) node {$\ttableau{{ }&{ }&{ }\\ { }}$};
\draw (-1,1) node {$\ttableau{{ }&{ }\\ { } \\ { }}$};
\draw (0,0) node {$\ttableau{{ }&{ }&{ }\\ { } \\ { }}$};
\end{tikzpicture}
\hspace{.5in}
\begin{tikzpicture}[xscale=1.3]
\draw[white] (0,-.4) -- (0,4);
\draw (0,5) -- (0,4) -- (2,2) -- (0,0) -- (-2,2) -- (0,4);
\draw (1,1) -- (-1,3);
\draw (-1,1) -- (1,3);
\foreach \x in {(0,0),(-1,1),(1,1),(-2,2),(0,2),(2,2),(-1,3),(1,3),(0,4),(0,5)}{\fill[white] \x++(-.25,-.25) rectangle ++(.5,.5);}
\draw (0,5) node {$\bm{236}145$};
\draw (0,4) node {$\bm{23}1\bm{6}45$};
\draw (1,3) node {$\bm{23}14\bm{6}5$};
\draw (-1,3) node {$\bm{2}1\bm{36}45$};
\draw (2,2) node {$\bm{23}145\bm{6}$};
\draw (0,2) node {$\bm{2}1\bm{3}4\bm{6}5$};
\draw (-2,2) node {$1\bm{236}45$};
\draw (1,1) node {$\bm{2}1\bm{3}45\bm{6}$};
\draw (-1,1) node {$1\bm{23}4\bm{6}5$};
\draw (0,0) node {$1\bm{23}45\bm{6}$};
\end{tikzpicture}
\caption{The interval $[\varnothing,(3,1,1)]^*$ and its isomorphic partner $[e,23614]$ of Example~\ref{ex:grassmannian example}.}\label{fig:grassmannian example}
\end{figure}
Both have $\EE(\Xdd)=13/10$ and $\EE(\Ydd)=23/18$, as predicted by the theorem.
These expectations, for both $\Xdd$ and $\Ydd$, would be shared by the
interval $[e,(236145)^{-1}] \cong [\varnothing,\lambda]$, because
$(236145)^{-1} = 412563$ is inverse Grassmannian.

\item On the other hand, Figure~\ref{fig:grassmannian example part 2} depicts the intervals $[e,4231]$
and $[e,25314]$. The permutation $4231$ is dominant with code $\lambda=(3,1,1)$,
while $25314$ is vexillary with code $(1,3,1)$ but is neither dominant nor Grassmannian nor inverse Grassmannian. 
\begin{figure}[htbp]
\begin{tikzpicture}[xscale=1.3]
\draw (0,0) -- (2,2) -- (2,3) -- (0,5) -- (-2,3) -- (-2,2) -- (0,0);
\draw (1,1) -- (0,2) -- (0,3) -- (1,4);
\draw (-1,1) -- (0,2) -- (0,3) -- (-1,4);
\foreach \x in {(0,0), (2,2), (2,3), (0,5), (-2,3), (-2,2),(1,1), (0,2), (0,3), (1,4),(-1,1),(-1,4)} {\fill[white] \x++(-.25,-.25) rectangle ++(.5,.5);}
\draw (0,0) node {$1234$};
\draw (-1,1) node {$2134$};
\draw (1,1) node {$1243$};
\draw (-2,2) node {$2314$};
\draw (0,2) node {$2143$};
\draw (2,2) node {$1423$};
\draw (-2,3) node {$2341$};
\draw (0,3) node {$2413$};
\draw (2,3) node {$4123$};
\draw (-1,4) node {$2431$};
\draw (1,4) node {$4213$};
\draw (0,5) node {$4231$};
\end{tikzpicture}
\hspace{.5in}
\begin{tikzpicture}[xscale=1.3]
\draw (0,0) -- (2,2) -- (1,3) -- (1,4) -- (0,5) -- (-1,4) -- (-1,3) -- (-2,2) -- (0,0);
\draw (1,1) -- (-1,3);
\draw (-1,1) -- (1,3);
\foreach \x in {(0,0), (2,2), (2,3), (0,5), (-2,3), (-2,2),(1,1), (0,2), (0,3), (1,4),(-1,1),(-1,4),(1,3),(-1,3)} {\fill[white] \x++(-.25,-.25) rectangle ++(.5,.5);}
\draw (0,0) node {$12345$};
\draw (-1,1) node {$21345$};
\draw (1,1) node {$12354$};
\draw (-2,2) node {$23145$};
\draw (0,2) node {$21354$};
\draw (2,2) node {$12534$};
\draw (-1,3) node {$23154$};
\draw (1,3) node {$21534$};
\draw (-1,4) node {$23514$};
\draw (1,4) node {$25134$};
\draw (0,5) node {$25314$};
\end{tikzpicture}
\caption{The intervals $[e,4231]$ and $[e,25314]$ of Example~\ref{ex:grassmannian example}.}\label{fig:grassmannian example part 2}
\end{figure}
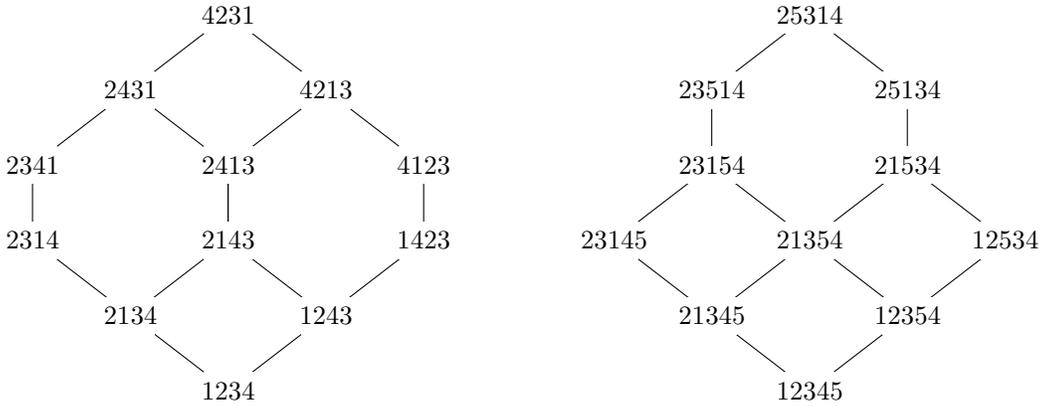
\end{enumerate}
These intervals all have $\EE(\Ydd)=23/18$, as predicted by the theorem  
because $236145$, $4231$, and $25314$ are all vexillary of shape $\lambda$.
However, $\EE(\Xdd_{[e,4231]})=5/4$, while
$\EE(\Xdd_{[e,25314]})=14/11$, neither of which
matches $\EE(\Xdd_{[e,236145]})=13/10$.
\end{Example}

The proof of Theorem~\ref{vexillary-theorem} uses the
relation between reduced words and 
$0$-Hecke words in type $A$ 
and the theory of Stanley symmetric functions and 
Lascoux and Sch\"utzenberger's 
theory of Schubert and Grothendieck polynomials (see, for example, 
\cite{Buch02, BKSTY, fomin.greene:noncommutative, LasSch2, Manivel01, Stanley84, Las03}).

\begin{Definition}
\label{Schub-Grothendieck-definitions}
Given a partition $\lambda$, the
\emph{Schur function} $s_\lambda$ and the \emph{stable Grothendieck polynomial} (for partitions)
$G_\lambda$ are computed by 
$$
\begin{aligned}
s_\lambda&=\sum_T \xxx^T \text{ \ \ and} \\
G_\lambda&=\sum_T (-1)^{|T|-|\lambda|} \xxx^T,
\end{aligned}
$$
(see, for example, \cite[Theorem 3.1]{Buch02}) where the first (respectively, second) sum runs over all column-strict tableaux (respectively, column-strict set-valued tableaux) $T$
of shape $\lambda$, 
and $\xxx^T$ is as defined in Equation~\eqref{eqn:tableau monomial}. Given $w\in \symm_n$, the
\emph{stable Schubert polynomial} (or \emph{Stanley symmetric function})
$F_w$ and the \emph{stable Grothendieck polynomial} (for permutations) $G_w$ are defined via
$$
\begin{aligned}
F_w&=\sum_{\substack{(\sigma_{a_1},\ldots,\sigma_{a_{\ell(w)}}),\\ (b_1,\ldots,b_{\ell(w)})}} 
                x_{b_1} \cdots x_{b_{\ell(w)}} \text{ \ \ and} \\
G_w&=\sum_{\substack{(\sigma_{a_1}, \cdots, \sigma_{a_L}),\\ (b_1,\ldots,b_{L})}} 
                (-1)^{L-\ell(w)} x_{b_1} \cdots x_{b_{\ell(w)}} \\
\end{aligned}
$$
(see, for example, \cite[Examples 2.2 and 2.5]{fomin.greene:noncommutative}).
In the first sum, $(\sigma_{a_1}, \cdots, \sigma_{a_{\ell(w)}})$ 
ranges over all reduced words for $w$, while
in the second sum, $(\sigma_{a_1}, \cdots, \sigma_{a_{L}})$ 
ranges over all $0$-Hecke words for $w$. In both cases, $(b_1,b_2,\ldots)$ are weakly increasing
sequences of positive integers satisfying the 
compatibility condition that $b_i < b_{i+1}$ whenever $a_i \leq a_{i+1}$.

Although it is not obvious, the functions $s_\lambda$, $G_\lambda$, $F_w$, and $G_w$ 
are all \emph{symmetric functions} in the infinite variable set $\{x_1,x_2,\ldots\}$.

Finally for any $w\in {\mathfrak S}_n$, the \emph{$(\beta)$-Grothendieck polynomial} 
is defined by
\begin{equation}
\label{beta-Grothendieck-definition}
{\mathfrak G}^{(\beta)}_w=\sum_{\substack{(\sigma_{a_1}, \cdots, \sigma_{a_L}),\\ (b_1,\ldots,b_{L})}} 
                \beta^{L-\ell(w)} x_{b_1} \cdots x_{b_{\ell(w)}}, 
\end{equation}
where the summation is over the same pairs of sequences
as for $G_w$, with the additional condition that $b_i\leq a_i$.
We also mention that the $\beta=0$ and $\beta=-1$ specializations
${\mathfrak G}_w^{(0)}$ and ${\mathfrak G}_w^{(-1)}$ are
called the \emph{Schubert polynomial} and \emph{Grothendieck polynomial}
for $w$, respectively.
\end{Definition}

The relevance of these polynomials comes from their
coefficients on certain squarefree monomials:
\begin{equation}
\label{squarefree-coefficients}
\begin{aligned}
f^\lambda&\text{ is the coefficient of }
   x_1 x_2 \cdots x_{|\lambda|}\text{ in }s_\lambda,\\
\#\Red(w)&\text{ is the coefficient of }
   x_1 x_2 \cdots x_{\ell(w)}\text{ in }F_w,\\
f^\lambda(+1)&\text{ is the coefficient of }
-x_1 x_2 \cdots x_{|\lambda|}x_{|\lambda|+1}\text{ in }G_\lambda,\\
\#\Red^{(+1)}(w)&\text{ is the coefficient of }
   -x_1 x_2 \cdots x_{\ell(w)}x_{\ell(w)+1}\text{ in }G_w.
\end{aligned}
\end{equation}

There are also various known relationships between them.

\begin{itemize}
\item
Note that
$s_{\lambda}$ and $F_w$ are the lowest-degree terms of $G_{\lambda}$
and $G_w$, respectively. 

\item $F_w$ and $G_w$ are
called \emph{stable} Schubert and Grothendieck polynomials because
\[
\begin{aligned}
F_w&=\lim_{N\to\infty} {\mathfrak G}_{1^N\times w}^{(0)}(x_1,\ldots,x_{N+n}) \text{ \ and}\\
G_w&=\lim_{N\to\infty} {\mathfrak G}_{1^N\times w}^{(-1)}(x_1,\ldots,x_{N+n}),
\end{aligned}
\]
where 
$
1^N\times w:=(1,2,\ldots,N,N+w(1),N+w(2),\ldots,N+w(n))
$
lies in $ {\mathfrak S}_{N+n}$.

\item
For  $w$ a Grassmannian permutation of shape $\lambda$, one has
\begin{equation}
\label{eqn:relate2}
\begin{aligned}
F_w&=s_{\lambda} \text{ \ and}\\
G_w&=G_{\lambda}.
\end{aligned}
\end{equation}
\end{itemize}

Our proof of Theorem~\ref{vexillary-theorem} will rest on
the following generalization of the relations in~\eqref{eqn:relate2}.

\begin{Lemma}
\label{vexillary-lemma}
For a vexillary permutation $w$ of shape $\lambda$, one has 
$$
\begin{aligned}
F_w&=s_\lambda,\\
G_w&=G_\lambda.
\end{aligned}
$$
\end{Lemma}

In order to prove this lemma, we will employ 
 a tableau formula for ${\mathfrak G}^{(-1)}_w$ from \cite{KMY}
that involves flagged set-valued tableaux.  First
recall the notion of a flag from Definition~\ref{defn:flags}.

Suppose $w$ is a vexillary permutation with shape $\lambda$ having $\ell$ nonzero parts.
One defines the flag $\varphi(w)=(\varphi_1,\varphi_2,\ldots,\varphi_{\ell})$ as follows (see \cite[\S 5.2]{KMY} for more details). 
Recall that the \emph{Rothe diagram} of $w$ is
\[
D(w):=\{(i,j): 1\leq i,j\leq n, w(i)>j, w^{-1}(j)>i\}
\quad \subset \quad 
\{1,2,\ldots,n\}\times \{1,2,\ldots,n\};
\]
(see, for example, \cite[\S 2.2.1]{Manivel01}).
Let $\mu(w)$ be the smallest Ferrers shape (northwest justified within the square shape $n^n$) that contains all the boxes of $D(w)$. Overlay the northwest corner of $\lambda(w)$ on the northwest
corner of the square $n^n$. Let the \emph{diagonal of row $i$} be the diagonal
occupied by the rightmost box of $\lambda(w)$ in row $i$. 
(In fact it is true that $\lambda(w)\subseteq \mu(w)$.)
Then
set $\varphi_i$ to be the row number of the southeastmost box of 
$\mu(w)$ in the diagonal of row $i$.

\begin{Example}
If $w$ is the vexillary permutation $14253 \in \symm_5$, 
then its Rothe diagram $D(w)$ is the set of row and column indices $(i,j)$
corresponding to the circles in
this picture:
$$
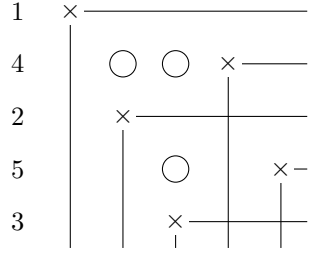
\begin{tikzpicture}[scale=.35]
\draw (0,0) -- (0,9) -- (9,9);
\draw (2,0) -- (2,5) -- (9,5);
\draw (4,0) -- (4,1) -- (9,1);
\draw (6,0) -- (6,7) -- (9,7);
\draw (8,0) -- (8,3) -- (9,3);
\foreach \x in {(0,9), (2,5), (4,1), (6,7), (8,3)} {\fill[white] \x +(-.5,-.5) rectangle +(.5,.5); \draw \x node {$\times$};}
\foreach \x in {(2,7), (4,7), (4,3)} {\draw \x circle (.5);}
\draw (-2,1) node {$3$};
\draw (-2,3) node {$5$};
\draw (-2,5) node {$2$};
\draw (-2,7) node {$4$};
\draw (-2,9) node {$1$};
\end{tikzpicture}$$
Then $\lambda(w)=(2,1)$ and
$\mu(w)=(3,3,3,3)$. Hence $\varphi(w)=(2,4)$.
\end{Example}

The following case is especially important to this paper.

\begin{Example}
If $w$ is a dominant permutation, then $\lambda(w)=\mu(w)$ and therefore 
$\varphi(w)=(1,2,3,\ldots)$.
\end{Example}

We now can state the following tableau formula,  found in \cite{KMY} (up to minor notational conventions).

\begin{Theorem}[{\cite[Theorem~5.8]{KMY}}]
\label{thm:theKMY}
Let $w$ be vexillary. Then
\begin{equation*}
{\mathfrak G}^{(-1)}_w(x_1,\ldots,x_n)=\sum_{T}(-1)^{|T|-|\lambda|}{\bf x}^T,
\end{equation*}
where the sum is over all set-valued tableaux of shape $\lambda(w)$
flagged by $\varphi(w)$.
\end{Theorem}

One checks that $w \mapsto \varphi(w)$ commutes as follows with
the operation $w \mapsto 1^N \times w$ on vexillary permutations:
\begin{equation}
\label{flag-stabilization}
\varphi(1^N \times w) = 
\varphi(w)+(N,N,\ldots)=(\varphi_1+N,\varphi_2+N,\ldots)\quad \text{ if }\varphi(w)=(\varphi_1,\varphi_2,\ldots).
\end{equation}

We can now complete the proof of Lemma~\ref{vexillary-lemma}, and 
then of Theorem~\ref{vexillary-theorem}.

\begin{proof}[Proof of Lemma~\ref{vexillary-lemma}]
The equality $F_w=s_\lambda$ for vexillary is well-known \cite[Corollary 4.2]{Stanley84},
but will also follow once we show $G_w=G_\lambda$, since
 $F_w$ and $s_\lambda$ are the lowest-degree terms in $G_w$ and $G_\lambda$,
respectively.

To this end, note that when working in finitely many variables $x_1,x_2,\ldots,x_N$
for any positive integer $N$, 
Definition~\ref{Schub-Grothendieck-definitions} implies that
\[
 G_w(x_1,\ldots,x_N)=  {\mathfrak G}^{(-1)}_{1^N\times w}(x_1,\ldots,x_N,0,0,0,\ldots).\]
On the other hand, by Theorem~\ref{thm:theKMY} and Equation~\eqref{flag-stabilization}, one has
\begin{equation}
\label{lemma:KMYconsequence}
{\mathfrak G}^{(-1)}_{1^N\times w}(x_1,\ldots,x_{N+n})
=\sum_{T} (-1)^{|T|-|\lambda|}{\bf x}^T,
\end{equation}
where the sum is over all column-strict set-valued tableaux of shape $\lambda$ 
flagged by $(\varphi_1+N,\varphi_2+N,\ldots)$.  Hence
\[{
\mathfrak G}^{(-1)}_{1^N\times w}(x_1,\ldots,x_N,0,0,0,\ldots)=\sum_{T}(-1)^{|T|-|\lambda|} {\bf x}^T,
\]
where the sum is over column-strict set-valued tableaux with entries from $1,2,\ldots,N$ (that is, the flagging condition on each row becomes redundant). Therefore, for any positive integer $N$,
\[
G_w(x_1,\ldots,x_N)=G_{\lambda}(x_1,\ldots,x_N).
\] 
Because $G_w$ and $G_\lambda$ are both symmetric functions 
this suffices to show $G_w=G_\lambda$.
\end{proof}

\begin{proof}[Proof of Theorem~\ref{vexillary-theorem}.] 
To prove the first assertion in the theorem,
note that Lemma~\ref{vexillary-lemma} together with 
Equation~\eqref{squarefree-coefficients}
show that when 
$w$ is vexillary of shape $\lambda$, one
has $\Red(w) = f^{\lambda}$ and
$\Red^{(+1)}(w) = f^{\lambda}(+1).$
Together with the fact that $\ell(w)=|\lambda|$, this gives the middle equality here
$$
\EE(Y_{[\varnothing,\lambda]})
= \frac{f^\lambda(+1)}{(|\lambda|+1)f^\lambda} 
= \frac{\#\Red^{(+1)}(w)}{(\ell(w)+1) \cdot \#\Red(w)}
= \EE(Y_{[e,w]}),
$$
while the first equality is 
Equation~\eqref{Young-Y-expectation-reformulated}
and the last equality is
Corollary~\ref{Y-expectation-as-word-counts}.

For the theorem's second assertion, 
use Equation~\eqref{Young-X-expectation-reformulated},
Proposition~\ref{trivial-weak-order-isomorphisms},
and Proposition~\ref{Grassmannian-poset-isomorphism} below.
\end{proof}

\begin{Remark} 
In fact, Lemma~\ref{vexillary-lemma} also shows that
any two $0$-Hecke words for a vexillary permutation are 
\emph{$K$-Knuth equivalent} in the sense defined by 
Buch and Samuel \cite[\S 5]{Buch.Samuel}.
We will not go into the details, but this can be deduced by
combining Lemma~\ref{vexillary-lemma}, along with
properties of the \emph{$K$-theoretic jeu de taquin} 
introduced by Thomas and Yong \cite{Thomas.Yong:Kjdt},
together with results on the \emph{Hecke insertion} introduced in \cite{BKSTY},
\cite[Theorem~4.2]{Thomas.Yong:longest}
and \cite[Theorem~6.2]{Buch.Samuel}.
\end{Remark}

\subsection{Evaluating $\EE(\Xdd_{[e,w]})$ via 
noninversion posets}

In type $A_{n-1}$, the set of reflections $T$ for the Coxeter system $(W,S)$ is equal to all (not necessarily adjacent) 
\emph{transpositions} $T=\{\tau_{ij}: 1 \leq i < j \leq n\}$,
and we have $S=\{\sigma_1,\ldots,\sigma_{n-1}\}$
where $\sigma_i:=\tau_{i,i+1}$.  Furthermore, the 
(left) inversion set of a permutation $w$ is 
$$
\#T_L(w):=\{ \tau_{ij}: 1 \leq i < j \leq n \text{ and }w^{-1}(i) > w^{-1}(j) \}
$$
and the number of inversions $\#T_L(w)$ is the same as the 
Coxeter group length $\ell(w)$.

\begin{Definition}
For $w \in \symm_n$, the \emph{noninversion poset} $P_{\noninv}(w)$ is
the partial order on $\{1,2,\ldots,n\}$ in which $i <_{P_{\noninv}(w)} j$
if and only if $i<_\ZZ j$ and $(i,j) \not\in T_L(w)$; that is, in which $i < j$ and $w^{-1}(i) < w^{-1}(j)$.
\end{Definition}

\begin{Definition}
For a poset $P$ on $\{1,2,\ldots,n\}$, a \emph{linear extension} of $P$ is a permutation
$w=w(1) \cdots w(n) \in \symm_n$ for which $i <_P j$ implies $w^{-1}(i) < w^{-1}(j)$; that is, $w(1) < w(2) < \cdots < w(n)$ extends $P$ to a linear order. Denote by $\LLL(P)$ the set of all linear extensions of $P$.
\end{Definition}

The following may then be viewed as the rephrasing in type $A$ of the 
characterization of the weak order that asserted $u \leq_R w$ if and only if $T_L(u) \subset T_L(w)$.

\begin{Proposition}
\label{noninversion-poset-proposition}
For any $w \in \symm_n$, one has
$[e,w]=\LLL(P_{\noninv}(w)).$
\end{Proposition}

This reformulation allows us to prove the following.

\begin{Proposition}
\label{Grassmannian-poset-isomorphism}
If $w$ is Grassmannian of shape $\lambda$, then $[e,w] \cong [\varnothing, \lambda]^*$.
\end{Proposition}
\begin{proof}
If $\lambda=(\lambda_1,\ldots,\lambda_\ell)$ with $\lambda_\ell >0$, 
then the one-line notation for $w$ is a concatenation of two increasing sequences;
namely, 
$$
w(1) = \lambda_\ell+1 < w(2) = \lambda_{\ell-1} + 2 < \cdots < w(\ell - 1) = \lambda_2 + \ell - 1 < w(\ell) = \lambda_1 + \ell
$$
concatenated with the sequence $w(\ell+1) < \cdots < w(n)$. Therefore the noninversion poset $P_{\noninv(w)}$ contains
these two sequences as chains, along with some extra order
relations between them.  Thus, any element $u$ in $[e,w]=\LLL(P_{\noninv(w)})$ is
a shuffle of these two increasing sequences,
and hence is completely determined by the positions 
$u^{-1}(w(1)) < \cdots < u^{-1}(w(\ell))$ occupied by the initial 
increasing sequence in the one-line notation for $w$.  This produces a poset 
isomorphism 
$[e,w] \rightarrow  [\varnothing, \lambda]^*$
defined by 
$$
u \longmapsto \mu=(\mu_1,\ldots,\mu_\ell)
                  =(u^{-1}(w(\ell))-\ell, \ldots, u^{-1}(w(1))-1).
$$
\end{proof}

Proposition~\ref{noninversion-poset-proposition}
lets us reinterpret the denominator $\#[e,w]$ of $\EE(\Xdd_{[e,w]})$.
We next work on the numerator.

\begin{Definition}
Given a covering relation $i \lessdot_P j$ in a poset $P$ on $\{1,2,\ldots,n\}$, define
a quotient poset $P/\{i,j\}$ that ``sets $i$ equal to $j$.''  More formally,
 consider the equivalence relation $\equiv_{ij}$ that
has $n-1$ blocks by merging $i$ and $j$ into a single block, and check that the (reflexive, symmetric) transitive closure of the union of the two
binary relations $\leq_P$  and $\equiv_{ij}$ gives a poset structure on the $n-1$ 
blocks of $\equiv$.
\end{Definition}

\begin{Proposition}
Fix a permutation $w \in \symm_n$ and set $P:=P_{\noninv(w)}$. Then
$$
\sum_{u \leq_R w} \#\{s \in S: u \lessdot_R us \not\leq_R w\}
=\sum_{ i \lessdot_P j }\#\LLL(P/\{i,j\}),
$$
and therefore
\begin{equation}
\label{poset-formulation-of-X-expectation}
\EE(\Xdd_{[e,w]})
=\frac{1}{2}\left( \#S - \sum_{ i \lessdot_P j }\frac{\#\LLL(P/\{i,j\})}{\#\LLL(P)} \right).
\end{equation}
\end{Proposition}
\begin{proof}
Given an element $u \leq_R w$ and $s=\sigma_k=(k,k+1)$ in $S$ for which
$u \lessdot_R us \not\leq_R w$, let $i:=u(k)$ and $j:=u(k+1)$. 
Then $u \lessdot_R us$ implies $i < j$.
Furthermore, $u \in \LLL(P)$ but $us \not\in \LLL(P)$ implies
that  $i <_{P} j$ must be a covering relation in $P$,
and one can regard $u/\{i,j\}$ as an element of $\LLL(P/\{i,j\})$.

Conversely, given a covering relation $i \lessdot_P j$ and
an element $\hat{u}$ of $\LLL(P/\{i,j\})$, say with 
 $\{i,j\}=\hat{u}_k$, one
can recover from it an element $u \leq_R w$ with $u \lessdot u\sigma_k \not\leq_R w$
by replacing the block $\hat{u}_k$ by $(u(k),u(k+1))=(i,j)$.
\end{proof}

\subsection{Dominant permutations and the forest hook-length formula}
\label{forest-section}

Arbitrary permutations $w \in \symm_n$ have no nice product
formula to compute $\#[e,w]=\#\LLL(P_\noninv(w))$, but dominant permutations do.

\begin{Definition}
A finite poset $P$ is a \emph{forest poset} if each element is covered by
at most one other element.
\end{Definition}

\begin{Proposition}[{\cite[Corollaries~5.3 and~ 5.4]{BFLR12}}]
The poset $P_{\noninv(w)}$ is a forest poset if and only if $w$ is dominant.
\end{Proposition}

Forest posets have the following
\emph{hook-length formula} counting their linear extensions, first observed by 
Knuth.

\begin{Proposition}[{\cite[\S 5.1.4 Exercise 20]{Knuth73}}]
Let $P$ be a forest poset, and set $P_{\leq i}:=\{j \in P: j \leq_P i\}$.
Then
\begin{equation}
\label{eqn:hooklength}
\#\LLL(P) = \frac{\#P !}{\prod\limits_{i \in P} \#P_{\leq i}}. 
\end{equation}
\end{Proposition}

In computing $\EE(\Xdd_{[e,w]})$ using
Equations~\eqref{poset-formulation-of-X-expectation} and~\eqref{eqn:hooklength},
 the following reduction for forests will be useful.

\begin{Lemma}
\label{lem:canceling hooks}
Fix a covering relation $i \lessdot_P j$ in a forest poset $P$.
Then 
$$
\frac{\#\LLL(P/\{i,j\})}{\#\LLL(P)} 
= \frac{\#P_{\leq i}}{\#P} \cdot \prod_{k >_P i} 
             \frac{\#P_{\leq k}}{\left(\#P_{\leq k} - 1\right)}
= \frac{\#P_{\leq i}}{\#P} \cdot \frac{\prod\limits_{k \in \beta(i)} \#P_{\leq k} }
                      {\prod\limits_{k \in \alpha(i)} \left(\#P_{\leq k} - 1\right)},
$$
where the sets $\alpha(i)$ and $\beta(i)$ are defined by
$$
\begin{aligned}
\alpha(i)&:=\{k \in P: k >_P i \text{ and either }k\gtrdot_P i,\text{ or }
              k\text{ covers more than one element}\} \text{ \ and}\\
\beta(i)&:=\{k \in P: k >_P i \text{ and }k
              \text{ is either maximal, or $k$ is covered by an element of }\alpha(i)\}.
\end{aligned}
$$
\end{Lemma}

Figure~\ref{fig:alpha and beta sets} shows a schematic for the local structure above a node $i$ in a forest poset $P$, with the nodes in $\alpha(i)$ circled and the nodes in $\beta(i)$ boxed.  Note that the sets $\alpha(i)$ and $\beta(i)$ may intersect.
\begin{figure}[htbp]
\begin{tikzpicture}[scale=.7]
\draw (0,.25) -- (0,9.5);
\draw (0,.75) node[right] {$i$};
\draw (.2,1.5) node[right] {};
\draw (.2,9.5) node[right] {};
\foreach \y in {.75,1.5,2.75,3.5,4.75,5.5,6.75,7.5,8.25,9.5} {\fill[black] (0,\y) circle (3pt);}
\foreach \y in {1.875,2.125,2.375,3.875,4.125,4.375,5.875,6.125,6.375,8.625,8.875,9.125} {\fill[black] (0,\y) circle (1.5pt);}
\foreach \y in {3.5,5.5,7.5,8.25} {\draw (0,\y) -- ++ (-.5,-.8); \draw (0,\y) --
 ++(.5,-.8);}
\foreach \y in {1.5,3.5,5.5,7.5,8.25} {\draw (0,\y) circle (5pt);}
\foreach \y in {2.75,4.75,6.75,7.5,9.5} {\draw (0,\y) ++(-.25,-.25) rectangle ++(.5,.5);}
\end{tikzpicture}
\caption{Local structure above a node $i$ in a forest posets $P$. Elements of $\alpha(i)$ are circled and elements of $\beta(i)$ are boxed.}\label{fig:alpha and beta sets}
\end{figure}

\begin{proof}
The first equality comes from Equation~\eqref{eqn:hooklength} via a calculation 
$$
\frac{\#\LLL(P/\{i,j\})}{\#\LLL(P)} 
=  \frac{ \displaystyle \left(\#P-1\right)! \prod_{k \in P} \#P_{\leq k} }
       {  \displaystyle \#P! \prod_{k \in P/\{i,j\}} \#(P/\{i,j\})_{\leq k}}
= \frac{\#P_{\leq i}}{\#P}
    \prod_{k >_P i} \frac{\#P_{\leq k}}{\left(\#P_{\leq k}-1\right) }, 
$$
because if we label elements of $P/\{i,j\}$ by $k \in P \setminus \{i\}$, then
 $\#(P/\{i,j\})_{\leq k}$ is either $\#P_{\leq k}$  for $k \not >_P i$,
or $\#P_{\leq k}-1$  for $k >_P i$.
The second equality comes from telescoping the factors in the rightmost product:
\begin{itemize}
\item for $k \not\in \alpha(i)$,  the denominator $\#P_{\leq k}-1$ cancels 
with $\#P_{\leq \ell}$ for the unique $\ell \lessdot k$, and
\item for $k \not\in \beta(i)$, 
the numerator $\#P_{\leq k}$ is canceled by 
$\#P_{\leq \ell}-1$ for the unique $\ell \gtrdot k$.
\end{itemize}
\end{proof}

\subsection{Computing $\EE(\Xdd_{[e,w]})$ for dominant permutations of rectangular staircase shape}
\label{clever-section}

We now turn to the computation of $\EE(\Xdd_{[e,w]})$
when $w$ is a dominant permutation of rectangular staircase shape $\lambdanab{d}{a}{b}$. 
As we will show, this has a very nice form.

Our strategy will approach this calculation via induction on $d$. 
To this end, throughout the remainder of this section, fix the rectangle dimensions $a,b \geq 1$, and assume, for convenience, that $a \leq b$.
For each $d \geq 2$, define $w^{(d)}$ to be the dominant permutation 
of shape $\lambdanab{d}{a}{b}$.    
One can check that $w^{(d)}$ lies in  
$\symm_N$ where $N:=a+(d-1)b$, and that the one-line notation for $w^{(d)}$ 
is the following concatenation of contiguous intervals of integers:
$$
w^{(d)} 
= I_{d-1} \, \cdots \, I_2 \, I_1 \, I_0 \, J_1 \, J_2 \, \cdots \, J_{d-1}, 
$$
where $I_m:=[mb+1, mb + a]$ and $J_m:=[(m-1)b + a +1,mb]$. 
The noninversion poset for this permutation,
$$P^{(d)}:=P_{\noninv}(w^{(d)}),$$ 
is a forest poset with the
schematic structure depicted in Figure~\ref{fig:forest structure}, where each $I_i$ and $J_j$ is totally ordered in increasing order. By convention, set $J_0:=\varnothing$ and define the poset
$P^{(1)}:=I_0=[1,a]$ totally ordered in increasing order.

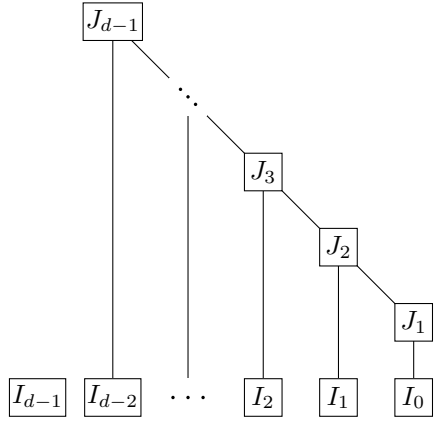
\begin{figure}[htbp]
\begin{tikzpicture}[scale=1]
\draw (1,0) -- (1,5);
\draw (2,0) -- (2,4);
\draw (3,0) -- (3,3);
\draw (4,0) -- (4,2);
\draw (5,0) -- (5,1);
\draw (1,5) -- (5,1);
\foreach \x in {(1,0),(2,0),(3,0),(4,0),(5,0),(1,5),(2,4),(3,3),(4,2),(5,1)} {\fill[white] \x++(-.25,-.25) rectangle ++(.5,.5);}
\node[draw,rectangle,minimum size=.5cm, inner sep=1pt] at  (0,0) {$I_{d-1}$}; 
\node[draw,rectangle,minimum size=.5cm, inner sep=1pt] at  (1,0) {$I_{d-2}$}; 
\foreach \x in {1.8,2,2.2} {\fill (\x,0) circle (.15ex);}
\node[draw,rectangle,minimum size=.5cm, inner sep=1pt] at  (3,0) {$I_2$};
\node[draw,rectangle,minimum size=.5cm, inner sep=1pt] at  (4,0) {$I_1$};
\node[draw,rectangle,minimum size=.5cm, inner sep=1pt] at  (5,0) {$I_0$};
\node[draw,rectangle,minimum size=.5cm, inner sep=1pt] at  (5,1) {$J_1$};
\node[draw,rectangle,minimum size=.5cm, inner sep=1pt] at  (4,2) {$J_2$};
\node[draw,rectangle,minimum size=.5cm, inner sep=1pt] at (3,3) {$J_3$};
\foreach \x in {(1.9,4.1),(2,4),(2.1,3.9)} {\fill \x circle (.15ex);}
\node[draw,rectangle,minimum size=.5cm, inner sep=1pt] at (1,5) {$J_{d-1}$}; 
\end{tikzpicture}
\caption{Structure of the noninversion poset for a dominant permutation of rectangular staircase shape $\lambdanab{d}{a}{b}$.}\label{fig:forest structure} 
\end{figure}

\begin{Example}\label{ex:chunky staircase ex}
Fix $a=3$ and $b=7$.  The posets for $P^{(d)}$ for $d=1,2,3,4$ 
are shown in Figure~\ref{fig:chunky staircase ex}.
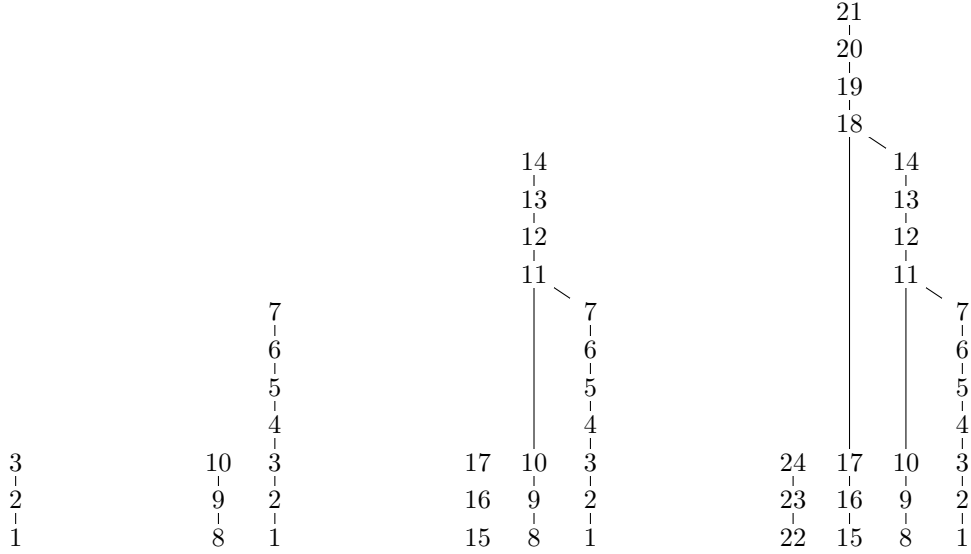
\begin{figure}[htbp]
\begin{tikzpicture}[xscale=.75,yscale=.5]
\draw (3,0) -- (3,2);
\foreach \y in {0,1,2,3,4,5,6} {\fill[white] (3,\y)++(-.35,-.35) rectangle ++(.7,.7);}
\draw (3,0) node {$1$};
\draw (3,1) node {$2$};
\draw (3,2) node {$3$};
\end{tikzpicture}
\hspace{.75in}
\begin{tikzpicture}[xscale=.75,yscale=.5]
\draw (2,0) -- (2,2);
\draw (3,0) -- (3,6);
\foreach \y in {0,1,2,7,8,9,10} {\fill[white] (2,\y)++(-.35,-.35) rectangle ++(.7,.7);}
\foreach \y in {0,1,2,3,4,5,6} {\fill[white] (3,\y)++(-.35,-.35) rectangle ++(.7,.7);}
\draw (3,0) node {$1$};
\draw (3,1) node {$2$};
\draw (3,2) node {$3$};
\draw (3,3) node {$4$};
\draw (3,4) node {$5$};
\draw (3,5) node {$6$};
\draw (3,6) node {$7$};
\draw (2,0) node {$8$};
\draw (2,1) node {$9$};
\draw (2,2) node {$10$};
\end{tikzpicture}
\hspace{.75in}
\begin{tikzpicture}[xscale=.75,yscale=.5]
\draw (2,0) -- (2,10);
\draw (3,0) -- (3,6);
\draw (3,6) -- (2,7);
\foreach \y in {0,1,2,11,12,13,14} {\fill[white] (1,\y)++(-.35,-.35) rectangle ++(.7,.7);}
\foreach \y in {0,1,2,7,8,9,10} {\fill[white] (2,\y)++(-.35,-.35) rectangle ++(.7,.7);}
\foreach \y in {0,1,2,3,4,5,6} {\fill[white] (3,\y)++(-.35,-.35) rectangle ++(.7,.7);}
\draw (3,0) node {$1$};
\draw (3,1) node {$2$};
\draw (3,2) node {$3$};
\draw (3,3) node {$4$};
\draw (3,4) node {$5$};
\draw (3,5) node {$6$};
\draw (3,6) node {$7$};
\draw (2,0) node {$8$};
\draw (2,1) node {$9$};
\draw (2,2) node {$10$};
\draw (2,7) node {$11$};
\draw (2,8) node {$12$};
\draw (2,9) node {$13$};
\draw (2,10) node {$14$};
\draw (1,0) node {$15$};
\draw (1,1) node {$16$};
\draw (1,2) node {$17$};
\end{tikzpicture}
\hspace{.75in}
\begin{tikzpicture}[xscale=.75,yscale=.5]
\draw (0,0) -- (0,2);
\draw (1,0) -- (1,14);
\draw (2,0) -- (2,10);
\draw (3,0) -- (3,6);
\draw (3,6) -- (2,7);
\draw (2,10) -- (1,11);
\foreach \y in {0,1,2} {\fill[white] (0,\y)++(-.35,-.35) rectangle ++(.7,.7);}
\foreach \y in {0,1,2,11,12,13,14} {\fill[white] (1,\y)++(-.35,-.35) rectangle ++(.7,.7);}
\foreach \y in {0,1,2,7,8,9,10} {\fill[white] (2,\y)++(-.35,-.35) rectangle ++(.7,.7);}
\foreach \y in {0,1,2,3,4,5,6} {\fill[white] (3,\y)++(-.35,-.35) rectangle ++(.7,.7);}
\draw (3,0) node {$1$};
\draw (3,1) node {$2$};
\draw (3,2) node {$3$};
\draw (3,3) node {$4$};
\draw (3,4) node {$5$};
\draw (3,5) node {$6$};
\draw (3,6) node {$7$};
\draw (2,0) node {$8$};
\draw (2,1) node {$9$};
\draw (2,2) node {$10$};
\draw (2,7) node {$11$};
\draw (2,8) node {$12$};
\draw (2,9) node {$13$};
\draw (2,10) node {$14$};
\draw (1,0) node {$15$};
\draw (1,1) node {$16$};
\draw (1,2) node {$17$};
\draw (1,11) node {$18$};
\draw (1,12)node {$19$};
\draw (1,13) node {$20$};
\draw (1,14) node {$21$};
\draw (0,0) node {$22$};
\draw (0,1) node {$23$};
\draw (0,2) node {$24$};
\end{tikzpicture}
\caption{The noninversion posets $P^{(1)}, P^{(2)}, P^{(3)}, P^{(4)}$ 
for $a = 3$, $b = 7$, as described in Example~\ref{ex:chunky staircase ex}.}\label{fig:chunky staircase ex}
\end{figure}
\end{Example}

Recall that $w^{(d)} \in \symm_N$, where $N=a+(d-1)b$. 
Thus we can rewrite Equation~\eqref{poset-formulation-of-X-expectation} using
Lemma~\ref{lem:canceling hooks} as
\begin{equation}
\label{EX-reformulation-via-thetas}
\EE\left(\Xdd_{[e,w^{(d)}]}\right) 
=\frac{1}{2}\left( \#S - \sum_{ i \lessdot j }
                         \frac{ \#\LLL(P^{(d)}/\{i,j\}) } { \#\LLL(P^{(d)}) } 
                  \right)
=\frac{1}{2}\left( N-1 - 
                    \frac{1}{N}\sum_{\ell=0}^{d-1} \theta_\ell^{(d)} \right), 
\end{equation}
where, for $\ell=0,1,2,\ldots,d-1$, we introduce the sums 
$$
\theta_\ell^{(d)}:= 
  \sum_{ \substack{i \in I_\ell \sqcup J_\ell\\\text{not maximal }\\\text{ in }P^{(d)}} }  
           f^{(d)}(i) 
\quad
\text{ \ with \ }
\quad
f^{(d)}(i) 
:= \#P^{(d)}_{\leq i}\cdot\frac{ \prod\limits_{k \in \beta(i)} \#P^{(d)}_{\leq k}}{ 
       \prod\limits_{k \in \alpha(i)} \left( \#P^{(d)}_{\leq k} - 1 \right)}. 
$$
Note that the sum for $\theta_0^{(d)}$ 
is why we have made the convention $J_0 = \varnothing$. 

For the sake of readability, we introduce the abbreviation
$$
c_j:=\frac{jb}{a+(j-1)b}.
$$

\begin{Lemma}
\label{clever-lemma}
The sums $\theta_\ell^{(d)}$ have these explicit formulas: 
$$
\theta_\ell^{(d)}= 
\begin{cases}
(a^2+\ell b(b-a)) \cdot c_{\ell+1} c_{\ell+2} \cdots c_{d-1} 
  & \text{ if }\ell=0,1,2,\ldots,d-2, \text{ and}\\ 
(a^2+(d-1)b(b-a)) - N  
  & \text{ if }\ell=d-1.\\ 
\end{cases}
$$
\end{Lemma}
\begin{proof}
For nonmaximal $i$ in $P^{(d)}$, 
the elements of $\alpha(i)$ are $\min J_j$ for
various $j$. Similarly, the elements of $\beta(i)$ are $\max J_j$ or $\max I_j$, for various $j$.  Now observe that
$$
\begin{aligned}
\#P^{(d)}_{\leq \max J_j}&=jb,\\ 
\#P^{(d)}_{\leq \max I_j}&=a, \text{ and}\\ 
\#P^{(d)}_{\leq \min J_j}&=a+(j-1)b + 1.\\ 
\end{aligned}
$$
From this, one can check that for nonmaximal $i$ in $P^{(d)}$,
$$
f^{(d)}(i)= 
c_{\ell+1} c_{\ell+2} \cdots c_{d-1} \cdot 
\left\{
\begin{matrix}
a 
  &\text{ if }i \in I_\ell,\\
\ell b
  &\text{ if }i \in J_\ell
\end{matrix}
\right\}.
$$
For $\ell=0,1,2,\ldots,d-2$, 
the intervals $I_\ell$ and $J_{\ell}$ contain $a$ and $b$ elements, respectively, and all are nonmaximal in $P^{(d)}$.  
On the other hand, all but one element from each of $I_{d-1}$ and $J_{d-1}$ are 
nonmaximal.  Therefore,
$$
\theta_\ell^{(d)}= 
\begin{cases}
c_{\ell+1} c_{\ell+2} \cdots c_{d-1} \cdot a \cdot a + 
 c_{\ell+1} c_{\ell+2} \cdots c_{d-1} \cdot \ell b(b-a)  
  &\text{ if }\ell=0,1,2,\ldots,d-2, \text{ and}\\ 
a\cdot (a-1) + (d-1) b \cdot (b-a-1)  
  &\text{ if } \ell=d-1. 
\end{cases}
$$
This agrees with the formulas given in the statement of the lemma.
\end{proof}

\begin{Corollary}
\label{EX-for-dominant-rectangular-staircase}
For $a,b \geq 1$ and $d \geq 2$, 
and for $w$ the dominant permutation of shape $\lambdanab{d}{a}{b}$, 
$$
\EE(\Xdd_{[e,w]}) 
= \frac{(d-1)ab}{a+b}. 
$$
\end{Corollary}
\begin{proof}
We may assume without loss of generality 
that $a \leq b$ by 
Proposition~\ref{trivial-weak-order-isomorphisms}(b),
and hence $w=w^{(d)}$. 
Set $N:=a+(d-1)b$.  
By Equation~\eqref{EX-reformulation-via-thetas}, it suffices to show that
\begin{equation}
\label{theta-sum-formula}
\frac{1}{N}\sum_{\ell=0}^{d-1} \theta_\ell^{(d)}  
 =  N-1-2(d-1)\frac{ab}{a+b} 
 =  a-1+(d-1)\frac{b(b-a)}{a+b} 
\end{equation}
for $d \ge 1$.  
We show the leftmost and rightmost sides of  
Equation~\eqref{theta-sum-formula} are equal via induction on $d$.   
In the base case $d=1$,  
$$
\frac{1}{N} \sum_{\ell=0}^{d-1} \theta_\ell^{(d)} 
=\frac{1}{a} \theta_0^{(1)}
=\frac{1}{a} (a(a-1)) 
=a-1
=a-1+(d-1)\frac{b(b-a)}{a+b}. 
$$
In the inductive step, we use
the following recursive reformulation of Lemma~\ref{clever-lemma}:
\begin{equation}
\label{recursive-clever-lemma}
\theta_\ell^{(d)}= 
\begin{cases}
  c_{d-1} \theta_\ell^{(d-1)} 
  & \text{ if }\ell=0,1,2,\ldots,d-3,\\ 
 c_{d-1} \left( \theta_\ell^{(d-1)} + a+(d-2)b \right) 
  & \text{ if }\ell=d-2, \text{ and}\\ 
(a^2+(d-1)b(b-a)) - N  
  & \text{ if }\ell=d-1. 
\end{cases}
\end{equation}
Now assume the left and right sides of 
Equation~\eqref{theta-sum-formula} are equal for $d-1$, and use  
Equation~\eqref{recursive-clever-lemma} to compute 
$$
\begin{aligned}
\frac{1}{N}
 \sum_{\ell=0}^{d-1} \theta_\ell^{(d)}  
  &=\frac{c_{d-1}}{N}\left(  
        \sum_{\ell=0}^{d-2} \theta_\ell^{(d-1)} 
        + a+(d-2)b \right) 
       +\frac{1}{N} \left( (a^2+(d-1)b(b-a)) - N \right) \\ 
  &=\frac{c_{d-1}}{N}\left(  
        (a+(d-2)b)\left( a-1+(d-2)\frac{b(b-a)}{a+b} \right) 
        + a+(d-2)b \right)\\  
  &\qquad     +\frac{ (a^2+(d-1)b(b-a)) }{N} - 1 \\ 
  &=a-1+(d-1)\frac{b(b-a)}{a+b}, 
\end{aligned}
$$
via straightforward algebra in the last step.
\end{proof}

Finally we can complete our goal.
\begin{proof}[Proof of Theorem~\ref{main-theorem}.]
Combine 
Propositions~\ref{Y-expectation-for-Young}, \ref{X-expectation-for-Young},
Theorem~\ref{vexillary-theorem}, and
Corollary~\ref{EX-for-dominant-rectangular-staircase}.
\end{proof}

\section{Macdonald and Fomin-Kirillov type formulas}
\label{FK-conjecture-section}

This section presents a conjecture, Conjecture~\ref{conj:mainone} below,
inspired both by Corollary~\ref{motivating-corollary}
and by an elegant formula of Fomin and Kirillov \cite{Fomin.Kirillov97} which we recall
now.
Let $w_0:=n(n-1) \cdots 21$ be the longest element
of $\symm_n$, which is the dominant permutation
of the staircase shape $\delta_n$,
having $\ell(w_0)=N:=\binom{n}{2}$. 

\begin{Theorem}[{\cite[Theorem~1.1]{Fomin.Kirillov97}}]
\label{Fomin-Kirillov-theorem}
$$
\sum_{(\sigma_{i_1}, \sigma_{i_2}, \ldots ,\sigma_{i_N})} 
(x+i_1)(x+i_2)\cdots(x+i_N)
=N! \prod_{1 \leq i < j \leq n} \frac{2x+i+j-1}{i+j-1},
$$
where the sum runs over all 
$(\sigma_{i_1}, \sigma_{i_2}, \ldots ,\sigma_{i_N})$ in $\Red(w_0)$.  
\end{Theorem}

Fomin and Kirillov consider arbitrary shapes as well, and this, too, will inspire the upcoming conjecture.

\begin{Theorem}[{\cite[Theorem~2.1]{Fomin.Kirillov97}}]
\label{Fomin-Kirillov-theorem for all shapes}
Let $\lambda$ be an arbitrary Ferrers shape, and $w$ the unique dominant permutation of shape $\lambda$, with length $\ell := \ell(w)$. Then
$$
\sum_{(\sigma_{i_1}, \sigma_{i_2}, \ldots ,\sigma_{i_{\ell}})} 
(x+i_1)(x+i_2)\cdots(x+i_{\ell})
=\ell! \cdot \#\{\text{plane partitions of shape $\lambda$ with parts $\le x$}\},$$
where the sum runs over all 
$(\sigma_{i_1}, \sigma_{i_2}, \ldots ,\sigma_{i_{\ell}})$ in $\Red(w)$.  
\end{Theorem}

Extracting the coefficient of $x^N$
in Theorem~\ref{Fomin-Kirillov-theorem}
gives Stanley's result \cite{Stanley84} that $\#\Red(w_0)=f^{\delta_n}$,
while setting $x=0$ 
recovers a result of Macdonald \cite[page 91]{Macdonald91}.

To state our conjecture, we define a sum generalizing the left 
side in Theorems~\ref{Fomin-Kirillov-theorem} and~\ref{Fomin-Kirillov-theorem for all shapes}.

\begin{Definition}
For a permutation $w$ and a nonnegative integer $L$,
define a polynomial in $x$ of degree $L$ by
$$
FK(w,L) :=
\sum_{(\sigma_{i_1}, \sigma_{i_2}, \cdots, \sigma_{i_L})}
(x+i_1)(x+i_2)\cdots(x+i_L),
$$
where the sum runs over all $0$-Hecke words 
$(\sigma_{i_1}, \sigma_{i_2}, \ldots, \sigma_{i_L})$ for $w$ of length $L$.
\end{Definition}

In particular, $FK\left(w_0, \binom{n}{2}\right)$ is the sum 
in Theorem~\ref{Fomin-Kirillov-theorem}. 

\begin{Conjecture}
\label{conj:mainone}
Let $w$ be the dominant permutation of rectangular
staircase shape $\lambda=\delta_d(b^a)$.
Then for $\ell:=\ell(w)=|\lambda|=\binom{d}{2}ab$, the
polynomial $FK(w,\ell)$ divides $FK(w,\ell+1)$ in $\QQ[x]$, with quotient

$$
\frac{FK(w,\ell+1)}{FK(w,\ell)}
= \binom{ \ell+1 }{2}\left( \frac{ 4x }{d(a+b)} + 1\right). 
$$
\end{Conjecture}

\begin{Example}
If $d=3,a=b=1$, 
then $\lambda=(2,1)$ with $\ell=|\lambda|=3$ and $w=321$.
Using the two reduced words and eight nearly reduced words for $w$ computed in 
Example~\ref{ex:long element in s3}, one finds that
$$
\begin{aligned}
FK(w,3)&=(x+1)(x+2)(x+1)+(x+2)(x+1)(x+2)\\
&=(x+1)(x+2)(2x+3),\\
& \\
FK(w,4)&=2(x+1)^3(x+2)+4(x+1)^2(x+2)^2+2(x+1)(x+2)^2\\
&=2(x+1)(x+2)(2x+3)^2,
\end{aligned}
$$
so that $FK(w,3)$ does divide $FK(w,4)$, with quotient
\[
\frac{FK(w,4)}{FK(w,3)}
=2(2x+3)=4x+6.
\] 
This agrees with Conjecture~\ref{conj:mainone}, which predicts
$$
\frac{FK(w,4)}{FK(w,3)}= \binom{4}{2}\left( \frac{4x}{3(1+1)}+1\right)=4x+6.
$$
\end{Example}

The following relation was one of our motivations for Conjecture~\ref{conj:mainone},
and provides some evidence for it.

\begin{Proposition}
Conjecture~\ref{conj:mainone} would imply 
Corollary~\ref{motivating-corollary}.
\end{Proposition}
\begin{proof}
For any permutation $w$ and any $L$,
the (leading) coefficient $c_L$ on $x^L$ in $FK(w,L)$
counts the number of $0$-Hecke words for $w$ of length $L$.
Therefore whenever $w$ is vexillary of shape $\lambda$ and $\ell=|\lambda|$, 
Lemma~\ref{vexillary-lemma} implies 
$c_{\ell}=f^{\lambda}$ and $c_{\ell+1}=f^{\lambda}(+1)$.

On the other hand, for $w$ dominant
of rectangular staircase shape $\lambda=\delta_d(b^a)$
and $\ell=|\lambda|$,
Conjecture~\ref{conj:mainone} would imply that 
$FK(w,\ell+1)/FK(w,\ell)$ is a linear polynomial $rx+s$
whose leading coefficient $r$ equals 
\[
\binom{\ell+1}{2}\frac{4}{d(a+b)}=\frac{c_{\ell+1}}{c_\ell}=\frac{f^{\lambda}(+1)}{f^{\lambda}}.
\]
This is equivalent to the assertion of Corollary~\ref{motivating-corollary}, using $\ell=\binom{d}{2}ab$.
\end{proof}

As further evidence in support of Conjecture~\ref{conj:mainone},
we will eventually verify it in the case $d=2$; that is for rectangular shapes $\lambda=b^a$.
In fact, it will turn out to be more convenient to work with an equivalent tableau version of the conjecture, whose statement requires some additional notation.

\begin{Definition}
For a  flag $\varphi=(\varphi_1,\varphi_2,\ldots)$, denote by ${\tt CST}(\lambda,\varphi, j)$ the collection of all
column-strict set-valued tableaux of shape $\lambda$ that are flagged by
$\varphi$, and whose total number of entries 
is $j$ (in other words, $j=\sum_{y} \# T(y)$ where $y$ runs through the cells of $\lambda$).
In particular, ${\tt CST}(\lambda,\varphi, j)$ is empty unless $j \ge |\lambda|$.
\end{Definition}

The equivalent tableau version of Conjecture~\ref{conj:mainone} is the following.

\vskip.1in
\noindent
{\bf Conjecture~\ref{conj:mainone}$^{\prime}$.}
\emph{
Let $w$ be the dominant permutation of rectangular
staircase shape $\lambda=\delta_d(b^a)$.
Then for $\ell:=\ell(w)=|\lambda|=\binom{d}{2}ab$, and for any positive integer $x$, the flag $\varphi=(1,2,3,\ldots)+(x,x,x,\ldots)$ produces
\begin{equation*}
\frac{\#{\tt CST}(\lambda,\varphi, \ell+1)}
{\#{\tt CST}(\lambda,\varphi,\ell)}=\frac{2\ell \cdot x}{d(a+b)}. 
\end{equation*}
}
\vskip.1in
 
The equivalence of Conjectures \ref{conj:mainone} and \ref{conj:mainone}$^{\prime}$ will follow 
from the next theorem, proven in Section~\ref{FKformula-proof-section},
combining ideas of \cite{Fomin.Stanley, Fomin.Kirillov, Fomin.Kirillov:YB, Fomin.Kirillov97} 
with Equation~\eqref{lemma:KMYconsequence}. 
To state it, recall that the \emph{Stirling number of the second
kind}  $S(L,j)$ counts partitions of $\{1,2,\ldots,L\}$ into $j$ blocks.

\begin{Theorem}
\label{thm:anFKformula}
For vexillary $w$ of shape $\lambda$ and positive integer $x$,
the flag $\varphi:=\varphi(w)+(x,x,x,\ldots)$ satisfies
\begin{equation}
\label{eqn:anFKformula}
FK(w,L)=\sum_{j=|\lambda|}^L \#{\tt CST}(\lambda,\varphi,j) \cdot j! S(L,j).
\end{equation}
\end{Theorem}

\begin{Remark}
Theorem~\ref{thm:anFKformula} is an extension of Theorem~\ref{Fomin-Kirillov-theorem} in the following sense.
If $w=w_0$ then $\lambda=\delta_n$. Setting $L=|\lambda|={n\choose 2}=N$ in Equation~\eqref{eqn:anFKformula} gives
\[FK(w_0,N)=\#{\tt CST}(\delta_n,(1+x,2+x,\ldots))\cdot N!.\]
Theorem~\ref{Fomin-Kirillov-theorem} is then derived from this in \cite{Fomin.Kirillov97} by an application of a formula
of Proctor \cite{Proctor:unpublished}, for plane partitions of staircase shape with largest bounded part, which implies
\[\#{\tt CST}(\delta_n,(1+x,2+x,\ldots)=\prod_{1 \leq i < j \leq n} \frac{2x+i+j-1}{i+j-1}.\]
In contrast, for arbitrary $j$, even when $w=w_0$, we know of no 
such product formulas for $\#{\tt CST}(\lambda,\varphi,j)$.
\end{Remark}

\begin{Remark}
Before proving that Conjectures~\ref{conj:mainone} and~\ref{conj:mainone}$^{\,\prime}$ are equivalent, we must observe that the former does not depend on $x$ being a positive integer, while the latter seems to make this assumption. In fact, the accurate relationship is that Conjecture~\ref{conj:mainone}$^{\,\prime}$ implies that $x \in \mathbb{Z}^+$: first note that replacing the flag $\varphi=(1,2,3,\ldots)+(x,x,x,\ldots)$ by the flag $\varphi'=(1,2,3,\ldots)+(\lfloor x \rfloor,\lfloor x \rfloor,\lfloor x \rfloor,\ldots)$ would not change the collection of permitted column-strict set-valued tableaux, so we may assume that $x \in \mathbb{Z}$; second, if $x < 0$, then $\varphi$ is not a valid flag, and if $x = 0$, then there are no barely set-valued tableaux filings flagged by $\varphi$; therefore, $x \in \mathbb{Z}^+$.
\end{Remark}

Let us assume the validity of Theorem~\ref{thm:anFKformula} for the moment, and check the following.

\begin{Corollary}
Conjectures~~\ref{conj:mainone} and ~\ref{conj:mainone}$^{\,\prime}$ are equivalent.
\end{Corollary}
\begin{proof}
Let $w$ be a dominant permutation of shape $\lambda$, and set $\ell:=|\lambda|=\ell(w)$.
Dominant permutations $w$ always have flag $\varphi(w)=(1,2,3,\ldots)$, which explains
the choice $\varphi:=(1,2,3,\ldots)+(x,x,x,\ldots)$.  Now applying
Theorem~\ref{thm:anFKformula} twice, with $L=\ell,$ and $L=\ell+1$, and using the facts
that $S(\ell,\ell)=1=S(\ell+1,\ell+1)$ and $S(\ell+1,\ell)=\binom{\ell+1}{2}$, one obtains
\[
\frac{FK(w,\ell+1)}{FK(w,\ell)}
 =\frac{\ell! \cdot {\ell+1\choose 2}\cdot \#{\tt CST}(\lambda,\varphi,\ell)
      \quad + \quad 
     (\ell+1)! \cdot \#{\tt CST}(\lambda,\varphi, \ell+1) }
  {\ell! \cdot \#{\tt CST}(\lambda,\varphi,\ell) }
 =\binom{\ell+1}{2} + (\ell+1) \cdot \rho 
\]
where $\rho:=\#{\tt CST}(\lambda,\varphi, \ell+1) /{ \#{\tt CST}(\lambda,\varphi, \ell) }$ is an abbreviation for the ratio that appears in Conjecture~\ref{conj:mainone}$^{\prime}$.
Thus, Conjecture~\ref{conj:mainone}  
holds if and only if
$$
\frac{FK(w,\ell+1)}{FK(w,\ell)}
=\binom{\ell+1}{2} + (\ell+1) \rho
=\binom{\ell+1}{2}\left( \frac{4x}{d(a+b)} +1 \right).
$$
Upon division by $\binom{\ell+1}{2}$, this assertion
is equivalent to
$$
1+\frac{2}{\ell} \rho =\frac{4x}{d(a+b)}+1,
 \quad \text{ and also }  \quad
\rho  = \frac{2\ell x}{d(a+b)},
$$
which is exactly Conjecture~\ref{conj:mainone}$^\prime$.
\end{proof}

\begin{Corollary}
Conjectures~\ref{conj:mainone} and ~\ref{conj:mainone}$^{\,\prime}$
hold when $d=2$; that is, for dominant $w$ of rectangle shape $\lambda=b^a$.
\end{Corollary}
\begin{proof}
It will be most convenient to work with Conjecture~\ref{conj:mainone}$^{\prime}$.
Our strategy will once again use the uncrowding map  in Definition~\ref{uncrowding-definition}
to convert barely set-valued tableaux to ordinary tableaux.

 For a partition $\lambda$, let
 $\CST{\lambda}{t}$ denote the set of column-strict (ordinary) tableaux of shape $\lambda$ with {all}
entries in the range $\{1,2,\ldots,t\}$. Note that column-strictness implies for rectangular shapes $b^a$ that
\[
{\tt CST}(b^a,\varphi, ab)=\CST{b^a}{x+a},
\]
where $\varphi=(1,2,3,\ldots,a)+(x,x,x,\ldots,x)$.

On the other hand, consider the  restriction of the uncrowding map to the domain
${\tt CST}((b^a),\varphi, ab+1)$. Since the rectangle $b^a$ has only outer corner cell in a row below $1$,
namely in row $a+1$,
the result always has shape $(b^a,1)$.  Thus, one obtains a bijection
$$
\begin{array}{rcl}
{\tt CST}(b^a,\varphi, ab+1) &\longrightarrow& \CST{(b^a,1)}{x+a} \times \{1,2,\ldots,a\},\\
T& \longmapsto &(T^+,i_0)
\end{array}
$$
which shows $\#{\tt CST}(b^a,\varphi, ab+1) = a \cdot \# \CST{(b^a,1)}{x+a}$, and hence
\begin{equation}
\label{ratio-of-CST-counts}
\frac{\#{\tt CST}(b^a,\varphi, \ell+1)}
      {\#{\tt CST}(b^a,\varphi,\ell)}
=\frac{ a \cdot \# \CST{(b^a,1)}{x+a} }
         {  \# \CST{b^a}{x+a} }.
\end{equation}
The numerator and denominator here are calculable via
the \emph{hook-content formula} (\cite[Theorem 15.3]{Stanley71})
$$
\#\CST{\lambda}{t}
=\prod_{y}\frac{t+c(y)}{h(y)}
$$
where the product runs over the cells $y=(i,j)$ of $\lambda$, with
$c(y)=j-i$ its \emph{content}, and 
$h(y)=\lambda_i +\lambda^{t}_{j}-(i+j)+1$
its \emph{hook-length}.
Here are the relevant values of $t+c(y)=(x+a)+c(y)$ and
$h(y)$ for cells $y$ of $(b^a,1)$:
\[
\begin{array}{|c|c|c|c|c|}
\hline
x+a & x+a+1 & x+a+2 & \cdots & x+a+b-1\\
\hline
\!\!x+a-1\!\! & x+a & x+a+1 & \cdots & x+a+b-2\\
\hline
\vdots & \vdots & \vdots &\ddots & \vdots\\
\hline
x+1 & x+2 & x+3 & \cdots & x+b-1\\
\hline
\boldsymbol{x}\\
\cline{1-1}
\end{array}
\ \ \ 
\begin{array}{|c|c|c|c|c|}
\hline
{a+b} & a+b-1 & a+b-2 & \cdots & a\\
\hline
\!\!a+b-1\!\! & \!\!a+b-2\!\! & \!\!a+b-3\!\! & \cdots & a-1\\
\hline
\vdots & \vdots & \vdots &\ddots & \vdots\\
\hline
{1+b} & b-1 & b-2 & \cdots & 1\\
\hline
{1}\\
\cline{1-1}
\end{array}
\]
On the other hand, for cells $y$ of $b^a$,
the relevant values of  $t+c(y)=(x+a)+c(y)$ are precisely the same
except that the boldfaced value $\boldsymbol{x}$ does not arise,
and the values of $h(y)$ are the same except for the first column, which are 
$a+b-1,a+b-2,\ldots,b+1,b$ in that setting. Therefore, 
\begin{equation*}
\frac{\#\CST{(b^a,1)}{x+a}}
{\#\CST{b^a}{x+a}}
=\frac{xb}{a+b}.
\end{equation*}
Comparing this with Equation~\eqref{ratio-of-CST-counts}
proves the $d=2$ case of Conjecture~\ref{conj:mainone}$^{\prime}$,
because then $\ell=|\lambda|=ab$.
\end{proof}

\section{Proof of Theorem~\ref{thm:anFKformula}.}
\label{FKformula-proof-section}

As mentioned earlier, our proof combines ideas of \cite{Fomin.Stanley, Fomin.Kirillov:YB, Fomin.Kirillov, Fomin.Kirillov97} with Equation~\eqref{lemma:KMYconsequence}.

Let $R$ denote the \emph{$0$-Hecke algebra} of type $A_{n-1}$.  That is, $R$ is the
monoid algebra for the $0$-Hecke monoid ${\mathcal H}_{W}(0)$, where 
$(W,S)$ is the Coxeter system $W={\mathfrak S}_n$ with
the adjacent transpositions $S=\{\sigma_1,\ldots,\sigma_{n-1}\}$ as
Coxeter generators.  Thus, $R$ has $\ZZ$-basis given by $\{T_w\}_{w \in \symm_n}$
and multiplication extended $\ZZ$-linearly from the $0$-Hecke monoid ${\mathcal H}_{W}(0)$.  Abbreviate here $T_i:=T_{\sigma_i}$
for $i=1,2,\ldots,n-1$, as was done in the Introduction.

We will also consider various rings obtained from $R$ by extension of scalars, such as
$R \otimes_\ZZ \QQ[[t]]$ or $R \otimes_\ZZ \QQ[[t_1,\ldots,t_n]]$.
Within these larger rings, we will also use without further mention a common exponential change-of-variables
in which $x:=e^t-1$ in $\QQ[[t]]$, and analogously, $x_i:=e^{t_i}-1$ in $\QQ[[t_1,\ldots,t_n]]$.

The following lemma is implicit
in \cite{Fomin.Kirillov} (cf.~\cite{Fomin.Stanley}), which explicitly states a consequence of it (see Remark~\ref{remark:FK} below).  Indeed, the result is known to the authors of \cite{Fomin.Kirillov}; see Lemma~5.6 of 
{\tt hep-th/9306005}. However, for convenience we include a proof.
Recall that for $w \in \symm_n$, the
$\beta$-Grothendieck polynomial ${\mathfrak G}^{(1)}={\mathfrak G}(x_1,\ldots,x_{n-1})$
defined in Equation~\eqref{beta-Grothendieck-definition}
is a polynomial in $n-1$ variables.

\begin{Lemma}[Fomin-Kirillov]
\label{lemma:exponential}
In the ring $R \otimes_\QQ \QQ[[t]]$,
\[
e^{t(T_{1}+2T_{2}+\cdots (n-1)T_{{n-1}})}
=\sum_{w\in {\mathfrak S}_n} {\mathfrak G}^{(1)}_w(e^{t}-1,\ldots, e^{t}-1)\cdot T_w
=\sum_{w\in {\mathfrak S}_n} {\mathfrak G}^{(1)}_w(x,\ldots,x) \cdot T_w.
\] 
\end{Lemma}
\begin{proof}
As in \cite{Fomin.Kirillov:YB}, define $h_i(t):=e^{tT_{i}}$.
Then following \cite{Fomin.Kirillov}, since $T_{i}^2=T_{i}$, one has
\begin{equation*}
h_i(t)=e^{tT_{i}}=\sum_{k=0}^{\infty} \frac{(tT_{i})^k}{k!}= 1+tT_{i}+t^2T_{i}/2!+\cdots
=1+(e^t-1)T_{i}=1+x T_{i}.
\end{equation*}
It is a main result of \cite{Fomin.Kirillov}~that, with notations
\begin{equation}
\label{Grothendieck-element-definition}
\begin{aligned}
A_i(t)&:=h_{n-1}(t)h_{n-2}(t)\cdots h_i(t) \text{ \ and}\\
{\mathfrak G}(t_1,t_2,\ldots,t_{n-1})
&:=A_1(t_1)A_2(t_2)\cdots A_{n-1}(t_{n-1}),
\end{aligned}
\end{equation}
the ``Grothendieck element'' ${\mathfrak G}(t_1,t_2,\ldots,t_{n-1})$ expands as follows 
in $R \otimes_\ZZ \QQ[[t_1,\ldots,t_{n-1}]]$:
\[
{\mathfrak G}(t_1,t_2,\ldots,t_{n-1})
=\sum_{w\in {\mathfrak S}_n} {\mathfrak G}^{(1)}_w(e^{t_1}-1,\ldots, e^{t_{n-1}}-1) \cdot T_w
=\sum_{w\in {\mathfrak S}_n} {\mathfrak G}^{(1)}_w(x_1,\ldots, x_{n-1}) \cdot T_w
\]
In fact, this is equivalent to the definition of ${\mathfrak G}_w^{(1)}$ from Equation~\eqref{beta-Grothendieck-definition}: working in the variables
$x_1,\ldots,x_{n-1}$, when one expands
${\mathfrak G}(t_1,t_2,\ldots,t_{n-1})$ as defined in Equation
\eqref{Grothendieck-element-definition}, its
coefficient of $T_w$ is exactly the sum 
over pairs $((\sigma_{a_1},\ldots,\sigma_{a_L}),(b_1,\ldots,b_L))$ on
the right side of Equation~\eqref{beta-Grothendieck-definition}. 

Thus, it remains to prove that specializing the variables $t_1=t_2=\cdots=t_{n-1}$ to
the single variable $t$ gives
\begin{equation}
\label{eqn:quickly}
{\mathfrak G}(t,t,\ldots,t)=A_1(t)A_2(t)\cdots A_{n-1}(t)
=e^{t(T_{1}+2T_{2}+\cdots +(n-1)T_{{n-1}})}.
\end{equation}
To this end, 
we employ a  \emph{mutatis mutandis}
modification of an argument of  \cite{Fomin.Stanley}. For brevity, we refer the reader to \cite{Fomin.Stanley}
for those details that remain unchanged.

It is easy to check that the collection $\{h_i\}$ satisfies the relations 
\begin{itemize}
\item[(I)] $h_i(s)h_j(t)=h_j(t)h_i(s)$ if $|i-j|\geq 2$,
\item[(II)] $h_i(s)h_i(t)=h_i(s+t)$, $h_i(0)=1$ (and therefore $h_i(s)h_i(-s)=1$),
\end{itemize} 
as well as the \emph{Yang-Baxter equation} \cite{Fomin.Kirillov:YB}
\begin{itemize}
\item[(III)] $h_i(s)h_{i+1}(s+t)h_i(t)=h_{i+1}(t)h_{i}(s+t)h_{i+1}(s)$.
\end{itemize}

The following lemma is the analogue of \cite[Lemma~2.1]{Fomin.Stanley}. Its proof is exactly the same as that result's, because it only depends on the relations (I)--(III).

\begin{Lemma}
\label{lemma:commute}
$A_i(s)$ and $A_i(t)$ commute.
\end{Lemma}

Define \cite[\S 4]{Fomin.Stanley}
\[{\widetilde A}_i(t)=h_i(t)h_{i+1}(t)\cdots h_{n-1}(t)\] 
and let
\[{\widetilde {\mathfrak G}}(t_1,\ldots, t_{n-1}):={\widetilde A}_{n-1}(t_{n-1}){\widetilde A}_{n-2}(t_{n-2})\cdots {\widetilde A}_1(t_1).\]

\begin{Lemma}
\label{lemma:4.1}
$A_i(s)$ and ${\widetilde A}_i(t)$ commute.
\end{Lemma}
\begin{proof}
The proof is the same as in \cite[Lemma~4.1]{Fomin.Stanley}, except we use
(II) and Lemma~\ref{lemma:commute} (where the original uses the
exact analogues, \cite[Lemmas~3.1(ii) and 2.1]{Fomin.Stanley}).
\end{proof}

\begin{Lemma}
\label{lemma:4.2}
${\widetilde A}_{n-1}(t_{n-1})\cdots {\widetilde A}_i(t_i)A_{i}(s)
=h_{n-1}(t_{n-1}+s)\cdots h_i(t_{i}+s){\widetilde A}_{n-1}(t_{n-2})
\cdots {\widetilde A}_{i+1}(t_i)$. 
\end{Lemma}
\begin{proof}
One follows \cite[Lemma~4.2]{Fomin.Stanley} except to
use Lemma~\ref{lemma:4.1} rather than their \cite[Lemma~4.1]{Fomin.Stanley}.
\end{proof}

\begin{Lemma}
\label{lemma:4.3}
\[{\widetilde {\mathfrak G}}(t_1,\ldots,t_{n-1}){\mathfrak G}(s_1,\ldots, s_{n-1})=\prod_{c=2-n}^{n-2} \ \ \prod_{i-j=c, i+j\leq n} h_{i+j-1}(s_i+t_j).\] 
Here the multiplication of the factors associated to $c=2-n,3-n,\ldots,n-2$ 
is done from left to right. The factors in the second product commute.
\end{Lemma}
\begin{proof}
As in \cite[Lemma~4.3]{Fomin.Stanley}, this follows from
repeated application of Lemma~\ref{lemma:4.2} combined with rearrangement of factors. To see that the factors in the second product commute, note that if
$(i,j),(i',j')\in {\mathbb N}\times {\mathbb N}$ satisfy $i-j=c=i'-j'$ and $i+j=i'+j'-1$, we would have
$2i-2i'=-1$, which is impossible.
\end{proof}

The following is the analogue of \cite[Lemma~5.1]{Fomin.Stanley}.

\begin{Lemma}
\label{lemma:5.1}
${\mathfrak G}(t,t,\ldots,t){\mathfrak G}(s,s,\ldots,s)={\mathfrak G}(t+s,\ldots, t+s)$.
\end{Lemma}
\begin{proof}
Same as that of \cite[Lemma~5.1]{Fomin.Stanley}, using
Lemma~\ref{lemma:4.3} above in place of their \cite[Lemma~4.3]{Fomin.Stanley}.
\end{proof}

We are now ready to prove Equation~\eqref{eqn:quickly} as in 
the proof of \cite[Lemma~2.3]{Fomin.Stanley}.  
Using Lemma~\ref{lemma:5.1}, and the fact that $G(0,0,\ldots,0)=1$, one finds that
\begin{align*}
\frac{d}{dt}{\mathfrak G}(t,t,\ldots,t) 
 &=\lim_{h \rightarrow 0}\frac{ {\mathfrak G}(t+h,t+h,\ldots,t+h) -{\mathfrak G}(t,t,\ldots,t) }{ h }\\
 &={\mathfrak G}(t,t,\ldots,t) \cdot 
         \lim_{h \rightarrow 0}\frac{ {\mathfrak G}(h,h,\ldots,h) - 1 }{ h } 
 = {\mathfrak G}(t,t,\ldots,t) \cdot F
\end{align*}
where 
$
F:=\left[ \frac{d}{dt}{\mathfrak G}(t,t,\ldots,t) \right]_{t=0},
$
an element of $R$. 
From this one concludes that, in $R \otimes_\ZZ \QQ[[t]]$, one has 
\begin{equation}
\label{G-as-an-exponential}
{\mathfrak G}(t,t,\ldots,t)=e^{tF}.
\end{equation} 
On the other hand, since
${\mathfrak G}(t,\ldots,t):=A_1(t)\cdots A_{n-1}(t)$, 
one can use the Leibniz rule repeatedly to compute
$$
\left[ \frac{d}{dt} h_i(t) \right]_{t=0}=T_{i} \qquad \text{and} \qquad
\left[ \frac{d}{dt} A_j(t) \right]_{t=0}=T_{{n-1}}+T_{{n_2}}+\cdots+T_{j}.
$$
Therefore
\begin{align*}
F=\left[ \frac{d}{dt} A_1(t)\cdots A_{n-1}(t)\right]_{t=0}
 &=\sum_{j=1}^{n-1} \left[ 
                                A_1(t) A_2(t) \cdots A_{j-1}(t)
                                 \cdot \frac{d}{dt} A_j(t) \cdot A_{j+1}(t) \cdots A_{n-1}(t)
                               \right]_{t=0} \\
&=\sum_{j=1}^{n-1} \left( T_{{n-1}}+T_{{n_2}}+\cdots+T_{j} \right)
=T_{1} + 2T_{2} + \cdots + (n-1) T_{{n-1}}.
\end{align*}
Plugging this into Equation~\eqref{G-as-an-exponential} proves
Equation~\eqref{eqn:quickly}, and completes the proof of Lemma~\ref{lemma:exponential}.
\end{proof}

\begin{proof}[Proof of Theorem~\ref{thm:anFKformula}]
By inspection, one has
\[
e^{t(T_{1}+2T_{2}+\cdots (n-1)T_{{n-1}})}=\sum_{w\in {\mathfrak S}_n} \left(\sum_{L} \frac{t^L}{L!} \sum_{(\sigma_{a_1},\ldots, \sigma_{a_L})}
a_1\cdots a_L \right) \cdot T_w
,\]
where the innermost sum is over $0$-Hecke words for $w$ of length $L$.
In view of Lemma~\ref{lemma:exponential}, this means that
\begin{equation}
\label{eqn:wesee}
{\mathfrak G}^{(1)}_w(e^{t}-1,\ldots,e^t-1)=
\sum_{L} \frac{t^L}{L!} \sum_{(\sigma_{a_1},\ldots, \sigma_{a_L})}
a_1\cdots a_L .
\end{equation}
For positive integers $x$, note that
 $(\sigma_{a_1},\ldots,\sigma_{a_L})$ is a $0$-Hecke word for $w$ if and only if 
$(\sigma_{x+a_1},\ldots, \sigma_{x+a_L})$ is a $0$-Hecke word for $1^x\times w$.
Therefore, one similarly has
\begin{equation*}
{\mathfrak G}^{(1)}_{1^x \times w}(e^{t}-1,\ldots,e^t-1)=
\sum_{L}  \frac{t^L}{L!} \sum_{(\sigma_{a_1},\ldots, \sigma_{a_L})}
(x+a_1)\cdots (x+a_L).
\end{equation*}
Equivalently, using $\left[\frac{t^{L}}{L!}\right] f(t)$ to denote the coefficient of $t^L/L!$ in $f(t)$, one has for any $w \in \symm_n$, 
\begin{equation}
\label{general-FK-expression}
FK(w,L) =\left[\frac{t^{L}}{L!}\right]{\mathfrak G}^{(1)}_{1^x\times w}(e^t-1,\ldots,e^t-1).
\end{equation}

For $w$ vexillary of shape $\lambda$, and $\varphi=\varphi(w)+(N,N,N,\ldots)$, 
replacing $x_i$ by $-x_i$ in Equation~\eqref{lemma:KMYconsequence} gives
$$
{\mathfrak G}^{(1)}_{1^N\times w}(x_1,\ldots,x_{N+n})
=\sum_{T} \xxx^T
$$
where the sum runs over all column-strict set-valued tableaux $T$ of shape $\lambda$ which
are flagged by $\varphi$.  Substituting $N=x$ and $x_i=e^t-1$ for all $i$, this shows that
$$
{\mathfrak G}^{(1)}_{1^x\times w}(e^t-1,\ldots,e^t-1)
=\sum_{j= |\lambda|}^{L} (e^t-1)^j \cdot \#{\tt CST}(\lambda,\varphi,j)
$$
and hence Equation~\eqref{general-FK-expression} becomes
\begin{align*}
FK(w,L) &= 
\sum_{j= |\lambda|}^{L} \left[\frac{t^L}{L!}\right](e^t-1)^j \cdot  \#{\tt CST}(\lambda,\varphi,j).
\end{align*}
Lastly, the well-known exponential generating function \cite[Equation~(1.94b)]{StanleyEC1} for Stirling numbers
$$
\sum_{L \geq j} S(L,j) \frac{t^L}{L!} = \frac{(e^t-1)^j}{j!}
$$
shows that $\left[\frac{t^L}{L!}\right](e^t-1)^j=j! S(L,j)$, completing the proof.
\end{proof}

\begin{Remark}
\label{remark:FK}
Since ${\mathfrak G}^{(1)}_{w_0}=x_1^{n-1}x_2^{n-2}\cdots x_{n-1}$, the case $w=w_0$ of 
Equation~\eqref{eqn:wesee} gives
\[
\sum_{L}\sum_{(a_1,\ldots, a_L)}
a_1\cdots a_L \frac{t^L}{L!}=(e^t-1)^{n\choose 2}.
\]
This formula was stated in \cite[\S 3]{Fomin.Kirillov}.
\end{Remark}

\section*{Acknowledgements}
The authors thank Zach Hamaker for pointing out an earlier argument for
Lemma~\ref{vexillary-lemma} in terms of Lascoux's transition formula for 
Grothendieck polynomials \cite{LascouxTransitions}. We thank 
K\'{a}ri Ragnarsson and Cara Monical for
providing computer code useful in the investigations of Sections~2
and~6, respectively.  We thank Susanna Fishel and Thomas McConville
for pointing out Proposition~\ref{Fishel-McConville-proposition}.  
We thank Thomas McConville for pointing out the examples at the end of Section~\ref{Tamari-subsection}, and thank
Alex Garver for help on oriented exchange graphs. 
We thank Sam Hopkins for helpful discussions on the work in \cite{Chan2},
and for sharing with us his preliminary results on shifted Young diagrams.
We thank Travis Scrimshaw for pointing out the connection between uncrowding and 
\cite[Theorem~6.11]{Buch02} (specifically the proof sketch).
We also thank Bruce Sagan 
for several helpful comments and questions. Finally, we thank the anonymous referees for their suggestions.
VR was partially supported 
by NSF RTG grant DMS-1148634. BT was partially 
supported by a Simons Foundation Collaboration Grant for Mathematicians. AY was partially supported by NSF grant
DMS-1500691.

\end{document}